\documentclass[10pt]{amsart}
\usepackage{amsmath,amssymb,amscd,amsfonts}
\usepackage{upgreek, mathabx}
\usepackage{hyperref}
\usepackage{stmaryrd}
\usepackage{tikz}
\usepackage[all,cmtip]{xy}

\usepackage{graphicx, color}
\usepackage{mathtools}

\numberwithin{equation}{section} 

\theoremstyle{plain}
\newtheorem{theoalph}{Theorem}

\newtheorem{thm}{Theorem}
\newtheorem{theorem}{Theorem}[section]
\newtheorem{proposition}[theorem]{Proposition}
\newtheorem{coroalph}[theoalph]{Corollary}

\newtheorem{lemma}[theorem]{Lemma}

\theoremstyle{remark}
\newtheorem{remark}[theorem]{Remark}

\newcommand{\A}{\mathbb{A}}

\newcommand{\C}{\mathbb{C}}

\newcommand{\F}{\mathbb{F}}
\newcommand{\K}{\mathbb{K}}
\newcommand{\N}{\mathbb{N}}
\newcommand{\Q}{\mathbb{Q}}
\newcommand{\R}{\mathbb{R}}

\newcommand{\Z}{\mathbb{Z}}

\renewcommand{\H}{\mathbb{H}}

\newcommand{\cA}{\mathcal{A}}

\newcommand{\cD}{\mathcal{D}}

\newcommand{\cF}{\mathcal{F}}

\newcommand{\cM}{\mathcal{M}}

\newcommand{\cO}{\mathcal{O}}

\newcommand{\cU}{\mathcal{U}}

\newcommand{\fD}{\mathfrak{D}}

\newcommand{\fj}{\mathfrak{j}}

\newcommand{\hN}{\widehat{N}}

\newcommand{\halpha}{\widehat{\alpha}}
\newcommand{\hbeta}{\widehat{\beta}}

\newcommand{\hpsi}{\widehat{\psi}}

\newcommand{\tE}{\widetilde{E}}

\newcommand{\tL}{\widetilde{L}}

\newcommand{\talpha}{\widetilde{\alpha}}

\newcommand{\tpsi}{\widetilde{\psi}}

\renewcommand{\=}{\coloneqq}
\newcommand{\dd}{\hspace{1pt}\operatorname{d}\hspace{-1pt}}

\newcommand{\wtf}{\widetilde{f}}

\newcommand{\wtg}{\widetilde{g}}
\newcommand{\wtm}{\widetilde{m}}

\newcommand{\whw}{\widehat{w}}
\newcommand{\whx}{\widehat{x}}

\newcommand{\bfB}{\mathbf{B}}
\newcommand{\bfD}{\mathbf{D}}

\newcommand{\bfO}{\mathbf{O}}

\DeclareMathOperator{\oD}{D}
\DeclareMathOperator{\oE}{E}

\DeclareMathOperator{\Aut}{Aut}
\DeclareMathOperator{\Berk}{Berk}
\DeclareMathOperator{\Div}{Div}
\DeclareMathOperator{\End}{End}
\DeclareMathOperator{\Ker}{Ker}
\DeclareMathOperator{\Frob}{Frob}

\DeclareMathOperator{\Hom}{Hom}
\DeclareMathOperator{\SL}{SL}
\DeclareMathOperator{\SO}{SO}
\DeclareMathOperator{\supp}{supp}
\DeclareMathOperator{\Tate}{Tate}

\newcommand{\CM}{CM}
\DeclareMathOperator{\hyp}{hyp}
\newcommand{\mhyp}{\mu_{\hyp}}
\newcommand{\odelta}{\overline{\delta}}
\newcommand{\supsn}[1]{| #1 |_{p\operatorname{-}\sups}}

\DeclareMathOperator{\bad}{bad}
\DeclareMathOperator{\ord}{ord}
\DeclareMathOperator{\sups}{sups}
\newcommand{\Ell}{Y}
\newcommand{\Bad}{Y_{\bad}(\C_p)}
\newcommand{\Ord}{Y_{\ord}(\C_p)}

\newcommand{\Sups}{Y_{\sups}(\C_p)}
\newcommand{\Supstwo}{Y_{\sups}(\C_2)}
\newcommand{\tSups}{Y_{\sups}(\Fpalg)}

\newcommand{\Fpalg}{\overline{\F}_p}
\newcommand{\Qpalg}{\overline{\Q_p}}
\DeclareMathOperator{\unr}{unr}
\newcommand{\Qpun}{\Q_p^{\unr}}
\newcommand{\Cpun}{\C_p^{\unr}}
\newcommand{\Cpunalg}{\overline{\Cpun}}
\newcommand{\Ordun}{Y_{\ord}(\Cpunalg)}
\newcommand{\Zp}{\Z_{(p)}}

\newcommand{\AKber}{\A^1_{\Berk}}
\DeclareMathOperator{\can}{can}
\newcommand{\xcan}{x_{\can}}
\newcommand{\bfone}{\mathbf{1}}
\newcommand{\Rzp}{\mathbb{R}_0^+}

\newcommand{\cht}{\widecheck{\t}}
\newcommand{\chalpha}{\widecheck{\alpha}}
\newcommand{\chbeta}{\widecheck{\beta}}
\newcommand{\chN}{\widecheck{N}}
\newcommand{\chk}{\widecheck{k}}
\renewcommand{\t}{\mathbf{t}}
\newcommand{\kval}{v_p}
\newcommand{\kproj}{\widehat{v}_p}
\newcommand{\kvaltwo}{v_2}

\renewcommand{\ss}{e}

\begin{document}

\title[Convergence of \CM{} points towards the Gauss point]{$p$-Adic distribution of \CM{} points and Hecke orbits.
  \\ \small{I. Convergence towards the Gauss point}} 

\author{Sebasti\'an Herrero}
\address{Instituto de Matem\'aticas, Pontificia Universidad Cat\'olica de Valpara\'iso, Blanco Viel 596, Cerro Bar\'on, Valpara\'iso,
Chile.}
\email{sebastian.herrero.m@gmail.com}

    \author{Ricardo Menares}

    \address{
    Facultad de Matem\'aticas, Pontificia Universidad Cat\'olica de Chile, Vicu\~na Mackenna 4860, Santiago, Chile.}

    \email{rmenares@mat.uc.cl}

\author{Juan Rivera-Letelier}
\address{Department of Mathematics, University of Rochester. Hylan Building, Rochester, NY~14627, U.S.A.}
\email{riveraletelier@gmail.com}
\urladdr{\url{http://rivera-letelier.org/}}

\begin{abstract}
  We study the asymptotic distribution of \CM{} points on the moduli space of elliptic curves over~$\C_p$, as the discriminant of the underlying endomorphism ring varies.
  In contrast with the complex case, we show that there is no uniform distribution.
  In this paper we characterize all the sequences of discriminants for which the corresponding \CM{} points converge towards the Gauss point of the Berkovich affine line.
  We also give an analogous characterization for Hecke orbits.
  In the companion paper we characterize all the remaining limit measures of \CM{} points and Hecke orbits.
\end{abstract}

\maketitle

\setcounter{tocdepth}{1}
\tableofcontents

\section{Introduction}\label{intro}

Given an algebraically closed field~$\K$, denote by~$\Ell(\K)$ the moduli space of elliptic curves over~$\K$.
It is the space of all isomorphism classes of elliptic curves over~$\K$, for isomorphisms defined over~$\K$.
For a class~$E$ in~$\Ell(\K)$, the $j$-invariant~$j(E)$ of~$E$ is an element of~$\K$ determining~$E$ completely.
The map $j\colon\Ell(\K) \to \K$ so defined is a bijection.
See for example~\cite{Sil09} and~\cite{Lan87} for background on elliptic curves.

If~$\K$ is of characteristic~$0$, then the endomorphism ring of an elliptic curve defined over~$\K$ is isomorphic to~$\Z$ or to an order in a quadratic imaginary extension of~$\Q$.
In the latter case, the order only depends on the class~$E$ in~$\Ell(\K)$ of the elliptic curve and~$E$ is said to have \emph{complex multiplication} or to be a \emph{\CM{} point}.
In this paper, the \emph{discriminant of a \CM{} point} is the discriminant of the corresponding order.\footnote{This notion of discriminant is not to be confused with the discriminant of a Weierstrass model of an elliptic curve~\cite[Chapter~III, Section~1]{Sil09}.}.
Moreover, a \emph{discriminant} is the discriminant of an order in a quadratic imaginary extension of~$\Q$.
An integer~$D$ is a discriminant if and only if~$D < 0$ and~$D \equiv 0, 1 \mod 4$.

For every discriminant~$D$, the set
\begin{equation}\label{e:CM set}
  \Lambda_D
  \=
  \{ E\in \Ell(\K) : \text{\CM{} point of discriminant } D \}
\end{equation}
is finite and nonempty.
So, we can define the probability measure~$\odelta_D$ on~$\Ell(\K)$, by
$$ \odelta_D
\=
\frac{1}{\# \Lambda_D} \sum_{E \in \Lambda_D} \delta_E, $$
where~$\delta_x$ denotes the Dirac measure on~$\Ell(\K)$ at~$x$.

Throughout the rest of this paper we fix a prime number~$p$ and a completion~$(\C_p, | \cdot |_p)$ of an algebraic closure of the field of $p$-adic numbers~$\Q_p$.
Our first goal is to study, for~$\K = \C_p$, the asymptotic distribution of~$\Lambda_D$ as the discriminant~$D$ tends to~$- \infty$.
This is motivated by the following result in the case where~$\K$ is the field of complex numbers~$\C$.
Recall that, if we consider the usual action of~$\SL_2(\Z)$ on the upper half-plane~$\H$ by M{\"o}bius transformations, then~$\Ell(\C)$ can be naturally identified with the quotient space~$\SL_2(\Z)\backslash \H$.
An appropriate multiple of the hyperbolic measure on~$\H$ descends to a probability measure~$\mhyp$ on~$\Ell(\C)$.

\begin{thm}
  \label{t:complex CM points}
  For every continuous and bounded function~$\varphi \colon \Ell(\C) \to \R$, we have
  \begin{displaymath}
    \frac{1}{\# \Lambda_D} \sum_{E \in \Lambda_D} \varphi(E)
    \to
    \int \varphi \dd \mhyp,
  \end{displaymath}
  as the discriminant~$D$ tends to~$- \infty$.
  Equivalently, we have the weak convergence of measures
  \begin{displaymath}
    \odelta_D
    \to
    \mhyp,
  \end{displaymath}
  as the discriminant~$D$ tends to~$- \infty$.
\end{thm}

The asymptotic distribution of \CM{} points on~$\Ell(\C)$ was part of a family of problems studied by Linnik, see~\cite{Lin68} and also~\cite{MicVen06}.
By applying a certain ``ergodic method'', Linnik proved the result above for sequences of discriminants satisfying some congruence restrictions.
In a breakthrough, Duke removed the congruence restrictions assumed by Linnik and proved Theorem~\ref{t:complex CM points} for fundamental discriminants~\cite{Duk88}.
Duke's proof uses the theory of non-holomorphic modular forms of half-integral weight and bounds for their Fourier coefficients, building on work of Iwaniec~\cite{Iwa87}.
Finally, Clozel and Ullmo obtained Theorem~\ref{t:complex CM points} for arbitrary discriminants, by studying the action of Hecke correspondences on \CM{} points and combining Duke's result together with the uniform distribution of Hecke orbits~\cite{CloUll04}.

\subsection{Convergence of \CM{} points  towards the Gauss point}
\label{ss:CM points}
Our first goal is to describe the asymptotic distribution of \CM{} points for the ground field~$\K = \C_p$.
However, it is easy to find sequences of discriminants~$(D_n)_{n = 1}^{\infty}$ for which the sequence of measures~$(\odelta_{D_n})_{n = 1}^{\infty}$ on~$\Ell(\C_p)$ has no accumulation measure.
A natural solution to this issue is to consider~$\Ell(\C_p)$ as a subspace of the Berkovich affine line~$\AKber$ over~$\C_p$, using the $j$-invariant to identify~$\Ell(\C_p)$ with the subspace~$\C_p$ of~$\AKber$.
In fact, every sequence of measures~$(\odelta_{D_n})_{n = 1}^{\infty}$ as above accumulates on at least one probability measure with respect to the weak topology on the space of Borel measures on~$\AKber$.
See Section~\ref{ss:berkovich} for a brief review of the space~$\AKber$ and the weak topology on the space of measures on~$\AKber$.

In contrast with Theorem~\ref{t:complex CM points}, for~$\K = \C_p$ the measures~$\odelta_D$ on~$\AKber$ do not converge to a limit as the discriminant~$D$ tends to~$- \infty$.
Our first main result is a characterization of all those sequences of discriminants~$(D_n)_{n = 1}^{\infty}$ tending to~$- \infty$, such that the sequence of measures~$(\odelta_{D_n})_{n = 1}^{\infty}$ in~$\AKber$ converges to the Dirac measure at the ``canonical'' or ``Gauss point'' $\xcan$ of~$\AKber$.
In the companion paper~\cite{HerMenRivII} we show that in all the remaining cases the sequence~$(\odelta_{D_n})_{n = 1}^{\infty}$ accumulates on at least one probability measure supported on a compact subset of the supersingular locus of~$\Ell(\C_p)$ and characterize all possible accumulation measures.

To state our first main result, we introduce some notation and terminology.
Identify the residue field of~$\C_p$ with an algebraic closure~$\Fpalg$ of the field with~$p$ elements~$\F_p$.
Recall that the endomorphism ring of an elliptic curve over~$\Fpalg$ is isomorphic to an order in either a quadratic imaginary extension of~$\Q$ or a quaternion algebra over~$\Q$.
In the former case the corresponding elliptic curve class is \emph{ordinary} and it is \emph{supersingular} in the latter.

Denote by~$\cO_p$ the ring of integers of~$\C_p$ and by~$\pi \colon \cO_p \to \Fpalg$ the reduction map.
An elliptic curve class~$E$ has \emph{good reduction}, if there is a representative elliptic curve defined over~$\cO_p$ whose reduction is smooth.
In this case the reduction is an elliptic curve defined over~$\Fpalg$, whose class~$\tE$ only depends on~$E$ and is the \emph{reduction of~$E$}.
Moreover, $E$ has \emph{ordinary} (resp. \emph{supersingular}) \emph{reduction} if~$\tE$ is ordinary (resp. supersingular).
An elliptic curve has good reduction precisely when~$j(E)$ is in~$\cO_p$ and when this is not the case~$E$ has \emph{bad reduction}.
The moduli space~$\Ell(\C_p)$ is thus partitioned into three pairwise disjoint sets: The \emph{bad}, \emph{ordinary} and \emph{supersingular reduction loci}, denoted by~$\Bad$, $\Ord$ and~$\Sups$, respectively.
Using~$j \colon \Ell(\C_p) \to \C_p$ to identify~$\Ell(\C_p)$ and~$\C_p$, we thus have the partition
\begin{displaymath}
  \cO_p
  =
  \Ord \sqcup \Sups.
\end{displaymath}
Moreover, if we denote by~$\tSups$ the finite subset of~$\Ell(\Fpalg)$ of supersingular classes, then~$\Sups = \pi^{-1}(\tSups)$ is a finite union of residue discs of~$\cO_p$.
Note that~$\Ord$ is a union of infinitely many residue discs of~$\cO_p$.

Every \CM{} point~$E$ has good reduction and the reduction type only depends on the discriminant~$D$ of~$E$, as follows.
\begin{enumerate}
\item[$(i)$]
  If~$p$ splits in~$\Q(\sqrt{D})$, then~$E$ has ordinary reduction.
\item[$(ii)$]
  If~$p$ ramifies or is inert in~$\Q(\sqrt{D})$, then~$E$ has supersingular reduction.
\end{enumerate}
See~\cite{Deu41} or~\cite[Chapter~13, Section~4, Theorem~12]{Lan87}.
We call a discriminant~$D$ \emph{$p$\nobreakdash-ordinary} in the first case and \emph{$p$-supersingular} in the second.
Moreover, we define
\begin{displaymath}
  \supsn{D}
  \=
  \begin{cases}
    0
    & \text{if~$D$ is $p$-ordinary};
    \\
    |D|_p
    & \text{if~$D$ is $p$-supersingular}.
  \end{cases}
\end{displaymath}

\begin{theoalph}
  \label{t:CM points}
  Let~$(D_n)_{n = 1}^{\infty}$ be a sequence of discriminants tending to~$- \infty$.
  Then we have the weak convergence of measures
  \begin{displaymath}
    \odelta_{D_n}\to \delta_{\xcan}
    \text{ as~$n \to \infty$}
  \end{displaymath}
  if and only if
  \begin{displaymath}
    \supsn{D_n} \to 0
    \text{ as $n \to \infty$}.
  \end{displaymath}
\end{theoalph}

For readers unfamiliar with the Berkovich affine line, we give a concrete formulation of the convergence of measures in Theorem~\ref{t:CM points} in terms of~$\C_p$ only, see Lemma~\ref{l:berkovich}$(ii)$ in Section~\ref{ss:berkovich}.

We obtain Theorem~\ref{t:CM points} as a direct consequence of quantitative estimates in the cases where all the discriminants in~$(D_n)_{n = 1}^{\infty}$ are $p$-ordinary (Theorem~\ref{t:ordinary CM points} in Section~\ref{ss:dynamics of t}) or $p$-supersingular (Theorem~\ref{t:supersingular CM points} in Section~\ref{s:supersingular CM points}).
Note that in the former case Theorem~\ref{t:CM points} asserts that~$\odelta_{D_n} \to \delta_{\xcan}$ weakly as~$n \to \infty$.   
The following stronger statement is a direct consequence of our quantitative estimate in this case.

\begin{coroalph}[Ordinary \CM{} points are isolated]
  \label{isolatedCMOrd}
  Every disc of radius strictly less than one contained in~$\Ord$ contains at most a finite number of \CM{} points.
  In particular, the set of \CM{} points in~$\Ord$ is discrete.
\end{coroalph}

  Corollary~\ref{isolatedCMOrd} seems to be well-known by the experts in the field, although we have not found this result explicitly stated in the literature.
  See Section~\ref{ss:notes} for comments and references.

\subsection{Convergence of Hecke orbits towards the Gauss point}
\label{ss:Hecke orbits}
To state our next main result, we first introduce Hecke correspondences.
See Section~\ref{ss:Hecke correspondences} for background.

Given an algebraically closed field~$\K$ of characteristic~$0$, a \emph{divisor on~$\Ell(\K)$} is an element of
$$\Div(\Ell(\K))
\=
\bigoplus_{E\in \Ell(\K)} \Z E, $$
the free abelian group spanned by the points of~$\Ell(\K)$.
The \emph{degree} and \emph{support} of a divisor $\fD = \sum_{E \in \Ell(\K)}n_EE$ in~$\Div(\Ell(\K))$ are defined by
\begin{displaymath}
  \deg(\fD)
  \=
  \sum_{E \in \Ell(\K)} n_E
  \text{ and }
  \supp(\fD)
  \=
  \{ E \in \Ell(\K) : n_E \neq 0 \},
\end{displaymath}
respectively.
If in addition~$\deg(\fD) \ge 1$ and for every~$E$ in~$\Ell(\K)$ we have~$n_E \ge 0$, then
$$\odelta_{\fD}
\=
\frac{1}{\deg(\fD)}\sum_{E \in \Ell(\K)} n_E \delta_E $$
is a probability measure on~$\Ell(\K)$.

For~$n$ in~$\N \= \{1,2,\ldots \}$ the $n$-th \emph{Hecke correspondence} is the linear map
\begin{displaymath}
  T_n \colon \Div(\Ell(\K)) \to \Div(\Ell(\K))
\end{displaymath}
defined for~$E$ in~$\Ell(\K)$, by
$$ T_n(E)
\=
\sum_{C\leq E \text{ of order } n}E/C, $$
where the sum runs over all subgroups~$C$ of~$E$ of order~$n$.
Note that~$\supp(T_n(E))$ is the set of all~$E'$ in~$\Ell(\K)$ for which there is an isogeny $E\to E'$ of degree~$n$.
Moreover,
$$ \deg(T_n(E))
=
\sum_{d|n,d>0}d
\ge
n, $$
so $\deg(T_n(E))\to \infty$ as $n \to \infty$.

In the case~$\K = \C_p$, it is easy to see that for each~$E$ in~$\Bad$ (resp.~$\Ord$, $\Sups$), we have that for every~$n$ in~$\N$ the divisor~$T_n(E)$ is supported on~$\Bad$ (resp.~$\Ord$, $\Sups$).
\begin{theoalph}
  \label{t:orbits}
  For every~$E$ in~$\Bad\cup \Ord$, we have the weak convergence of measures
  \begin{displaymath}
    \odelta_{T_n(E)}\to \delta_{\xcan} \text{ as } n \to \infty.
  \end{displaymath}
  Moreover, for~$E$ in~$\Sups$ and a sequence~$(n_j)_{j = 1}^{\infty}$ in~$\N$ tending to~$\infty$, we have the weak convergence of measures
  \begin{displaymath}
    \odelta_{T_{n_j}(E)} \to \delta_{\xcan} \text{ as } j \to \infty
  \end{displaymath}
  if and only if
  \begin{displaymath}
    | n_j |_p\to 0
    \text{ as }
    j \to \infty.
  \end{displaymath}
\end{theoalph}

When restricted to the case where~$E$ is in~$\Bad$, the above theorem is~\cite[\emph{Th{\'e}or{\`e}me}~1.2]{Ric18}.

To the best of our knowledge, Theorem~\ref{t:orbits} gives the first example where equidistribution of orbits fails for correspondences of degree bigger than one, see Section~\ref{ss:Hecke correspondences} for a description of Hecke correspondences as algebraic correspondences.
In the complex case, pluri-potential theory has been used successfully to prove equidistribution for correspondences satisfying a mild ``non-modularity'' condition, see for example~\cite{DinKauWu1808} and references therein.

The uniform distribution of Hecke orbits on~$\Ell(\C)$ is a well-known result from the spectral theory of automorphic forms, see~\cite[\emph{Th{\'e}or{\`e}me}~2.1]{CloUll04}, and also~\cite{CloOhUll01,EskOh06} for extensions and~\cite{LinSku64} for the related work of Linnik and Skubenko.

\begin{remark}
  In \cite{CloOhUll01,EskOh06}, the starting point is an algebraic group~$G$ over~$\Q$ and a congruence subgroup~$\Gamma$ of~$G(\Q)$, and the ambient space is $X = \Gamma\backslash G(\R)$.
  In this context, there is a natural notion of Hecke correspondences on~$X$.
  The aforementioned works establish the uniform distribution of every orbit of such Hecke correspondences under general hypotheses.
  In particular, the $\Q$-structure of~$G$ allows for $p$-adic variants of such results, see, \emph{e.g.}, \cite[Remark~(1) in p.~332]{CloOhUll01}.
  In the particular case $G = \SL_2$ and $\Gamma = \SL_2(\Z)$, there is a natural isomorphism~$\Ell(\C) \simeq \SL_2(\Z)\backslash \SL_2(\R)/ \SO_2(\R)$ and the natural projection from~$X$ to~$\Ell(\C)$ takes Hecke orbits as in~\cite{CloOhUll01,EskOh06} to Hecke orbits on~$\Ell(\C)$ as defined in this paper.
  The uniform distribution of Hecke orbits on~$\Ell(\C)$ is thus a special case of~\cite[Theorem~1.6]{CloOhUll01}, see also \cite[Theorem~1.2]{EskOh06}.
  However, this  strategy breaks down for Hecke orbits on~$\Ell(\C_p)$, because there is no analogous uniformization of~$\Ell(\C_p)$ as a double quotient.
  Moreover, Theorem~\ref{t:orbits} shows that there is no uniform distribution of Hecke orbits on~$\Ell(\C_p)$.
  Indeed, Theorem~\ref{t:orbits} and our results in the companion paper~\cite{HerMenRivII} show that, in contrast with~\cite{CloOhUll01,CloUll04,EskOh06}, the asymptotic distribution of~$(T_{n_j}(E))_{j = 1}^{\infty}$ on~$\Ell(\C_p)$ depends on both the starting point~$E$ and the sequence of integers~$(n_j)_{j = 1}^{\infty}$. 
\end{remark}

\subsection{Notes and references}
\label{ss:notes}
After the first version of this paper was written, we learned about the related work of Goren and Kassaei, appeared as~\cite{GorKas1711}.
For a prime number~$\ell$ different from~$p$, Goren and Kassaei study in~\cite{GorKas1711} the dynamics of the Hecke correspondence~$T_{\ell}$ acting on the moduli space of elliptic curves with a marked torsion point of exact order~$N$ coprime to~$p\ell$.
So, on one hand~\cite{GorKas1711} is more general than this paper in that it considers modular curves with level structure.
On the other hand, \cite{GorKas1711} is more restrictive in that it only considers the dynamics of a single Hecke correspondence of prime index different from~$p$, as opposed to the dynamics of the whole algebra of Hecke correspondences considered here.
Note also that we use~$\C_p$ as a ground field, which is natural to study equidistribution problems, whereas~\cite{GorKas1711} is restricted to algebraic extensions of~$\Q_p$.
In spite of the fact that both papers study the dynamics of similar maps, there is no significant intersection between the results of~\cite{GorKas1711} and those of this paper.
See also~\cite{HerMenRivII} for our additional results in the supersingular locus and the corresponding comparison with the results of~\cite{GorKas1711}.
Finally, our results on the dynamics of the canonical branch~$\t$ of~$T_p$ (defined on~$\Ord$ in Section~\ref{ss:canonical branch}) on ordinary \CM{} points show that this map gives rise to a ``$(p + 1)$-volcano'' in the sense of~\cite[Section~2.1]{GorKas1711}, see Remark~\ref{r:volcano}.

  Corollary~\ref{isolatedCMOrd} seems well-known among experts in the field, although we have not found this result explicitly stated in the literature.
  Even for higher-dimensional abelian varieties it can be deduced from the explicit characterization of the Serre--Tate local coordinates of \CM{} points as torsion points of the multiplicative group,  see, \emph{e.g.}, \cite[Proposition~3.5]{deJNoo91}.
  Our approach makes no use of these local coordinates, and is based on rigid analytic properties of the canonical branch~$\t$ of~$T_p$. For \CM{} elliptic curves with ordinary reduction, the connection between these two approaches is well-known, see, \emph{e.g.}, \cite[Section~7~d)]{Dwo69}.

Since every \CM{} point of~$\Ell(\C_p)$ is in the bounded set~$\cO_p$, Theorem~\ref{t:CM points} yields the following stronger statement: For every continuous function~$\varphi \colon \Ell(\C_p) \to \R$ and every sequence of discriminants~$(D_n)_{n = 1}^{\infty}$ tending to~$- \infty$ and satisfying $\supsn{D_n} \to 0$ as~$n \to \infty$, we have
\begin{displaymath}
  \frac{1}{\# \deg(\Lambda_{D_n})} \sum_{E \in \Lambda_{D_n}} \varphi(E)
  \to
  \int \varphi \dd \delta_{\xcan}
  \text{ as $n \to \infty$}.
\end{displaymath}

Although our formulation of Theorem~\ref{t:complex CM points} seems stronger than the one in~\cite[\emph{Th{\'e}or{\`e}me}~2.4]{CloUll04}, it is easy to see that it is equivalent, see for example~\cite[Lemma~2.2]{Bil97}.

\subsection{Strategy and organization}
\label{ss:strategy}
We now explain the strategy of the proof of Theorems~\ref{t:CM points} and~\ref{t:orbits} and simultaneously describe the organization of the paper.

After some preliminaries in Section~\ref{s:preliminaries}, we proceed to the proof of Theorem~\ref{t:CM points} in Sections~\ref{s:ordinary CM points} and~\ref{s:supersingular CM points}.
Theorem~\ref{t:CM points} is a direct consequence of stronger quantitative estimates in two separate cases: The case where all the discriminants in~$(D_n)_{n = 1}^{\infty}$ are $p$-ordinary and the case where they are all $p$\nobreakdash-supersingular.

The $p$-ordinary case is treated in Section~\ref{s:ordinary CM points}.
There are two main ingredients, both of which are related to the ``canonical branch~$\t$'' of~$T_p$ that is defined in terms of the ``canonical subgroup'' in Section~\ref{ss:canonical branch}, see also Appendix~\ref{s:canonical analyticity}.
The first main tool is a simple formula, for every integer~$m \ge 1$, of~$T_{p^m}$ on~$\Ord$ in terms of~$\t$ (Proposition~\ref{prop-tpm} in Section~\ref{ss:canonical branch}).
To establish this formula we use results of Tate and Deligne to show that~$\t$ is rigid analytic.
The second main tool is the interpretation of $p$-ordinary \CM{} points as preperiodic points of~$\t$ on~$\Ord$ (Theorem~\ref{t:ordinary CM points}$(i)$), which is based on Deuring's work on the canonical subgroup.
Our quantitative estimate in the $p$-ordinary case is stated as Theorem~\ref{t:ordinary CM points}$(ii)$ in Section~\ref{ss:dynamics of t} and its proof is given at the end of this section.

The $p$-supersingular case is technically more difficult.
We use Katz--Lubin's extension of the theory of canonical subgroups to ``not too supersingular'' elliptic curves and ``Katz' valuation''.
We recall these in Section~\ref{ss:katz valuation}, where we also give an explicit formula relating Katz' valuation to the $j$\nobreakdash-invariant (Proposition~\ref{vnorm}).
We use Katz' valuation to give a concrete description of the action of Hecke correspondences on the supersingular locus in terms of a sequence of correspondences~$(\uptau_m)_{m = 1}^{\infty}$ acting on the interval~$\left[0, \frac{p}{p + 1} \right]$ (Proposition~\ref{p:katz kite} in Section~\ref{ss:katz kite}).
To do this, we rely on results in~\cite[Section~3]{Kat73} and, for~$p = 2$ and~$3$, on certain congruences satisfied by certain Eisenstein series, see Proposition~\ref{prop-E1-E2} in Appendix~\ref{s:lift Hasse 2 3}.
Our quantitative estimate in the $p$-ordinary case is stated as Theorem~\ref{t:supersingular CM points} at the beginning of Section~\ref{s:supersingular CM points} and its proof is given at the end of this section.

  In Appendix~\ref{s:canonical analyticity} we formulate some of our results on the canonical branch~$\t$ of~$T_p$, as a lift of the classical Eichler--Shimura congruence relation (Theorem~\ref{t:canonical analyticity}).

The proof of Theorem~\ref{t:orbits} splits in three complementary cases, according to the reduction type of~$E$.
In each case we obtain a stronger quantitative estimate.
For the bad reduction case we use Tate's uniformization theory (Proposition~\ref{p:bad orbits} in Section~\ref{ss:bad orbits}).
Thanks to the multiplicative properties of Hecke correspondences~\eqref{eq-multiplicativity-Hecke-operators}, the ordinary reduction case (Proposition~\ref{p:ordinary orbits} in Section~\ref{ss:ordinary orbits}) is reduced to two special cases: The asymptotic distribution of~$(T_{p^m}(E))_{m = 1}^{\infty}$  (Proposition~\ref{distord}) and, for a sequence~$(n_j)_{j = 1}^{\infty}$ of integers in~$\N$ that are not divisible by~$p$, the asymptotic distribution of~$(T_{n_j}(E))_{j = 1}^{\infty}$ (Proposition~\ref{orderrante}).
The former case is obtained using the tools developed in Section~\ref{t:ordinary CM points} and the latter is reduced to the study of the action of Hecke correspondences on ordinary elliptic curves in~$\Ell(\Fpalg)$ and is elementary.
Finally, the supersingular case (Proposition~\ref{p:supersingular orbits} in Section~\ref{ss:supersingular orbits}) is obtained from the description of the action of Hecke correspondences on the supersingular locus in Section~\ref{ss:katz kite} and an explicit formula for the correspondences~$(\uptau_m)_{m = 1}^{\infty}$ (Lemma~\ref{l:katz hyper kite}).

\subsection*{Acknowledgments}
The second and third named authors thank Leon Takhtajan for references.
The second named author thanks Emmanuel Ullmo for sharing his interest in the questions we study here.
He also acknowledges Rodolphe Richard for explaining him the basic ideas leading to Propositions~\ref{p:bad orbits} and~\ref{orderrante}.
We thank the anonymous referees for their valuable comments that helped us improve the exposition.

During the preparation of this work the first named author was supported by the Chilean CONICYT grant 21130412 and the Royal Swedish Academy of Sciences.
The second named author was supported by FONDECYT grant 1171329.
The third named author acknowledges partial support from NSF grant DMS-1700291.
The authors would like to thank the Pontificia Universidad Cat{\'o}lica de Valpara{\'{\i}}so, the University of Rochester and Universitat de Barcelona for hospitality during the preparation of this work.

\section{Preliminaries}
\label{s:preliminaries}
Recall that~$\N = \{1, 2, \ldots \}$.
Given~$n$ in~$\N$, denote by
\begin{displaymath}
  d(n)
  \=
  \sum_{d > 0, d \mid n} 1
  \text{ and }
  \sigma_1(n)
  \=
  \sum_{d > 0, d \mid n} d
\end{displaymath}
the number and the sum of the positive divisors of~$n$, respectively.
We use several times the inequality
  \begin{equation}
    \label{eq-lowb-sigma1}
  \sigma_1(n)\geq n,
\end{equation}
and the fact that for every~$\varepsilon > 0$ we have
  \begin{equation}
    \label{e:2}
    d(n) = o(n^{\varepsilon}),
  \end{equation}
  see for example~\cite[p.~296]{Apo76}.

  For a set~$X$ and a subset~$A$ of~$X$, we use~$\bfone_A \colon X \to \{0, 1\}$ to denote the indicator function of~$A$.

  For a topological space~$X$, denote by~$\delta_x$ the \emph{Dirac mass on~$X$ supported at~$x$}.
It is the Borel probability measure characterized by the property that for every Borel subset~$Y$ of~$X$ we have~$\delta_x(Y) = 1$ if~$x \in Y$ and~$\delta_x(Y) = 0$ otherwise.

Normalize the norm~$| \cdot |_p$ of~$\C_p$ so that~$|p|_p = \frac{1}{p}$ and denote by~$\ord_p \colon \C_p \to \R \cup \{ + \infty \}$ the valuation defined by~$\ord_p(0) = + \infty$ and for~$z$ in~$\C_p^{\times}$ by~$\ord_p(z) = - \frac{\log |z|_p}{\log p}$.
Denote by~$\cM_p$ the maximal ideal of~$\cO_p$ and recall that we identify~$\cO_p / \cM_p$ with~$\Fpalg$ and that~$\pi \colon \cO_p \to \Fpalg$ denotes the reduction morphism.
For~$\zeta$ in~$\Fpalg$, denote by~$\bfD(\zeta) \= \pi^{-1}(\zeta)$ the residue disc corresponding to~$\zeta$.

  \subsection{Divisors}
\label{ss:divisors}
A \emph{divisor} on a set~$X$\footnote{We only use this definition in the case~$X$ is one of several types of one-dimensional objects. For such~$X$, the notion of divisor introduced here can be seen as a natural extension of the usual notion of Weil divisor.} is a formal finite sum~$\sum_{x\in X}n_xx$ in~$\bigoplus_{x \in X} \Z x$.
In the special case where for some~$x_0$ in~$X$ we have~$n_{x_0} = 1$ and~$n_x = 0$ for every~$x \neq x_0$, we use~$[x_0]$ to denote this divisor.
When there is no danger of confusion, sometimes we use~$x_0$ to denote~$[x_0]$.

Let $\fD = \sum_{x \in X}n_x[x]$ be a divisor on~$X$.
The \emph{degree} and the \emph{support} of~$\fD$ are defined by
\begin{displaymath}
  \deg(\fD)
  \=
  \sum_{x \in X} n_x
  \text{ and }
  \supp(\fD)
  \=
  \{ x \in X : n_x \neq 0 \},
\end{displaymath}
respectively.
The divisor~$\fD$ is \emph{effective}, if for every~$x$ in~$X$ we have~$n_x \ge 0$.
For~$A\subseteq X$, the \emph{restriction of~$\fD$ to~$A$} is the divisor on~$X$ defined by
$$\fD|_{A} \= \sum_{x \in A}n_x[x].$$
For a set~$X'$ and a map~$f \colon X \to X'$, the \emph{push-forward action of~$f$ on divisors~$f_* \colon \Div(X) \to \Div(X')$} is the linear extension of the action of~$f$ on points.
In the particular case in which~$X' = G$ is a commutative group, also define~$f \colon \Div(X) \to G$ by
$$f(\fD)\=\sum_{x\in X} n_x f(x)\in G.$$
If~$X$ is a topological space and~$\fD$ is an effective divisor satisfying~$\deg(\fD) \ge 1$, then~$\odelta_{\fD} \= \frac{1}{\deg(\fD)} \sum_{x \in X} n_x \delta_x$ is a Borel measure on~$X$.
Note that in the case~$G = \R$ and~$f$ is measurable, we have
$$\int f \dd \odelta_{\fD}
=
\frac{f(\fD)}{\deg(\fD)}.$$

Since we are identifying~$\Ell(\C_p)$ with~$\C_p$ via~$j$, we identify divisors on~$\Ell(\C_p)$ and on~$\C_p$ accordingly.

\subsection{Hecke correspondences}
\label{ss:Hecke correspondences}
In this section we recall the construction and main properties of the Hecke correspondences.
For details we refer the reader to~\cite[Sections~7.2 and~7.3]{Shi71} for the general theory, or to the survey~\cite[Part~II]{DiaIm95}.

Let~$\K$ be an algebraically closed field of characteristic~$0$.
First, note that for every integer~$n \ge 1$ and divisor~$\fD$ in~$\Div(\Ell(\K))$, we have
\begin{displaymath}
  \deg(T_n(\fD))
  =
  \sigma_1(n) \deg(\fD).
\end{displaymath}
Moreover, for~$n = 1$ the correspondence~$T_1$ is by definition the identity on~$\Div(\Ell(\K))$.

  We also consider the linear extension of Hecke correspondences to~$\Div(\Ell(\K)) \otimes \Q$.
  
For an integer $N\geq 1$, denote by~$\Ell_0(N)$ the \emph{modular curve of level~$N$}.
It is a quasi-projective variety defined over~$\Q$.
The points of~$\Ell_0(N)$ over~$\K$ parametrize the moduli space of equivalence classes of pairs~$(E,C)$, where~$E$ is an elliptic curve over~$\K$ and~$C$ is a cyclic subgroup of~$E$ of order~$N$.
Here, two such pairs~$(E,C)$ and~$(E',C')$ are equivalent if there exists an isomorphism $\phi \colon E\to E'$ over~$\K$ taking~$C$ to~$C'$.
In particular, when~$N = 1$, for every algebraically closed field~$\K$ we can parametrize~$\Ell(\K)$ by~$\Ell_0(1)(\K)$, and~$\Ell_0(1)$ is isomorphic to the affine line $\A_{\Q}^1$.

For~$N > 1$, denote by~$\Phi_N(X,Y)$ the \emph{modular polynomial of level~$N$}, which is a symmetric polynomial in~$\Z[X,Y]$ that is monic in both~$X$ and~$Y$, see, \emph{e.g.}, \cite[Chapter~5, Sections~2 and~3]{Lan87}.
This polynomial is characterized by the equality
\begin{equation}\label{eq-mod-polynomial}
\Phi_N(j(E),Y)=\prod_{C\leq E \text{ cyclic of order }N}(Y-j(E/C)) \text{ for every } E \text{ in } \Ell(\K).
\end{equation}
This implies that a birational model for~$\Ell_0(N)$ is provided by the plane algebraic curve
\begin{equation}\label{eq-mod-equation}
\Phi_N(X,Y)=0.
\end{equation}

For each prime~$q$, let~$\alpha_q,\beta_q \colon \Ell_0(q)\to \Ell_0(1)$ be the rational maps over~$\Q$ given in terms of moduli spaces by 
$$ \alpha_q(E,C) \= E
\text{ and }
\beta_q(E,C) \= E/C. $$
In terms of the model~\eqref{eq-mod-equation} with~$N = q$, the rational maps~$\alpha_q$ and~$\beta_q$ correspond to the projections on the~$X$ and~$Y$ coordinate, respectively.
Denote by~$(\alpha_q)_*$ and~$(\beta_q)_*$ the push-forward action of~$\alpha_q$ and~$\beta_q$ on divisors, respectively, as in Section~\ref{ss:divisors}.
Denote also by~$\alpha_q^{*}$ the pull-back action of~$\alpha_q$ on divisors, defined at~$x$ in~$\Ell_0(1)(\K)$ by
\begin{displaymath}
  \alpha_q^{*}(x)
  \=
  \sum_{\substack{y\in \Ell_0(q)(\K) \\ \alpha_q(y)=x}} \deg_{\alpha_q}(y)[y],
\end{displaymath}
where~$\deg_{\alpha_q}(y)$ is the local degree of~$\alpha_q$ at~$y$.
This definition is extended by linearity to arbitrary divisors.
The pull-back action~$\beta_q^{*}$ of~$\beta_q$ is defined in a similar way.
Then the Hecke correspondence $T_q \colon \Div(\Ell(\K))\to \Div(\Ell(\K))$ is recovered as
  \begin{displaymath}
    T_q
    =
    (\alpha_q)_* \circ \beta_q^*
    =
    (\beta_q)_* \circ \alpha_q^*,
\end{displaymath}
where the second equality follows from the first and from the symmetry of~$T_q$.

For an arbitrary integer~$n \ge 2$, the correspondence~$T_n$ can be recovered from different~$T_q$'s, for~$q$ running over prime divisors of~$n$, by using the identities
\begin{equation}\label{eq-Hecke-Tpr}
  T_{q^r}
  =
  T_q \circ T_{q^{r-1}} - q \cdot T_{q^{r-2}} \text{ for } q \text{ prime and } r \geq 2;
\end{equation}
\begin{equation}
  \label{eq-multiplicativity-Hecke-operators}
  T_{\ell} \circ T_{m}
  =
  T_{\ell m} \, \text{ for } \ell, m \geq 1 \text{ coprime}.
\end{equation}

We conclude this section with the following lemma used in Sections~\ref{ss:canonical branch} and~\ref{ss:ordinary orbits}.

\begin{lemma}\label{lemma-Hecke-continuity}
  Let $n\geq 1$ be an integer.
  For~$E$ in~$\Ell(\C_p)$, the divisor~$T_n(E)$ varies continuously with respect to~$E$ in the following sense: For every commutative topological group~$G$ and every continuous function $f \colon \Ell(\C_p)\to G$, the function $T_nf \colon \Ell(\C_p)\to G$ given by
$$T_nf(E)
\=
f(T_n(E))$$
is continuous.
In particular, for every open and closed subset $A\subseteq \Ell(\C_p)$, the integer valued map
$$E\mapsto \deg\left(T_n(E)|_A\right)$$
is locally constant.
\end{lemma}
\begin{proof}
      We first treat the case where~$n$ equals a prime number $q$.
  Let~$P_0(X)$, \ldots, $P_{q}(X)$ be the polynomials in~$\mathbb{Z}[X]$ such that
  \begin{displaymath}
    \Phi_q(X,Y)
    =
    P_0(X)+P_1(X)Y + \ldots + P_q(X) Y^q + Y^{q+1}.
  \end{displaymath}
  Let $(E_m)_{m=1}^{\infty}$ be a sequence and~$E_0$ be a point in $\Ell(\C_p)$, such that $j(E_m) \to j(E_0)$ when~$m$ tends to infinity.
  Then for every~$k$ in~$\{0,1,\ldots,q\}$, we have $P_k(j(E_m))\to P_k(j(E_0))$ when~$m$ tends to infinity.
  It follows that the roots of the polynomial $\Phi_q(j(E_m),Y)$ converge to the roots of $\Phi_q(j(E_0),Y)$, in the following sense: For every~$m$ in~$\{0, 1, 2,\ldots\}$ we can find~$z_{m, 0}$, \ldots, $z_{m, q}$ in~$\C_p$, so that
  \begin{displaymath}
    \Phi_q(j(E_m),Y)
    =
    \prod_{k=0}^{q} (Y-z_{m, k}),
  \end{displaymath}
  and so that for every~$k$ in~$\{0,1,\ldots,q\}$ we have $z_{m, k} \to z_{0, k}$ when~$m$ tends to infinity, see for example~\cite[Theorem~2]{Bri06}.
  For each~$m$ in~$\{0, 1, 2,\ldots\}$ and~$k$ in~$\{0,1,\ldots,q\}$, let~$E_{m,k}$ be the curve in~$\Ell(\C_p)$ with~$j(E_{m,k})=z_{m,k}$.
  By the definition of~$T_q$ and~\eqref{eq-mod-polynomial}, we have for every~$m \ge 0$
  \begin{displaymath}
    T_q(E_m)
    =
    \sum_{k=0}^q [E_{m,k}].
  \end{displaymath}
  Since for every~$k$ in~$\{0, 1,\ldots,q\}$ we have~$j(E_{m,k})\to j(E_{0,k})$  when~$m$ tends to infinity, we conclude that for every continuous function $f \colon \Ell(\C_p)\to G$ we have 
  \begin{displaymath}
    T_qf(E_m)
    =
    \sum_{ k=0}^q f(E_{k,m})
    \to
    \sum_{ k=0}^q f(E_{k,0})
      =
      T_qf(E_0).
  \end{displaymath}
  This proves that $T_qf$ is continuous.

  We now treat the general case by using multiplicative induction, the relations~\eqref{eq-Hecke-Tpr} and~\eqref{eq-multiplicativity-Hecke-operators}, and the fact that for every pair of linear maps $L, \tL \colon \Div(\Ell(\C_p)) \to \Div(\Ell(\C_p))$, every pair of integers~$m, \wtm$, and every function~$F \colon \Ell(\C_p) \to G$, one has
  \begin{equation}
    \label{eq:1}
    (L \circ \tL)(F)
    =
    \tL(L(F))
    \text{ and }
    (mL  + \wtm \tL)(F)
    =
    m L(F) + \wtm \tL(F).    
  \end{equation}
  Denote by~$I$ the set of those integers~$n \ge 1$ such that for every continuous function $f \colon \Ell(\C_p)\to G$, the function~$T_n(f)$ is also continuous.
  Clearly~$I$ contains~$1$, since for every function~$f$ we have $T_1(f)=f$.
  By the proof given above, $I$ contains all prime numbers.
  Let~$n \ge 1$ be a given integer having each divisor in~$I$, and let~$q$ be a prime number.
  Let~$s \ge 0$ and~$n_0 \ge 1$ be the integers such that~$n=q^sn_0$, and such that~$q$ does not divide $n_0$.
  Then by the relations~\eqref{eq-Hecke-Tpr} and~\eqref{eq-multiplicativity-Hecke-operators}, and by~\eqref{eq:1}, we have
\begin{displaymath}
  T_{qn}(f)
  =
  T_{q^{s+1}n_0}(f)
  =
  T_{n_0}(T_{q^{s+1}}(f)),
\end{displaymath}
and for~$s \ge 1$
\begin{displaymath}
  T_{q^{s+1}}(f)
  =
  T_{q^{s}}(T_{q}(f)) - qT_{q^{s-1}}(f).
\end{displaymath}
Since $n_0$, $q$, $q^s$, and $q^{s-1}$ if $s\geq 1$, are all in~$I$, we conclude that $T_{qn}(f)$ is continuous, and that~$qn$ is in $I$.    
This completes the proof of the multiplicative induction step, and of the first part of the lemma.

The second part of the lemma is an easy consequence of the first.
Indeed, let~$A\subseteq \Ell(\C_p)$ be an open and closed subset.
Then the function~$\bfone_A$ is continuous and the first part implies that
$$ E \mapsto T_n\bfone_A(E) = \bfone_A(T_n(E)) = \deg(T_n(E)|_A) $$
is also continuous.
But~$T_n\bfone_A$ has integer values, hence it must be locally constant.
This completes the proof of the lemma.
\end{proof}

\subsection{Hecke orbits of \CM{} points and an estimate on class numbers}
\label{sect-prelim-CM}
In this section we first recall a special case of a formula of Zhang describing the effect of Hecke correspondences on \CM{} points (Lemma~\ref{Zhang general}), which is used in Sections~\ref{s:ordinary CM points}, \ref{s:supersingular CM points} and~\ref{ss:ordinary orbits}.
To do this, and for the rest of the paper, for every discriminant~$D$ we consider~$\Lambda_D$ as a divisor.
We also use Siegel's classical lower bound on class numbers of quadratic imaginary extensions of~$\Q$, to give the following estimate used in the proof of  Theorem~\ref{t:CM points}: For every~$\varepsilon > 0$ there is a constant~$C > 0$ such that for every negative discriminant~$D$, we have
\begin{equation}
  \label{eq-Siegel-low-general}
  h(D)
  \=
  \deg(\Lambda_D)
  \ge
  C |D|^{\frac{1}{2}-\varepsilon}.
\end{equation}
In this section we follow~\cite[Section~2.3]{CloUll04}, adding some details for the benefit of the reader.

We use~$d$ to denote a negative fundamental discriminant.
For each discriminant~$D$ there is a unique negative fundamental discriminant~$d$ and integer~$f \ge 1$ such that~$D = d f^2$.
These are the \emph{fundamental discriminant} and \emph{conductor of~$D$}, respectively.
We denote by~$\cO_{d,f}$ the unique order of discriminant~$D$ in the quadratic imaginary extension~$\Q(\sqrt{d})$ of~$\Q$ and put
  $$ w_{d, f}
  \=
  \# \left( \cO_{d, f}^{\times} / \Z^{\times} \right)
  =
   \left(\# \cO_{d, f}^{\times}\right) / 2. $$
The integer~$f$ is the index of~$\cO_{d,f}$ inside the ring of integers of~$\Q(\sqrt{d})$. 
Note that~$w_{-3, 1} = 3$, $w_{-4, 1} = 2$, and that in all the remaining cases~$w_{d, f} = 1$.

Recall that the \emph{Dirichlet convolution} of two functions~$g, \wtg \colon \N \to \C$, is defined by
\begin{displaymath}
  (g \ast \wtg)(n)
  \=
  \sum_{d \in \N, d \mid n} g(d) \wtg \left(\frac{n}{d} \right).
\end{displaymath}
Given a fundamental discriminant~$d$, denote by~$R_d \colon \N \to \N \cup \{ 0 \}$ the function that to each~$n$ in~$\N$ assigns the number of integral ideals of norm~$n$ in the ring of integers of~$\Q ( \sqrt{d} )$. 
Moreover, denote by~$R_d^{-1}$ the inverse of~$R_d$ with respect to the Dirichlet convolution. 

  \begin{lemma}\label{Zhang general}
  For every fundamental discriminant $d<0$ and any pair of coprime integers~$f \ge 1$ and~$\wtf \ge 1$, we have the relations  
  \begin{equation}
  \label{e:general Zhang formula}
T_f \left( \frac{\Lambda_{d\wtf^2}}{w_{d, \wtf}} \right)
  =
  \sum_{f_0 \in \N, f_0 \mid f} R_d \left(\frac{f}{f_0}\right) \frac{\Lambda_{d(f_0 \wtf)^2}}{w_{d, f_0 \wtf}};
  \end{equation}
  \begin{equation}
    \label{e:general inverse Zhang formula}
   \frac{\Lambda_{d(f \wtf)^2}}{w_{d,f\wtf}}
    =
    \sum_{f_0 \in \N, f_0 \mid f} R_d^{-1} \left( \frac{f}{f_0} \right) T_{f_0}\left(\frac{\Lambda_{d\wtf^2}}{w_{d,\wtf}}\right).
  \end{equation}
  If in addition~$f$ is not divisible by~$p$, then we have
  \begin{equation}
    \label{e:p Zhang}
    \Lambda_{d (pf)^2}
    =
    \begin{cases}
      T_p \left( \frac{\Lambda_{d f^2}}{w_{d, f}} \right) - 2 \frac{\Lambda_{d f^2}}{w_{d, f}}
      & \text{if $p$ splits in~$\Q(\sqrt{d})$};
      \\
      T_p \left( \frac{\Lambda_{d f^2}}{w_{d, f}} \right) - \frac{\Lambda_{d f^2}}{w_{d, f}}
      & \text{if $p$ ramifies in~$\Q(\sqrt{d})$};
      \\
      T_p \left( \frac{\Lambda_{d f^2}}{w_{d, f}} \right)
      & \text{if $p$ is inert in~$\Q(\sqrt{d})$},
    \end{cases}
  \end{equation}
and for every integer~$m \ge 2$ we have
  \begin{multline}
    \label{e:pm Zhang}
    \Lambda_{d (p^m f)^2}
    \\ =
    \begin{cases}
      T_{p^m} \left(\frac{\Lambda_{d f^2}}{w_{d, f}} \right)
      - 2 T_{p^{m - 1}} \left(\frac{\Lambda_{d f^2}}{w_{d, f}} \right)
      + T_{p^{m - 2}} \left(\frac{\Lambda_{d f^2}}{w_{d, f}} \right)
      & \text{if $p$ splits in~$\Q(\sqrt{d})$};
      \\
      T_{p^m} \left(\frac{\Lambda_{d f^2}}{w_{d, f}} \right)
      - T_{p^{m - 1}} \left(\frac{\Lambda_{d f^2}}{w_{d, f}} \right)
      & \text{if $p$ ramifies in~$\Q(\sqrt{d})$};
      \\
      T_{p^m} \left(\frac{\Lambda_{d f^2}}{w_{d, f}} \right)
      - T_{p^{m - 2}} \left(\frac{\Lambda_{d f^2}}{w_{d, f}} \right)
      & \text{if $p$ is inert in~$\Q(\sqrt{d})$}.
    \end{cases}
  \end{multline}
\end{lemma}

To prove this lemma, we first record the following identity, which is also used in the proof~\eqref{eq-Siegel-low-general} below and of Lemma~\ref{conteoo} in Section~\ref{ss:ordinary orbits}.
Let~$\psi_d$ be the quadratic character associated to~$K=\Q(\sqrt{d})$, which is given by the Kronecker symbol~$\left(\frac{d}{\cdot}\right)$, and denote by~$\bfone \colon \N \to \C$ the constant function equal to~$1$.
Then we have the equality of functions
\begin{equation}
  \label{eq-Rd-char}
  R_d
  =
  \psi_d \ast \bfone.
\end{equation}
In fact, if we denote by~$\zeta(s)$ the Riemann zeta function, by~$\zeta_K(s)$ the Dedekind zeta function associated to~$K$, and by~$L(\psi_d,s)$ the Dedekind $L$-function associated to~$\psi_d$, then the formula above is equivalent to the factorization $\zeta_K(s)=\zeta(s)L(\psi_d,s)$, whose proof can be found for example in~\cite[Proposition~10.5.5 in p.~219]{Coh07b}, or~\cite[Chapter~XII, Section~1, Theorem~1]{Lan94}.

\begin{proof}[Proof of Lemma~\ref{Zhang general}]
  From the M{\"o}bius inversion formula we deduce that~\eqref{e:general Zhang formula} and~\eqref{e:general inverse Zhang formula} are equivalent.
  Hence, it is enough to prove  \eqref{e:general inverse Zhang formula}.
  We have the following formula of Zhang
  \begin{equation}
    \label{e:Zhang formula}
  T_f \left( \frac{\Lambda_{d}}{w_{d, 1}} \right)
  =
  \sum_{f_0 \in \N, f_0 \mid f}R_d\left(\frac{f}{f_0}\right) \frac{\Lambda_{df_0^2}}{w_{d, f_0}},
\end{equation}
see for example~\cite[\emph{Lemme}~2.6]{CloUll04} or~\cite[Proposition~4.2.1]{Zha01}.
Applying the M{\"o}bius inversion formula, one obtains
\begin{equation}
  \label{e:inverse Zhang formula}
  \frac{\Lambda_{df^2}}{w_{d, f}}
  =
  \sum_{f_0 \in \N, f_0 \mid f} R_d^{-1}\left(\frac{f}{f_0}\right)T_{f_0} \left( \frac{\Lambda_{d}}{w_{d, 1}} \right).
\end{equation}
On the other hand, note that if~$f$ and~$\wtf$ in~$\N$ are coprime, then by \eqref{eq-multiplicativity-Hecke-operators} and ~\eqref{e:inverse Zhang formula}, we obtain \eqref{e:general inverse Zhang formula}. 

Finally, \eqref{e:p Zhang} and~\eqref{e:pm Zhang} are a direct consequence of~\eqref{e:general Zhang formula}, \eqref{eq-Rd-char} and the fact that~$\psi_d(p) = 1$ (resp.~$0$, $-1$) if~$p$ splits (resp. ramifies, is inert) in~$\Q(\sqrt{d})$.
\end{proof}
  
To prove~\eqref{eq-Siegel-low-general}, recall from the theory of complex multiplication that for a fundamental discriminant~$d$ the number~$h(d)$ equals the class number of the quadratic extension~$\Q(\sqrt{d})$ of~$\Q$, see for example~\cite[Corollary~10.20]{Cox13}.
A celebrated result by Siegel states that for every~$\varepsilon>0$ there exists a constant~$C>0$ such that for every fundamental discriminant~$d<0$ we have
\begin{equation}
  \label{eq-Siegel-low-bound}
h(d)\geq C|d|^{\frac{1}{2}-\varepsilon},
\end{equation}
see for example~\cite{Sie35}, or~\cite[Chapter~XVI, Section~4, Theorem~4]{Lan94}.
On the other hand, by~\cite[Chapter~8, Section~1, Theorem~7]{Lan87} for every integer~$f \ge 2$ we have
\begin{equation}
  \label{eq:2}
  h(df^2)
=
\frac{w_{d, f}}{w_d} h(d) f\prod_{q\mid f, \text{ prime}} \left(\frac{q-\psi_d(q)}{q}\right).  
\end{equation}
Given~$\varepsilon>0$, there are~$C'$ in~$(0, 1)$ and~$N$ in~$\N$ such that $\frac{q-1}{q}\geq q^{-\varepsilon}$ for every~$q > N$ and $\frac{q-1}{q} \geq C' q^{-\varepsilon}$ for every $2 \le q\leq N$.
Hence, for every integer~$f \ge 2$ we have
$$\prod_{q\mid f, \text{ prime}} \left(\frac{q-\psi_d(q)}{q}\right)
\geq
\prod_{q\mid f, \text{ prime}} \left(\frac{q-1}{q}\right)
\geq
(C')^{N}\prod_{q\mid f, \text{ prime}}q^{-\varepsilon}
\geq
(C')^{N} f^{-\varepsilon}.$$
Combined with~\eqref{eq-Siegel-low-bound} and~\eqref{eq:2}, this completes the proof of~\eqref{eq-Siegel-low-general}.

\subsection{The Berkovich affine line over~$\C_p$ and the Gauss point}
\label{ss:berkovich}
We refer the reader to~\cite{Ber90} for the general theory of Berkovich spaces, and to~\cite[Chapter~1]{BakRum10} for the special case of the Berkovich affine line over~$\C_p$, which is the only Berkovich space of relevance in this paper.

The \emph{Berkovich affine line over~$\C_p$}, which we denote by~$\AKber$, is a topological space defined as follows.
As a set, $\AKber$ is the collection of all multiplicative seminorms on the polynomial ring~$\C_p[X]$ that take values in~$\Rzp$ and that extend the $p$-adic norm~$| \cdot |_p$ on~$\C_p$.
Hence, a point $x\in \AKber$ is given by a map $x \colon \C_p[X]\to \Rzp$ satisfying for every~$a$ in~$\C_p$ and for all~$f$ and~$g$ in~$\C_p[X]$,
\begin{displaymath}
  x(a) = |a|_p,
  x(f + g) \leq x(f) + x(g)
  \text{ and }
  x(fg)=x(f)x(g).
\end{displaymath}
The topology of~$\AKber$ is the weakest topology such that for every $f\in \C_p[X]$, the function $\AKber\to \C_p$ given by $x\mapsto x(f)$ is continuous.
The topological space~$\AKber$ is Hausdorff, locally compact, metrizable and path-connected.
It contains~$\C_p$ as a dense subspace via the map $\iota \colon \C_p\to \AKber$ given, for $z\in \C_p$ and $f\in \C_p[X]$, by $\iota(z)(f) \= |f(z)|_p$. 
We identify divisors on~$\C_p$ and on~$\iota(\C_p)$ accordingly.

The \emph{canonical point} or \emph{Gauss point}~$\xcan$ of~$\AKber$ is the Gauss norm
$$ \sum_{n=0}^N a_nX^n
\mapsto
\sup\left\{\left|\sum_{n=0}^N a_nz^n\right|_p:z\in \cO_p\right\}=\max\{|a_n|_p:n\in\{0,\ldots,N\}\}.$$

Given $a\in \C_p$ and $r>0$, define
\begin{align*}
  \bfD(a,r)
  &\=
    \{x\in \C_p:|x-a|_p< r\};
  \\
  \bfD^{\infty}(a,r)
  &\=
    \{x\in \C_p: |x-a|_p>r\};
  \\
  \cD(a,r)
  &\=
    \{x\in \AKber : x(X-a)< r\};
  \\
  \cD^{\infty}(a,r)
  &\=
    \{x\in \AKber : x(X-a)>r\}.
\end{align*}
A basis of neighborhoods of~$\xcan$ in~$\AKber$ is given by the collection of sets
\begin{equation}
  \label{eq-basis-affinoid}
  \cA(A;R)
  \=
  \cD(0,R) \cap  \bigcap_{a \in A}\cD^{\infty}(a,R^{-1}),
\end{equation}
where $R > 1$ and~$A$ is a finite subset of~$\cO_p$.

We conclude this section with the following result.
Recall that a sequence of Borel probability measures $(\mu_n)_{n \in \N}$ on a topological space~$X$ \emph{converges weakly to a Borel measure~$\mu$ on~$X$}, if for every continuous and bounded function $f \colon X \to \R$ we have
$$\lim_{n\to \infty}\int f \dd \mu_n
=
\int f \dd \mu, $$
see, \emph{e.g.}, \cite[Section~1.1]{Bil99}.

\begin{lemma}
  \label{l:berkovich}
  Let $(\fD_n)_{n \in \N}$ be a sequence of effective divisors on~$\C_p$ such that for every~$n$ we have~$\deg(\fD_n) \ge 1$.
  Then, the following are equivalent: 
\begin{enumerate}
\item[$(i)$]
  $\odelta_{\iota(\fD_n)}\to \delta_{\xcan}$ weakly as $n \to \infty$.
\item[$(ii)$]
  For every~$R > 1$ and every~$a$ in~$\cO_p$, we have for~$\bfD = \bfD(a,R^{-1})$ and $\bfD=\bfD^{\infty}(a,R)$,
  $$ \lim_{n\to \infty}\frac{\deg(\fD_n|_{\bfD})}{\deg(\fD_n)}
  =
  \lim_{n\to \infty }\odelta_{\fD_n}(\bfD)
  =
  0.$$
\end{enumerate}
\end{lemma}

  For the reader's convenience we provide a self-contained proof of this lemma, which applies to the Berkovich affine line over an arbitrary complete and algebraically closed field.
  Using that~$\AKber$ is metrizable, the lemma can also be obtained as a direct consequence of the following observations: $(i)$ is equivalent to the assertion that for every neighborhood~$\cU$ of~$\xcan$ in~$\AKber$ we have
\begin{displaymath}
    \lim_{n\to \infty }\odelta_{\fD_n}(\cU)
  =
  1.
\end{displaymath}
This last statement is equivalent to the contrapositive of~$(ii)$.

\begin{proof}[Proof of Lemma~\ref{l:berkovich}]
  Assume that~$(i)$ holds and let~$R > 1$ and~$a$ in~$\cO_p$ be given.
  Note that the first equality in~$(ii)$ is a direct consequence of the definitions.
  To prove the second equality, take a continuous function $\phi \colon \Rzp \to [0,1]$ satisfying~$\phi(1)=0$ and~$\phi(t)=1$ for $0\leq t\leq R^{-1}$ and for~$t \ge R$.
  Let $\alpha \colon \AKber\to \R$ be the continuous function given by $\alpha(x)=x(X-a)$ and put $F \= \phi\circ \alpha$.
  By construction we have
  $$ F(\xcan)=\phi(1)=0
  \text{ and }
  F(x)=1 \text{ for all }x\in \cD(a,R^{-1}) \cup \cD^{\infty}(a, R). $$
Using that for~$z\in \C_p$ we have
\begin{multline}\label{eq-iota-z}
  z \in \bfD(a,R^{-1})
  \Leftrightarrow
  \iota(z)\in \cD(a,R^{-1})
  \text{ and }
  z \in \bfD^{\infty}(a,R)
  \Leftrightarrow
  \iota(z)\in \cD^{\infty}(a, R),
\end{multline}
we get
$$ 0
\leq
\odelta_{\fD_n}(\bfD(a,R^{-1}) \cup \bfD^{\infty}(a, R))
=
\odelta_{\iota(\fD_n)}(\cD(a,R^{-1}) \cup \cD^{\infty}(a, R))
\leq
\int F \dd \odelta_{\iota(\fD_n)}.$$
Since $F$ is continuous and bounded, our hypothesis~$(i)$ implies that
\begin{displaymath}
  \odelta_{\fD_n}(\bfD(a,R^{-1})) \to 0
\text{ and }
\odelta_{\fD_n}(\bfD^{\infty}(a,R)) \to 0
\text{ as }
n \to \infty.
\end{displaymath}
This completes the proof of the implication~$(i) \Rightarrow (ii)$.

Now, assume that~$(ii)$ holds, let $F \colon \AKber\to \R$ be a continuous and bounded function and let~$\varepsilon > 0$ be given.
Since the sets~\eqref{eq-basis-affinoid} form a basis of neighborhoods of~$\xcan$, there are~$R > 1$ and a finite subset~$A$ of~$\cO_p$ such that
\begin{equation}
  \label{eq-cont-F}
  |F(x)-F(\xcan)| < \varepsilon
  \ \text{ for all } x \in \cA(A; R).
\end{equation}
Let $R'$ in $(1,R)$ be fixed. From the definition of~$\cA \= \cA(A; R)$, we have
\begin{displaymath}
  \cA' \= \AKber \setminus \cA \subseteq \cD^{\infty}(0, R') \cup \bigcup_{a \in A} \cD(a, (R')^{-1}).
\end{displaymath}
Using~\eqref{eq-iota-z} and~$(ii)$ with~$R$ replaced by~$R'$ and with~$a$ in~$A \cup \{ 0 \}$, we obtain
\begin{align*}
    \deg(\iota(\fD_n)|_{\cA'})
  & \le
  \deg(\iota(\fD_n)|_{\cD^{\infty}(0,R')}) + \sum_{a \in A} \deg(\iota(\fD_n)|_{\cD(a, (R')^{-1})})
  \\ & =
  \deg(\fD_n|_{\bfD^{\infty}(0,R')}) + \sum_{a \in A} \deg(\fD_n|_{\bfD(a, (R')^{-1})})
  \\ & =
  o(\deg(\iota(\fD_n))).
\end{align*}
Together with our choice of~$\cA(A; R)$, this implies
\begin{align*}
    \left| \int F \dd \odelta_{\iota(\fD_n)}-F(\xcan) \right|
    & \leq 
    \left| \frac{F(\iota{(\fD_n)}|_\cA)-F(\xcan)\deg(\iota(\fD_n)|_\cA)}{\deg(\fD_n)} \right|
    \\ & \quad
    + \left| \frac{F(\iota{(\fD_n)}|_{\cA'})-F(\xcan)\deg(\iota(\fD_n)|_{\cA'})}{\deg(\fD_n)} \right|
    \\ & \le
    \varepsilon + 2 \left( \sup_{x \in \AKber} |F(x)| \right) \frac{\deg(\iota(\fD_n)|_{\cA'})}{\deg(\iota(\fD_n)},
\end{align*}
and therefore
$$\limsup_{n\to \infty}\left|\int F \dd \odelta_{\iota(\fD_n)}-F(\xcan)\right|\leq \varepsilon.$$
Since $\varepsilon > 0$ is arbitrary, this completes the proof of the implication~$(ii) \Rightarrow (i)$ and of the lemma.
\end{proof}

\section{\CM{} points in the ordinary reduction locus}
\label{s:ordinary CM points}

The purpose of this section is to give a strengthened version of Theorem~\ref{t:CM points} in the case where all the discriminants in the sequence~$(D_n)_{n = 1}^{\infty}$ are $p$-ordinary (Theorem~\ref{t:ordinary CM points}$(ii)$ in Section~\ref{ss:dynamics of t}).
An important tool is ``the canonical branch~$\t$'' of~$T_p$ on~$\Ord$, which is defined using the canonical subgroup in Section~\ref{ss:canonical branch}.
We use it to give, for every integer~$m \ge 1$, a simple formula of~$T_{p^m}$ (Proposition~\ref{prop-tpm} in Section~\ref{ss:canonical branch}).
Moreover, we show that $p$-ordinary \CM{} points correspond precisely to the preperiodic points of~$\t$ on~$\Ord$ (Theorem~\ref{t:ordinary CM points}$(i)$).
Once these are established, Theorem~\ref{t:ordinary CM points}$(ii)$ follows from dynamical properties of~$\t$ on~$\Ord$ (Lemma~\ref{periodicpoint}).
In Appendix~\ref{s:canonical analyticity} we extend and further study the canonical branch~$\t$ of~$T_p$.

We use properties of reduction morphisms that are stated in most of the classical literature only for elliptic curves over discrete valued fields.
To extend the application of these results to elliptic curves over~$\C_p$ we use the continuity of the Hecke correspondences (Lemma~\ref{lemma-Hecke-continuity} in Section~\ref{ss:Hecke correspondences}).
To this purpose, we introduce the following notation: $\Qpun$ is the maximal unramified extension of~$\Q_p$ inside~$\Qpalg$, and~$\Cpun$ its completion.
Then, $\Cpun$ is an infinite degree extension of~$\Q_p$ with the same valuation group and with residue field~$\Fpalg$.
The algebraic closure~$\Cpunalg$ of $\Cpun$ inside~$\C_p$ is dense in $\C_p$.
Since~$\Cpunalg$ can be written as the union of finite extensions of~$\Cpun$, it follows that every elliptic curve in~$\Ell(\Cpunalg)$ can be defined over a complete discrete valued field with residue field~$\Fpalg$.
The same holds for finite subgroups and isogenies between elliptic curves over~$\Cpunalg$.

In what follows, we use~$\Ordun \= \Ord \cap \Ell(\Cpunalg)$.

\subsection{The canonical branch of~$T_p$ on~$\Ord$}
\label{ss:canonical branch}
In this section we define a branch of the Hecke correspondence~$T_p$ on~$\Ord$ that we use to give a simple description, for every integer~$m \ge 1$, of~$T_{p^m}$ that is crucial in what follows (Proposition~\ref{prop-tpm}).
See also Appendix~\ref{s:canonical analyticity}.
We start recalling the following result describing the endomorphism ring of the reduction of a \CM{} point in the ordinary locus.

\begin{proposition}[\cite{Lan87}, Chapter~13, Section~4, Theorem~12]
  \label{redendring}
  Let $d<0$ be a fundamental discriminant and let~$f \ge 1$ and~$m \ge 0$ be integers such that~$f$ is not divisible by~$p$.
  Then, for an elliptic curve~$E$ defined over a discrete valued subfield of~$\C_p$ having ordinary reduction, $\End(E) \simeq \cO_{d, p^m f}$ implies that the reduction~$\tE$ of~$E$ satisfies $\End(\tE)\simeq \cO_{d, f}$.
  In particular, if~$\End(E)$ is an order in a quadratic imaginary extension of~$\Q$ whose conductor is not divisible by~$p$, then the reduction map $\End(E)\to \End(\tE)$ is an isomorphism. 
\end{proposition}

To define the canonical branch of~$T_p$ on~$\Ord$,
we use the \emph{canonical subgroup} of an elliptic curve~$E$ in~$\Ordun$, which is defined as the unique subgroup of order~$p$ of~$E$ in the kernel of the reduction morphism $E\to \tE$.
Equivalently, $H(E)$ is the kernel of the reduction morphism $E[p]\to \tE[p]$.
For an elliptic curve~$e\in \Ell(\Fpalg)$ denote by $\Frob \colon e\to e^{(p)}$ the Frobenius morphism, which is the isogeny given in affine coordinates by $(x,y)\mapsto (x^p,y^p)$.

\begin{theorem}
  \label{teo-Deuring}
  \

  \begin{enumerate}
  \item[$(i)$]
    For $E\in \Ordun$ the natural isogeny $\varphi \colon E\to E/H(E)$ reduces to the Frobenius morphism $\Frob \colon \tE\to \tE^{(p)}$.
    Moreover, the kernel of the isogeny dual to~$\varphi$ is different from the canonical subgroup of~$E/H(E)$.
  \item[$(ii)$]
    For each ordinary elliptic curve $e\in \Ell(\Fpalg)$ there exists a unique elliptic curve $e^{\uparrow}\in \Ell(\Cpunalg)$ reducing to $e$ for which the reduction map induces a ring isomorphism $\End(e^{\uparrow})\simeq\End(e)$.
  \item[$(iii)$]
    Given two ordinary elliptic curves $e_1,e_2\in \Ell(\Fpalg)$, the reduction map induces a group isomorphism $\Hom(e_1^{\uparrow},e_2^{\uparrow})\simeq \Hom(e_1,e_2)$. In particular, the Frobenius morphism $\Frob \colon e\to e^{(p)}$ lifts to an isogeny $e^{\uparrow}\to (e^{(p)})^{\uparrow}$ with kernel $H(e^{\uparrow})$, and $e^{\uparrow}/H(e^{\uparrow})=(e^{(p)})^{\uparrow}$.
\end{enumerate}
\end{theorem}
\begin{proof}
Item~$(i)$ follows from the definition of canonical subgroup and properties of reduction morphisms, see, \emph{e.g.}, \cite[Proof of Lemma~8.7.1]{DiaShu05}.
\textcolor{red}{Item}~$(ii)$ is usually known as ``Deuring's Lifting Theorem'', see for example~\cite{Deu41} or~\cite[Chapter~13, Section~5, Theorem~14]{Lan87}.
  Item~$(iii)$ is another known consequence of Deuring's work.
  To prove surjectivity, first note that every isogeny in~$\Hom(e_1, e_2)$ can be written as a composition of Frobenius morphisms, of duals of Frobenius morphisms, and of an isogeny whose degree is not divisible by~$p$.
  In view of items~$(i)$ and~$(ii)$, and of Proposition~\ref{redendring}, we can restrict to the case of an isogeny of degree~$n$ not divisible by~$p$.
  This case is a direct consequence of item~$(ii)$, and the fact that the reduction morphism~$E \to \tE$ induces a bijective map~$E[n] \to \tE[n]$, see for example~\cite[Chapter~VII, Proposition~3.1(b)]{Sil09}.
\end{proof}

  The following result is due to Tate in the case~$p = 2$ and to Deligne in the general case.
To state it, define
\begin{equation}
  \label{defoff}
  \begin{array}{rccl}
    \t \colon & \Ordun  & \to & \Ordun
                                \\ & E & \mapsto & \t(E) \= E/H(E),
  \end{array}
\end{equation}
and for~$\ss$ in~$\tSups$ put
\begin{equation}
  \label{eq:8}
  \delta_{\ss}
\=
\begin{cases}
  1
  & \text{if } p \geq 5, j(\ss)\neq 0,1728;
  \\
  3
  & \text{if }p\geq 5, j(\ss)= 0;
  \\
  2
  & \text{if }p\geq 5, j(\ss)= 1728;
  \\
  6
  & \text{if }p=3, j(\ss)=0=1728;
  \\
  12
  & \text{if }p=2, j(\ss)=0=1728.
\end{cases}
\end{equation}
Note that in all the cases~$\delta_{\ss} = (\# \Aut(\ss)) / 2$, see, \emph{e.g.}, \cite[Chapter~III, Theorem~10.1]{Sil94a}.

\begin{theorem}
  \label{teo-Deligne}
  For each~$\ss$ in~$\tSups$ choose~$\beta_{\ss}$ in~$\bfD(j(\ss)) \cap \Qpun$, so that~$\pi(\beta_{\ss})=j(\ss)$, and put~$\delta_{\ss}' \= \delta_{\ss}$ if~$\beta_{\ss} = 0$ and~$p \neq 3$ or if~$\beta_{\ss} = 1728$ and~$p \neq 2$, and~$\delta_{\ss}' \= 1$ otherwise.
Then, the map~$\t$ admits an expansion of the form
\begin{equation}
  \label{eq:6}
  \t(z)
  =
  z^p + pk(z)+ \sum_{\ss \in \tSups} \sum_{n=1}^{\infty} \frac{ A_n^{(\ss)}}{(z-\beta_{\ss})^n},
\end{equation}
where $k(z)$ is a polynomial of degree~$p-1$ in~$z$ with coefficients in $\Z$, and for each~$n \ge 1$ the coefficient~$A_n^{(\ss)}$ belongs to $\Q_p(\{ \beta_{\ss} : \ss \in \tSups \})$ and 
  \begin{equation}
    \label{eq:13}
  \ord_p(A_n^{(\ss)})
\ge
\delta_{\ss}' \left( \frac{1}{p+1}+n \frac{p}{p+1} \right).    
  \end{equation}

In particular, $\t(z)$ extends to a rigid analytic function $\Ord \to \Ord$ of degree~$p$ that we also denote by~$\t$.
\end{theorem}

  For~$p \ge 5$, this result is proved in~\cite[Chapter~7]{Dwo69}.
  In the case~$\delta_{\ss}' > 1$, \eqref{eq:13} can be obtained from the method of proof described in~\cite{Dwo69}, or from the estimate in~\cite[p.~80]{Dwo69} combined with the fact that~$\ord_p(A_n^{(\ss)})$ is an integer and that~$\beta_{\ss} = 0$ implies~$p \equiv 2 \mod 3$.  
  For~$p = 2$ and~$3$, this result is stated in~\cite[p.~89]{Dwo69} with a weaker version of~\eqref{eq:13}.
  We provide the details of the proof when~$p = 2$ and~$3$, see Proposition~\ref{p:low canonical analyticity} in Appendix~\ref{s:canonical analyticity}.

The theorem above implies that~$\t$ extends to a rigid analytic map from~$\Ord$ to itself.
  We denote this extension also by~$\t$ and call it the \emph{canonical branch of~$T_p$ on~$\Ord$}.

For $z\in \Ord$, let~$\t^{*}(z)$ be the divisor on~$\Ord$ given by
$$\t^{*}(z)
\=
\sum_{\substack{w\in \Ord \\ \t(w)=z}}\deg_\t(w) [w],$$
where~$\deg_\t(w)$ is the local degree of~$\t$ at~$w$.
Note that by Theorem~\ref{teo-Deligne} the rigid analytic map~$\t \colon \Ord \to \Ord$ is of degree~$p$, so for~$z$ in~$\Ord$ we have
\begin{displaymath}
  \deg(\t^{*}(z)) = p
  \text{ and }
  \t_*(\t^{*} (z)) = p [z].
\end{displaymath}
As usual, for an integer~$i \ge 1$ we denote by~$\t^i$ the $i$-th fold composition of~$\t$ with itself.
We also use~$\t^0$ to denote the identity on~$\Ord$.

\begin{proposition}
  \label{prop-tpm}
For every~$E$ in~$\Ord$ and every integer~$m \ge 1$, we have
\begin{equation}\label{tpm}
  T_{p^m}(E)
  =
  \sum_{i=0}^m (\t^*)^{m-i}([\t^{i}(E)]).
\end{equation}
\end{proposition}
When $m=1$, the relation~\eqref{tpm} reads
\begin{equation}\label{tp1}
  T_p(E)
  =
  \t^*(E) + [\t(E)].
\end{equation}
See Theorem~\ref{t:canonical analyticity} in Appendix~\ref{s:canonical analyticity} for an extension.

\begin{proof}
The relation~\eqref{tpm} for~$m \ge 2$ follows from~\eqref{tp1} by induction using the recursive formula~\eqref{eq-Hecke-Tpr}.
To prove~\eqref{tp1}, first note that for~$E$ in~$\Ord$ satisfying~$\deg_{\t}(E) \ge 2$ we have~$\t'(E) = 0$.
Therefore there are at most a finite number of such~$E$ in the affinoid~$\Ord$, see for example~\cite[Proposition~3.3.6]{FrevdP04}.
It follows that for every~$E$ in~$\Ord$ outside a finite set of exceptions, we have~$\# \supp(\t^{*}(E)) = p$.
Thus, the set~$\oD$ of all those~$E$ in~$\Ordun$ with this property is dense in~$\Ord$. To prove~\eqref{tp1} for~$E$ in~$\oD$, use the definition of~$T_p(E)$ and~$\t(E)$, and Theorem~\ref{teo-Deuring}$(i)$, to obtain
  \begin{displaymath}
    T_p(E)
=
[\t(E)]+\sum_{\substack{C\leq E,\#C=p \\ C\neq H(E)}}[E/C]
  =
  [\t(E)] + \t^*(E).
\end{displaymath}
To prove~\eqref{tp1} for an arbitrary~$E$ in~$\Ord$, first note that by Lemma~\ref{lemma-Hecke-continuity} for every open and closed subset~$A$ of~$\Ord$ the function
\begin{displaymath}
  E
  \mapsto
  \bfone_A(T_n(E) - \t^{*}(E) - [\t(E)])
  =
  \deg((T_n(E) - \t^{*}(E) - [\t(E)])|_A)
\end{displaymath}
is continuous.
Since it is equal to~$0$ on the dense subset~$\oD$ of~$\Ord$, we conclude that it is constant equal to~$0$.
Since this holds for every open and closed subset~$A$ of~$\Ord$, this proves~\eqref{tp1} and completes the proof of the lemma.
\end{proof}

\subsection{\CM{} points as preperiodic points}
\label{ss:dynamics of t}

The purpose of this section is to prove the following result.
In the case where all the discriminants in the sequence~$(D_n)_{n = 1}^{\infty}$ are $p$-ordinary, Theorem~\ref{t:CM points} is a direct consequence of item~$(ii)$ of this result together with~\eqref{eq-Siegel-low-general} and Lemma~\ref{l:berkovich}.

Given a set~$X$ and a map~$T \colon X \to X$, a point~$x$ in~$X$ is \emph{periodic} if for some integer~$r \ge 1$ we have~$T^r(x) = x$.
Then the integer~$r$ is a \emph{period of~$x$} and the smallest such integer is the \emph{minimal period of~$x$}.
Moreover, a point~$y$ is \emph{preperiodic} if it is not periodic and if for some integer~$m \ge 1$ the point~$T^m(y)$ is periodic.
We call the least such integer~$m$ the \emph{preperiod of~$y$}.
\begin{theorem}
  \label{t:ordinary CM points}
  Let~$\zeta$ in~$\Fpalg$ be the $j$-invariant of an ordinary elliptic curve and denote by~$r$ the minimal period of~$\zeta$ under the Frobenius map $z\mapsto z^p$.
  Then there is a unique periodic point~$E_0$ of~$\t$ in~$\bfD(\zeta)$.
  The minimal period of~$E_0$ is~$r$.
  Moreover, $E_0$ is a \CM{} point and, if we denote by~$D_0$ the discriminant of the endomorphism ring of~$E_0$, then the conductor of~$D_0$ is not divisible by~$p$ and the following properties hold.
\begin{enumerate}
\item[$(i)$]
  Given a discriminant~$D$, the set~$\supp(\Lambda_D|_{\bfD(\zeta)})$ is nonempty if and only if for some integer~$m \ge 0$ we have~$D = D_0 p^{2m}$.
  Moreover,
  \begin{displaymath}
    \supp(\Lambda_{D_0}|_{\bfD(\zeta)}) = \{ E_0 \}
  \end{displaymath}
  and for each integer~$m \ge 1$ the set~$\supp(\Lambda_{D_0p^{2m}}|_{\bfD(\zeta)})$ is equal to the set of all the preperiodic points of~$\t$ on~$\bfD(\zeta)$ of preperiod~$m$, and is contained in~$\t^{-m}(\t^m(E_0))$.
  In particular, \CM{} points in~$\Ord$ correspond precisely to the periodic and preperiodic points of~$\t$ on~$\Ord$.
\item[$(ii)$]
  For every disc~$\bfB$ of radius strictly less than~$1$ contained in~$\bfD(\zeta)$ there is a constant~$C > 0$ such that for every discriminant~$D < 0$, we have
  $$ \deg(\Lambda_D|_{\bfB})
  \le
  C.$$
\end{enumerate}
\end{theorem}

\begin{remark}
  \label{r:volcano}
  The natural directed graph associated to the dynamics of~$\t$ on the set of ordinary \CM{} points is a ``$(p + 1)$-volcano'' in the sense of~\cite[Section~2.1]{GorKas1711}.
  This follows from Theorem~\ref{t:ordinary CM points}$(i)$ and the fact that~$\t$ is of degree~$p$ on~$\Ord$ by Theorem~\ref{teo-Deligne}.
  Note in particular that the ``rim'' is the directed subgraph associated to the dynamics of~$\t$ on the set of its periodic points in~$\Ord$.
  Moreover, on the set of preperiodic points of~$\t$ in~$\Ord$, the preperiod corresponds to the function~``$b$'' of~\cite{GorKas1711}.
\end{remark}

To prove Theorem~\ref{t:ordinary CM points}, we describe the dynamics of~$\t$ on~$\Ord$ in Lemma~\ref{periodicpoint} below.
This description is mostly based on the fact that
\begin{equation}
  \label{eq-cong-t}
\t(z)\equiv z^p \mod p \cO_p,
\end{equation}
see Theorem~\ref{teo-Deligne}.
We deduce from general considerations that each residue disc~$\bfD \subseteq \Ord$ contains a unique periodic point~$z_0$ of~$\t$, that this point satisfies~$|\t'(z_0)| < 1$, and that every point in~$\bfD$ is asymptotic to~$z_0$.\footnote{This is somewhat similar to the case of a rational map having good reduction equal to the Frobenius map, see for example~\cite[Sections~3.1 and~4.5]{Riv03c}.}
The fact that no periodic point of~$\t$ in~$\Ord$ is a ramification point is used in a crucial way in the proof of the estimate~\eqref{disttpm} of Proposition~\ref{distord} in Section~\ref{ss:ordinary orbits}.

\begin{lemma}[Dynamics of~$\t$ on~$\Ord$]
  \label{periodicpoint}
  Let~$\ss$ be an ordinary elliptic curve defined over~$\Fpalg$ and let~$r \ge 1$ be the minimal period of~$j(\ss)$ under the Frobenius map.
  Then, $\ss^{\uparrow}$ is the unique elliptic curve in~$\bfD(j(\ss))$ that is periodic for~$\t$.
  The minimal period of~$\ss^{\uparrow}$ for~$\t$ is~$r$ and~$\ss^{\uparrow}$ is also characterized as the unique elliptic curve in~$\bfD(j(\ss)) \cap \Cpunalg$ whose endomorphism ring is an order in an quadratic imaginary extension of~$\Q$ of conductor not divisible by~$p$.
Moreover, if for every integer~$i \ge 0$ we put~$z_i \= \t^i(\ss^{\uparrow})$, then the following properties hold.
\begin{enumerate}
\item[$(i)$]
  For each integer~$i \ge 0$ we have~$0 < |\t'(z_i)|_p < 1$.
\item[$(ii)$]
  There is~$\rho$ in~$(0, 1)$ such that for every integer~$i \ge 0$ and all~$z$ and~$z'$ in~$\bfD(z_i, \rho)$, we have
  \begin{displaymath}
    \deg_{\t}(z) = 1
    \text{ and }
    |\t(z) - \t(z')|_p
    =
    |\t'(z_i)|_p \cdot |z - z'|_p.
  \end{displaymath}
  In particular, $\t$ is injective on~$\bfD(z_i, \rho)$.
\item[$(iii)$]
  For every $c \in (0,1)$ there exists~$\kappa_c$ in~$(0, 1)$ such that for every integer~$i \ge 0$, every~$z$ in~$\bfD(z_i, 1)$ satisfying $|z - z_i|_p\leq c$ and every integer~$m \ge 1$, we have
  \begin{displaymath}
|\t^m(z) - z_{i + m}|_p\leq \kappa_c^m |z - z_i|_p.
\end{displaymath}
\item[$(iv)$]
  For all~$i \ge 0$ and~$z$ in~$\bfD(z_i, 1)$, the sequence
  $$ (|\t^m(z) - z_{i + m})|_p)_{m = 0}^{\infty}$$
  is nonincreasing and converges to~$0$.
\end{enumerate}
\end{lemma}

\begin{proof}
  We start proving~$(i)$.
  Suppose by contradiction that~$z_i$ is a ramification point of~$\t$.
  Without loss of generality, assume that~$i = 0$ and put~$E \= \ss^{\uparrow}$ and~$E^p \= (\ss^{(p)})^{\uparrow}$.
By Proposition~\ref{prop-tpm} with $m=1$ there are distinct subgroups~$C$ and~$C'$ of~$E^p$ of order~$p$ such that
\begin{displaymath}
  E^p/C = E^p/C' = E,
  C \neq H(E^p)
  \text{ and }
  C' \neq H(E^p).
\end{displaymath}
Let~$\psi$ (resp.~$\psi'$) be an isogeny~$E^p \to E$ with kernel~$C$ (resp.~$C'$) and denote by~$\hpsi$ (resp.~$\hpsi'$) its dual isogeny.
Then the kernel of~$\hpsi$ and of~$\hpsi'$ are both equal to~$H(E)$.
It follows that there is~$\sigma$ in~$\Aut(E^p)$ such that~$\sigma \circ \hpsi = \hpsi'$, see, \emph{e.g.}, \cite[Chapter~III, Corollary~4.11]{Sil09}.
Since $\sigma \neq \pm 1$, we have $j(E^p)\in \{0,1728\}$ and therefore $r=1$, $\t(z_0)=z_0$ and $E^p= E$.
In particular, $C$ and~$C'$ are subgroups of~$E$ and~$\psi,\psi' \in \End(E)$.
The kernel of each of the reduced isogenies~$\tpsi$ and~$\tpsi'$ is equal to~$\ss[p](\Fpalg)$, so there is~$\talpha$ in~$\Aut(\ss)$ such that~$\talpha \circ \tpsi = \tpsi'$.
Since the reduction map $\End(E)\to \End(e)$ is an isomorphism by Theorem~\ref{teo-Deuring}$(ii)$, we can find an automorphism $\alpha\in \Aut(E)$ satisfying $\alpha\circ \psi = \psi'$.
This implies that the kernel~$C$ of~$\psi$ is equal to the kernel~$C'$ of~$\psi'$, and we obtain a contradiction.
This completes the proof that~$z_i$ is not a ramification point of~$\t$ and therefore that~$\t'(z_i) \neq 0$.

To prove that~$|\t'(z_i)| < 1$ note that by Theorem~\ref{teo-Deligne}, we can write
  \begin{equation}
    \label{e:6}
    \t(w + z_i) - z_{i + 1}
    =
    \t(w + z_i) - \t(z_i)
  =
  \sum_{n = 1}^{\infty} B_n^{(i)} w^n,    
\end{equation}
where the coefficients~$B_n^{(i)}$ belong to~$\cO_p$ and satisfy~$|B_n^{(i)}|_p \le \frac{1}{p}$ for~$n \neq p$.
Since~$\t'(z_i) = B_1^{(i)}$, this completes the proof of~$(i)$.

To prove the assertions at the beginning of the lemma, for each integer~$i \ge 0$ denote by~$\ss^{(p^i)}$ the image of~$\ss$ by the $i$-th iterate of the Frobenius morphism.
Then by Theorem~\ref{teo-Deuring}$(iii)$ we have
  \begin{displaymath}
    z_i
    =
    \t^i(\ss^{\uparrow})
    =
    (\ss^{(p^i)})^{\uparrow} \in \pi^{-1}(j(\ss)^{p^i}).
  \end{displaymath}
  It follows that~$z_0$ is periodic of minimal period~$r$ for~$\t$.
  To prove uniqueness, note that by~\eqref{e:6} for every integer~$i \ge 0$ and distinct~$z$ and~$z'$ in~$D(z_i, 1)$ we have
  \begin{equation}
    \label{e:7}
    |\t(z) - \t(z')|_p
    <
    |z - z'|_p.
  \end{equation}
  Thus, there can be at most one periodic point of~$\t$ in~$\bfD(z_0, 1)$.
  Finally, combining Theorem~\ref{teo-Deuring}$(ii)$ and Proposition~\ref{redendring} we obtain that~$e^{\uparrow}$ is the unique elliptic curve reducing to~$e$ and whose endomorphism ring is an order of conductor not divisible by~$p$.
  This completes the proof of the assertions at the beginning of the proposition, so it only remains to prove~$(ii)$, $(iii)$ and~$(iv)$.

  To prove~$(ii)$, let~$\rho$ in~$(0, 1)$ be sufficiently small so that for every~$i$ in~$\{0, \ldots, p - 1 \}$, we have
  \begin{displaymath}
    \max \{ |B_n^{(i)}|_p \rho^{n - 1} : n \ge 2 \}
    \le
    |B_1^{(i)}|_p.
  \end{displaymath}
  Then by the ultrametric inequality for every integer~$i \ge 0$ and~$z \in \bfD(z_i, \rho)$ we have~$|\t'(z)|_p = |B_1^{(i)}|_p$, which is different from~$0$ by~$(i)$.
  In particular, $\deg_{\t}(z_i) = 1$.
  Moreover, for~$z'$ in~$\bfD(z_i, \rho)$ we have by the ultrametric inequality
  \begin{displaymath}
    |\t(z) - \t(z')|_p = |B_1^{(i)}|_p |z - z'|_p.
  \end{displaymath}
  This completes the proof of~$(ii)$.

  Item~$(iii)$ is a direct consequence of~\eqref{e:6} with
\begin{displaymath}
  \kappa_c
  \=
  \max \{ |B_n^{(i)}| c^{n - 1} : n \ge 1, i \in \{0, \ldots, p - 1 \} \},
\end{displaymath}
noting that for every integer~$n \ge 1$ and all integers~$i, i' \ge 0$ such that~$i - i'$ is divisible by~$p$, we have~$B_n^{(i')} = B_n^{(i)}$.

To prove item~$(iv)$, note that the fact that the sequence is nonincreasing follows from~\eqref{e:7} and the fact that it converges to~$0$ form~$(iii)$ with~$c = |z - z_i|$.
This completes the proof the lemma.
\end{proof}

\begin{proof}[Proof of Theorem~\ref{t:ordinary CM points}]
  The first assertions are given by Lemma~\ref{periodicpoint}.

  To prove~$(i)$, note that Proposition~\ref{redendring} implies that if a discriminant~$D < 0$ is such that~$\supp(\Lambda_D|_{\bfD(\zeta)})$ is nonempty, then there is an integer~$m \ge 0$ such that~$D = D_0 p^{2m}$.
  On the other hand, Lemma~\ref{periodicpoint} implies~$\supp(\Lambda_{D_0}|_{\bfD(\zeta)}) = \{ E_0 \}$.
  Fix an integer~$m \ge 1$ and note that by Lemma~\ref{periodicpoint} for every integer~$j \ge 1$ the point~$E_j \= \t^j(E_0)$ is the unique periodic point of~$\t$ in~$\bfD(\zeta^{p^j})$.
  So, if~$E$ is a preperiodic point of~$\t$ in~$\bfD(\zeta)$ of preperiod~$m$, then~$\t^m(E) = E_m$.
  This implies that the set of all preperiodic points of~$\t$ in~$\bfD(\zeta)$ is contained in~$\t^{-m}(E_m)$ and is equal to
  \begin{displaymath}
    \t^{-m}(E_m) \setminus \t^{-(m - 1)}(E_{m - 1})
    =
    \t^{- (m - 1)}(\t^{-1}(E_m) \setminus \{ E_{m - 1} \} ).
  \end{displaymath}
  Since the degree of~$\t$ is~$p$ and by Lemma~\ref{periodicpoint}$(i)$ we have~$\t'(E_{m - 1}) \neq 0$, the set~$\t^{-1}(E_m) \setminus \{ E_{m - 1} \}$ is nonempty and equal to~$\supp(\t^*([E_m]) - [E_{m - 1}])$.
  We thus conclude that the set of preperiodic points of~$\t$ in~$\bfD(\zeta)$ of preperiod~$m$ is equal to $\t^{- (m - 1)}(\supp(\t^*([E_m]) - [E_{m - 1}]))$ and it is nonempty.
  Thus, to complete the proof of~$(i)$ it is sufficient to show that the set of preperiodic points of~$\t$ in~$\bfD(\zeta)$ of preperiod~$m$ is equal to~$\supp(\Lambda_{D_0 p^{2m}}|_{\bfD(\zeta)})$.
  Note that by~\eqref{e:p Zhang} and Proposition~\ref{prop-tpm} we have
  $$\supp(\t_*(\Lambda_{D_0}))
  \subseteq
  \supp(T_p(\Lambda_{D_0}))
  =
  \supp(\Lambda_{D_0})\cup \supp(\Lambda_{D_0p^2}).$$
  By Lemma~\ref{periodicpoint} the set $\supp(\Lambda_{D_0})$, hence $\supp(\t_*(\Lambda_{D_0}))$, is formed by periodic points of~$\t$ while points in $\supp(\Lambda_{D_0p^2})$ are not periodic.
  This implies
     \begin{equation}
    \label{e:tfixesLambda}
    \t_*(\Lambda_{D_0})
    =
    \Lambda_{D_0}.
  \end{equation}
  Let~$d$ and~$f_0$ be the fundamental discriminant and conductor of $D_0$, respectively.
  Since~$p$ splits in $\Q(\sqrt{d})$ we deduce that for every integer~$k\geq 0$ we have~$R_d(p^k) = k + 1$.
  By~\eqref{e:general Zhang formula}, Proposition~\ref{prop-tpm} and~\eqref{e:tfixesLambda} we get
  \begin{displaymath}
    \supp ((\t^{\ast})^m(\Lambda_{D_0}))
    =
    \bigcup_{k=0}^m\supp(\Lambda_{D_0p^{2k}}).
  \end{displaymath}
This implies the equality
\begin{equation}\label{eq:12}
  \supp((\t^{\ast})^m(\Lambda_{D_0}))\setminus \supp((\t^*)^{m-1}(\Lambda_{D_0}))
  =
  \supp(\Lambda_{D_0p^{2m}}).
\end{equation}
By Lemma~\ref{periodicpoint} and~\eqref{e:tfixesLambda} the set~$\supp(\Lambda_{D_0})\cap \left(\bfD(\zeta)\cup \bfD(\zeta^p)\cup \cdots \cup \bfD(\zeta^{p^{r-1}})\right)$ equals the set of periodic points of $\t$ in~$\bfD(\zeta)\cup \bfD(\zeta^p)\cup \cdots \cup \bfD(\zeta^{p^{r-1}})$.
By~\eqref{eq:12} we conclude that the set $\supp(\Lambda_{D_0p^{2m}}|_{\bfD(\zeta)})$ equals the set of preperiodic points of~$\t$ in~$\bfD(\zeta)$ of preperiod~$m$. This completes the proof of~$(i)$.

  To prove~$(ii)$, let~$c$ in~$(0, 1)$ be such that $\bfB \subseteq \bfD(z_0,c)$, let~$\rho$ and~$\kappa_c$ be given by Lemma~\ref{periodicpoint} and let~$M \ge 1$ be an integer such that~$c \kappa_c^{rM} < \rho$.
  Let~$D < 0$ be a discriminant and~$z$ in~$\supp(\Lambda_D) \cap \bfB$ be given.
  By~$(i)$ there is an integer~$m \ge 0$ such that~$\t^{rm}(z) = E_0$.
  Assume by contradiction that the least integer~$m$ with this property satisfies~$m > M$.
  Then by Lemma~\ref{periodicpoint} and our choice of~$M$ we have
  $$ |\t^{rM}(z) - E_0|_p
  \le c\kappa_c^{rM}
  <
  \rho. $$
  On the other hand, $\t^{r(m - M)}$ is injective on~$\bfD(z_0, \rho)$ by Lemma~\ref{periodicpoint}$(ii)$ and it maps~$\t^{rM}(z)$ and~$E_0$ to~$E_0$, so~$\t^{rM}(z) = E_0$.
  This contradicts the minimality of~$m$ and proves that for every~$z$ in~$\supp(\Lambda_D) \cap \bfB$ we have~$\t^{rM}(z) = E_0$.
  Equivalently,
  $$ \supp(\Lambda_D|_{\bfB})
  \subseteq
  \bigcup_{i=1}^M \t^{-ir}(E_0). $$
  Since this last set is finite and independent of~$D$, this proves~$(ii)$ and completes the proof of the theorem.
\end{proof}

\section{\CM{} points in the supersingular reduction locus}\label{s:supersingular CM points}

The goal of this section is to prove the following result on the asymptotic distribution of \CM{} points in the supersingular reduction locus. From this result and Theorem~\ref{t:ordinary CM points}$(ii)$, we deduce Theorem~\ref{t:CM points} at the end of this section.

\begin{theorem}
  \label{t:supersingular CM points}
  For every~$\ss$ in~$\tSups$ fix an arbitrary~$\gamma_{\ss}$ in~$\bfD(j(\ss))$ and for~$r$ in~$(0, 1)$, put
  \begin{displaymath}
    \bfB(r)
    \=
    \bigcup_{e \in \tSups} \bfD(\gamma_{\ss}, r).
  \end{displaymath}
  Then the following properties hold.
\begin{enumerate}
\item[$(i)$]
  For every~$r$ in~$(0, 1)$ there exists~$m > 0$ such that for every discriminant~$D < 0$ satisfying~$\ord_p(D) \ge m$, we have~$\deg(\Lambda_{D}|_{\bfB(r)})=0$.
\item[$(ii)$]
  For every $m > 0$ there exists~$r$ in~$(0, 1)$ such that for every $p$-supersingular discriminant $D<0$ satisfying~$\ord_p(D) \le m$, we have~$\supp(\Lambda_{D}) \subseteq \bfB(r)$.
\end{enumerate}
\end{theorem}

We present the proof of Theorem~\ref{t:supersingular CM points} in Section~\ref{ss:proof of CM-sups} below.  In Section~\ref{ss:katz valuation}  we recall the definition of Katz' valuation. For that purpose, we briefly review Katz' theory of algebraic modular forms and the interpretation of the Eisenstein series~$\oE_{p-1}$ as an algebraic modular form over $\Q \cap \Z_{p}$. 
In Section~\ref{ss:katz kite} we use Katz--Lubin's extension of the theory of  canonical subgroups to not too supersingular elliptic curves to give a description of the action of Hecke correspondences on the supersingular locus (Proposition~\ref{ss:katz kite}).
For~$p = 2$ and~$3$, we also rely on certain congruences satisfied by certain Eisenstein series (Proposition~\ref{prop-E1-E2} in Appendix~\ref{s:lift Hasse 2 3}).
This description is used in the proof of Theorem~\ref{t:supersingular CM points}  and also in Section~\ref{ss:supersingular orbits} on Hecke orbits in the supersingular locus.

\subsection{Katz' valuation}
\label{ss:katz valuation}
In this section we define Katz' valuation, which is based on Katz' theory of algebraic modular forms, and give an explicit formula relating it to the $j$-invariant (Proposition~\ref{vnorm}).

For the reader's convenience we start with a short review of Katz' theory of algebraic modular forms. For details see~\cite[Chapter~1]{Kat73}.
Let $k\in \Z$ be an integer and let~$R_0$ be a ring (commutative and with identity).
Denote by $R_0\text{-Alg}$ the category of $R_0$-algebras. 
Given an $R_0$-algebra~$R$, define an elliptic curve~$E$ over~$R$ as a proper, smooth morphism of schemes~$E\to \mathrm{Spec}(R)$, whose geometric fibres are connected curves of genus one, together with a section~$\mathrm{Spec}(R)\to E$, and denote by~$\Omega^1_{E/R}$ the invertible sheaf of differential forms of degree~$1$ of~$E$ over~$R$. By replacing~$\mathrm{Spec}(R)$ by an appropriate affine subset we can assume that~$\Omega^1_{E/R}$ admits a nowhere vanishing global section. In this paper we assume, for simplicity, that this is always the case and denote by~$\Omega^1_{E/R}(E)'$ the (non-empty) set of nowhere vanishing global sections of~$\Omega^1_{E/R}$.
An algebraic modular form~$F$ of weight~$k$ and level one over~$R_0$ is a family of maps 
\begin{displaymath}
  F_R \colon \{(E,\omega): E\text{ elliptic curve over }R,\, \omega \in \Omega^1_{E/R} (E)' \} \to R
  \quad
  (R\in R_0\text{-Alg}),
\end{displaymath}
satisfying the following properties:
\begin{enumerate}
\item[$(i)$]
  $F_R{(E,\omega)}$ depends only on the isomorphism class of the pair $(E,\omega)$.
  More precisely, for every isomorphism of elliptic curves $\varphi \colon E\to E'$ over $R$,  we have $F_R{(E',\varphi_*\omega)}=F_R{(E,\omega)}$.
  Here, $\varphi_*\omega$ denotes the push-forward of $\omega$ by~$\varphi$.
\item[$(ii)$]
  $F_R(E,\lambda \omega)=\lambda^{-k}\,F_R(E,\omega)$ for every $\lambda \in R ^{\times}$.
\item[$(iii)$]
  $F_R$ is compatible  with base change.
  Namely, for every $R_0$-algebra morphism $g \colon R\to R'$, for the base change $(E,\omega)_{R'}$ of $(E,\omega)$ to~$R'$ by~$g$ we have $F_{R'}((E,\omega)_{R'})=g(F_R(E,\omega))$.
\end{enumerate}
Taking into account property $(iii)$, from now on we simply write $F$ instead of $F_R$. Moreover,   let $R_1$ be an $R_0$-algebra. Then, property $(iii)$ ensures that $F$ induces an algebraic modular form $F_1$ over $R_1$. We say that $F_1$ is the \emph{base change} of $F$ to $R_1$. We also say that $F$ is a \emph{lifting} of $F_1$ to $R$.

  Let ~$q$ be a formal variable and denote by~$\Tate(q)$ the \emph{Tate curve}, which is an elliptic curve over the field of fractions fractions~$\Z(\!( q)\!)$ of the ring of formal power series~$\Z\llbracket q\rrbracket$, see~\cite[Appendix~1]{Kat73}.
  The $j$-invariant of $\Tate(q)$ has the form 
\begin{equation}\label{jtateq}
j\left(\Tate(q)\right)=\frac{1}{q}+744+\sum_{n=1}^\infty c_nq^n, \quad c_n\in \Z.
\end{equation}

The \emph{$q$-expansion} of an algebraic modular form $F$  over $R_0$  as above is defined as the element  $F(q)\in \Z(\!( q)\!) \otimes_{\Z}R_0$ obtained by evaluating $F$ at the pair $(\Tate(q),\omega_{\can})$  consisting of the Tate curve together with its canonical differential $\omega_{\can}$, both considered  as defined over $\Z(\!( q)\!)\otimes_{\Z}R_0$. Moreover, $F$ is said to be \emph{holomorphic at infinity} if $F(q) \in \Z\llbracket q\rrbracket\otimes_\Z R_0$.

Now, we state a version of the \emph{$q$-expansion principle}, which is a particular case of ~\cite[Corollary~1.9.1]{Kat73}.

\begin{theorem}\label{qep}
Let $R_0$ be a ring and let $K \supseteq R_0$ be a  $R_0$-algebra. Let $k \in \Z$ be an integer and let  $F$ be an algebraic modular form over $K$ of weight $k$, level one and holomorphic at infinity. Assume that $F(q) \in \Z(\!( q)\!) \otimes_{\Z}R_0$. Then, $F$ is the base change of a unique algebraic modular form over $R_0$ of weight $k$. 
\end{theorem}

  There is a natural link between the previous theory  and the classical theory of modular forms. We refer to \cite[Section~A1.1]{Kat73} for details.
  For each classical holomorphic modular form of weight~$k$ and level one~$f \colon  \H \rightarrow \C$, there exists a unique algebraic modular form~$F$ over~$\C$ associated to~$f$ that is holomorphic at infinity.
The Fourier expansion at infinity of~$f$ and the $q$-expansion of~$F$ are related by
\begin{displaymath}
  f(\tau)
  =
  \sum_{n=0}^\infty a_ne^{2\pi i n\tau} \textrm{ if and only if } F(q)=\sum_{n=0}^{\infty}a_nq^n.
\end{displaymath}
 
  For an even integer $k\geq 4$, let~$\oE_{k}$ be the \emph{normalized Eisenstein series}
  \begin{displaymath}
    \oE_{k}(\tau)
  =
  1-\frac{2k}{B_k}\sum_{n=1}^{\infty}\sigma_{k-1}(n)e^{2\pi i n\tau},
  \quad \tau \in \H.
  \end{displaymath}
  Here, the symbol $B_k$ denotes the $k$-th Bernoulli number and $\sigma_{k-1}(n)\=\sum_{d \mid n,d>0}d^{k-1}$. The complex function $\oE_k$ is a classical holomorphic modular form of weight~$k$ and level one. Then, this function induces an algebraic modular form over $\C$, that we also denote by $\oE_k$, having the $q$-expansion with rational coefficients
 \begin{equation}\label{eq-q-exp-Ek}
  \oE_{k}(q)
  =
  1-\frac{2k}{B_k}\sum_{n=1}^{\infty}\sigma_{k-1}(n)q^n.
\end{equation}
When $p\geq 5$ and $k=p-1$, the von Staudt--Clausen Theorem ensures that   $\ord_p\big((2k)B_k^{-1}\big)=1$.
In particular,  the coefficients of the Fourier expansion of $\oE_{p-1}$ lie in $\Zp\=\Q \cap \Z_p$.
Hence, by Theorem~\ref{qep}, we can consider~$\oE_{p-1}$ as an algebraic modular form of weight $p-1$ over~$\Zp$. On the other hand, the same reasoning and a direct examination of the Fourier expansions of $\oE_4$ and $\oE_6$ allow us to consider these Eisenstein series as  algebraic modular forms of weight four and  six over $\Z$.

For~$E$ in~$\Sups$, that we regard as defined over $\cO_p$, choose $\omega$ in~$\Omega_{E/\cO_p}^1(E)'$ and define \emph{Katz' valuation}
$$\kval(E)
\=
\begin{cases}
  \ord_p(\oE_{p-1}(E, \omega ))
  & \text{if } p\geq 5;
  \\
  \frac{1}{3} \cdot \ord_3(\oE_{6}(E, \omega))
  & \text{if }p=3;
  \\
  \frac{1}{4} \cdot \ord_2(\oE_{4}(E, \omega))
  & \text{if }p=2.  
\end{cases} $$
Since for every~$\lambda$ in~$\cO_p^{\times}$ we have $\oE_{k}(E,\lambda \omega )=\lambda^{-k}\oE_{k}(E, \omega )$, this definition does not depend on the particular choice of $\omega$.
The above definition is motivated by the following considerations. 
The Hasse invariant $A_{p-1}$ is the unique algebraic modular form of weight $p-1$ over $\F_p$ with $q$-expansion $A_{p-1}(q)=1$, see~\cite[Chapter~2]{Kat73}.
When $p \geq 5$, the base change to $\F_p$ of the form~$\oE_{p-1}$  equals~$A_{p-1}$.  On the other hand, when~$p$ equals~$2$ or~$3$ it is not possible to lift $A_{p-1}$ to an algebraic modular form of level one, holomorphic at infinity, over~$\Zp$.
However, the base change of~$\oE_4$ (resp.~$\oE_6$) to $\F_2$ (resp.~to $\F_3$) is~$A_1^4$  (resp.~$A_2^3$).
See Appendix~\ref{s:lift Hasse 2 3} for details.

Since the Hasse invariant vanishes at supersingular elliptic curves, for every~$E$ in~$\Sups$ we have that $0<\kval(E)\leq \infty$.
An elliptic curve~$E$ in~$\Sups$ is \emph{not too supersingular} if $\kval(E)< \frac{p}{p+1}$, and it is \emph{too supersingular} otherwise.

The following result gives an explicit relation between~$\kval(E)$ and~$j(E)$.
For~$\ss$ in~$\tSups$, we use the number~$\delta_{\ss}$ defined by~\eqref{eq:8} in Section~\ref{ss:canonical branch}.

\begin{proposition}\label{vnorm}
  For each~$\ss$ in~$\tSups$, denote by~$\fj_e$ the~$j$-invariant of the unique zero of~$\oE_{p-1}$ (resp.~$\oE_4,\oE_6$) in~$\bfD(\ss)$ if~$p\geq 5$ (resp.~$p=2,3$).
  Then,  for every~$E$ in~$\Sups$ we have
  \begin{displaymath}
    \kval(E)
=
\sum_{\ss \in \tSups} \frac{1}{\delta_{\ss}} \ord_p(j(E) - \fj_{\ss}).
  \end{displaymath}
Moreover, if $p\geq 5$ and $\fj_{\ss} \equiv 0$ (resp.~$\fj_{\ss}\equiv 1728$) $\mod \cM_p$, then $\fj_{\ss}=0$ (resp.~$\fj_e=1728$).
In the case $p=2$ (resp.~$p=3$), $\tSups$ has a unique element~$\ss$ and $\fj_{\ss} = 0$ (resp.~$\fj_{\ss} = 1728$).
\end{proposition}

  It follows from the proof of this proposition that for every~$\ss$ in~$\tSups$ the number~$\fj_{\ss}$ is algebraic over~$\Q$ and is in~$\Cpun$.
  We note that in the case~$\fj_{\ss} \not\equiv 0, 1728 \mod \cM_p$, the elliptic curve class whose $j$-invariant is~$\fj_{\ss}$ is not \CM{},\footnote{In fact, $\fj_{\ss}$ need not be an algebraic integer: For~$p = 13$ (resp. $17, 19, 23$) there is a unique~$\ss$ in~$\tSups$ whose~$j$-invariant is different from~$0$ and~$1728$, and we have $\fj_{\ss} = 2^7 \cdot 3^3 \cdot 5^3 / 691$ (resp. $2^{10} \cdot 3^3 \cdot 5^3 / 3617$, $2^8 \cdot 3^3 \cdot 5^3 \cdot 11 / (7 \cdot 79^2)$, $2^8 \cdot 3^3 \cdot 5^3 \cdot 41 / (131 \cdot 593)$).} but it is ``fake \CM{}'' in the sense of~\cite{ColMcM06}, see Remark~\ref{r:quasi-canonic} below.

\begin{proof}[Proof of Proposition~\ref{vnorm}]
  Assume $p\geq 5$, so $p-1 \not\equiv 2, 8 \mod 12$.
  We can thus write $p-1$ uniquely in the form $p-1=12m+4\delta+6\varepsilon$ with $m\geq 0 $ integer and $ \delta,\varepsilon \in \{0,1\}$.
  The modular discriminant
  \begin{displaymath}
    \Delta(\tau)=e^{2\pi i\tau}\prod_{n=1}^{\infty}(1-e^{2\pi in\tau})^{24},
    \quad \tau \in \H.
  \end{displaymath}
is a classical holomorphic modular form of weight~$12$ and level one, see, \emph{e.g.}, \cite[Sections~1.1 and~1.2]{DiaShu05}.
  The infinite product above shows that the Fourier coefficients of $\Delta$ are rational integers. Hence, Theorem~\ref{qep} ensures that $\Delta$ can be considered as an algebraic modular form over $\Z$. 
  At the level of classical modular forms, we have the identity 
\begin{equation*}
\oE_{p-1}=\Delta^m\oE_4^{\delta}\oE_6^{\varepsilon}P\left(j\right),
\end{equation*}
where~$P(X)$ is a monic polynomial over~$\Zp$ of degree~$m$ such that $P_{\sups}(X)\=X^{\delta}(X-1728)^{\varepsilon}P(X)$ reduces modulo~$p$ to the supersingular polynomial, \emph{i.e.}, the monic separable polynomial over~$\F_{p}$ whose roots are the $j$-invariants of the supersingular elliptic curves over $\Fpalg$, see, \emph{e.g.}, \cite[Theorem~1]{KanZag98}.
Using the classical identities $\oE_4^3=\Delta j$ and $\oE_6^2=\Delta (j-1728)$ we get
\begin{equation*}
  \oE^{12}_{p-1}
  =
  \Delta^{p-1} j^{4\delta}(j-1728)^{6\varepsilon} P(j)^{12}.
\end{equation*}
Theorem~\ref{qep} ensures that the above identity also holds at the level of algebraic modular forms over $\Zp$.  Write
$$ P_{\sups}(X)
=
\prod_{\ss \in \tSups} (X-\fj_{\ss}), $$
where $\fj_{\ss} \in \bfD\big(j(\ss)\big)$ for each~$\ss \in \tSups$.
Now, for every pair $(E,\omega)$ defined over~$\cO_p$ and having good reduction we have $\Delta(E,\omega)\in \cO_p^{\times}$, hence
\begin{equation*}
  |\oE_{p-1}(E,\omega)|_p^{12}
  =
  |j(E)|_p^{4\delta} |j(E)-1728|_p^{6\varepsilon} \prod_{\substack{\ss \in \tSups \\ \fj_{\ss} \not \equiv 0,1728}} |j(E)-\fj_{\ss}|_p^{12}.
\end{equation*}
Since~$p\geq 5$, we have that $j=0$ (resp. $j=1728$) is supersingular at $p$ if and only if $p \equiv 2 \mod 3$ (resp. $p\equiv 3 \mod 4$) \cite[Chapter~V, Examples~4.4, 4.5]{Sil09}. This implies the result when $p\geq 5$. The cases $p=2$ and $3$ follow similarly from the formulas
$$|\oE_4(E,\omega)|_2^3
=
|j(E)|_2
\text{ and }
|\oE_6(E,\omega)|_3^2
=
|j - 1728|_3,$$
respectively. This completes the proof of the proposition. 
\end{proof}

\begin{remark}
  \label{r:quasi-canonic}
      Let~$\ss$ in~$\Sups$ be such that $\fj_{\ss}\not \equiv 0,1728\mod \cM_p$, and let~$E_{\ss}$ be the elliptic curve class in~$\Ell(\C_p)$ such that~$j(E_{\ss}) = \fj_{\ss}$.
Then~$E_{\ss}$ is not~\CM{}, but it is ``fake \CM{}'' in the sense of~\cite{ColMcM06}.
In particular, $\fj_{\ss}$ is not a singular modulus over~$\C_p$.
To show that~$E_{\ss}$ is not~\CM{}, choose a field isomorphism $\C_p\simeq \C$ and~$\tau_{\ss}$ in~$\H$ such that~$E_{\ss}(\C) \simeq \C / (\Z + \tau_{\ss} \Z)$.
It is sufficient to show that~$\tau_{\ss}$ is transcendental over~$\Q$, see, \emph{e.g.}, \cite[Chapter~1, Section~5]{Lan87}.
The complex number~$\tau_\ss$ must be a zero of the holomorphic function $\tau\mapsto \oE_{p-1}(\tau)$.
Since~$j(\tau_{\ss}) = \fj_{\ss}$ is different from~$0$ and~$1728$, it follows that~$\tau_{\ss}$ is not equivalent to $\rho=\frac{1+\sqrt{-3}}{2}$ or $i=\sqrt{-1}$ under the action of the modular group~$\SL_2(\Z)$ by M{\"o}bius transformations on~$\H$.
Then~\cite[Theorem~1]{Koh03} implies that~$\tau_\ss$ is transcendental over~$\Q$.

    To see that~$E_{\ss}$ is fake \CM{}, note first that, since the reduction modulo~$p$ of~$P_{\sups}(X)$ is separable and splits completely over~$\F_{p^2}$, by Hensel's lemma all roots of~$P_{\sups}(X)$ are in the ring of integers~$\cO$ of the unramified quadratic extension of~$\Q_p$.
    As~$\fj_{\ss}$ is a root of~$P_{\sups}(X)$, this implies that~$E_{\ss}$ is defined over~$\cO$.
    Let~$[p]_{\ss}$ and~$\phi$ be the multiplication by~$p$ and the $p^2$-power Frobenius endomorphism on the supersingular curve~$\ss$, respectively.
    Then there exists~$\sigma$ in $\Aut(\ss)$ satisfying~$\sigma\circ [p]_{\ss} = \phi$, see~\cite[Chapter~II, Corollary~2.12]{Sil09}.
    Since $j(\ss) = \pi(\fj_{\ss})$ is different from~$0$ and~$1728$, we have $\sigma = \pm 1$ and~$\pm [p]_{\ss}=\phi$.
    Choose~$\pi_0=\pm p$ as a uniformizer of~$\cO$.
    The multiplication by~$\pi_0$ map on the formal group $\cF_{E_{\ss}}$ of~$E_{\ss}$ defines an endomorphism~$f(X)$ of~$\cF_{E_{\ss}}$, satisfying
    \begin{displaymath}
      f(X) \equiv \pi_0 X \mod X^2
      \text{ and }
      f(X) \equiv X^{p^2} \mod~\pi_0.
    \end{displaymath}
    It follows that~$\cF_{E_{\ss}}$ is a Lubin--Tate formal group over $\cO$, see~\cite[Section~8]{Haz78}, and compare with~\cite[Remark~3.4]{ColMcM06}.
    In particular $\End(\cF_{E_{\ss}})\simeq \cO$ and therefore~$E_{\ss}$ is fake \CM{}, see \cite[Theorem~8.1.5 and Proposition~23.2.6]{Haz78}.
   \end{remark}

\subsection{Katz' kite}
\label{ss:katz kite}
The goal of this section is to give the following description of the action of Hecke correspondences on the supersingular locus.

\begin{proposition}
  \label{p:katz kite}
  Let~$\kproj \colon \Sups \to \left[0, \frac{p}{p + 1} \right]$ be the map defined by
$$ \kproj
\=
\min\left\{ \kval, \frac{p}{p+1} \right\}. $$
Moreover, denote by~$\uptau_0$ the identity on~$\Div \left( \left[0, \frac{p}{p + 1} \right] \right)$, let~$\uptau_1$ be the piecewise-affine correspondence on~$\left[ 0, \frac{p}{p + 1} \right]$ defined by
    \begin{displaymath}
      \uptau_1(x)
      \=
      \begin{cases}
        [px] + p[\frac{x}{p}]
        & \text{if } x \in \left[0, \frac{1}{p + 1} \right];
        \\
        [1 - x] + p[\frac{x}{p}]
        & \text{if } x \in \left] \frac{1}{p + 1}, \frac{p}{p + 1} \right],
      \end{cases}
    \end{displaymath}
and for each integer~$m \ge 2$ define the correspondence~$\uptau_m$ on~$\left[ 0, \frac{p}{p + 1} \right]$ recursively, by
    \begin{displaymath}
      \uptau_m
      \=
      \uptau_1 \circ \uptau_{m - 1} - p \uptau_{m - 2}.
    \end{displaymath}
Then for every integer~$m \ge 0$ and every integer~$n_0 \ge 1$ not divisible by~$p$, we have
  \begin{displaymath}
    (\kproj)_* \circ T_{p^m n_0}|_{\Sups}
    =
    \sigma_1(n_0) \cdot \uptau_m \circ (\kproj)_*.
  \end{displaymath}
\end{proposition}

  See Figure~\ref{f:katz kite} for the graph of the correspondence~$\uptau_1$ and Lemma~\ref{l:katz hyper kite} in Section~\ref{ss:supersingular orbits} for a formula of~$\uptau_m$ for every~$m \ge 0$.

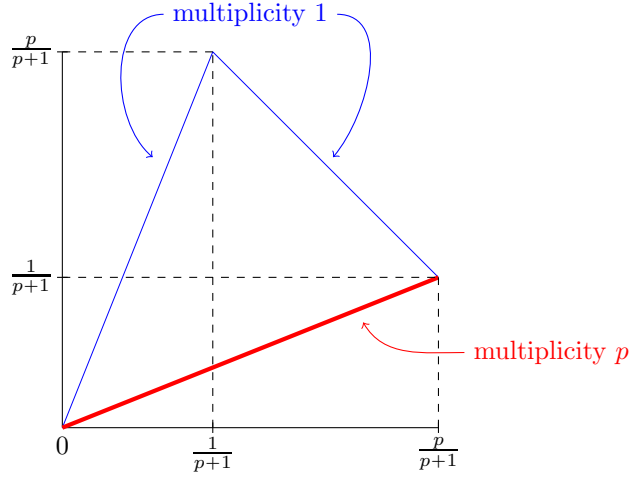
\begin{figure}[htp]
\centering
  \begin{tikzpicture}

\draw[-] (0,0) -- (5,0) node[anchor=north] {$\frac{p}{p+1}$};
\draw (2,2pt) -- (2,-2pt)
      (5,2pt) -- (5,-2pt);

\draw	(0,0) node[anchor=north] {0}
		(2,0) node[anchor=north] {$\frac{1}{p+1}$}
		(5,0) node[anchor=north] {};

\draw[-] (0,0) -- (0,5) node[anchor=east] {$\frac{p}{p+1}$};

\draw (-2pt,2) -- (2pt,2)
      (-2pt,5) -- (2pt,5);

\draw	(0,5) node[anchor=east] {}
        (0,2) node[anchor=east] {$\frac{1}{p+1}$};

\draw[color=blue,very thin] (0,0) -- (2,5) -- (5,2);
\draw[color=red,ultra thick] (0,0) -- (5,2);


\draw[dashed] (2,0) -- (2,5)
              (5,0) -- (5,2)
              (0,2) -- (5,2)
              (0,5) -- (2,5);

\node[color=blue](m1) at (2.5,5.5) {multiplicity 1};

\draw[color=blue,->] (m1.east) to[out=0,in=50] (3.6,3.6);
\draw[color=blue,->] (m1.west) to[out=180,in=140] (1.2,3.6);

\node[color=red](m2) at (6.5,1) {multiplicity $p$};

\draw[color=red,->] (m2.west)  to[out=-180,in=-60] (4,1.4);

\end{tikzpicture}
\caption{Graph of the correspondence~$\uptau_1$ representing the action of~$T_p$ in terms of the projection~$\kproj$.}
\label{f:katz kite}
\end{figure}

The proof of Proposition~\ref{p:katz kite} is given after a couple of lemmas.
The following is a reformulation, in our setting, of a theorem of Katz--Lubin on the existence of canonical subgroups for elliptic curves that are not too supersingular, see~\cite[Theorems~3.1 and~3.10.7]{Kat73} and also~\cite[Theorem~3.3]{Buz03}.
\begin{lemma}
  \label{l:sups canonical subgroup}
  For every elliptic curve~$E$ in~$\Sups$ that is not too supersingular there is a unique subgroup~$H(E)$ of~$E$ of order~$p$ satisfying
  \begin{equation}
    \label{e:sups canonical subgroup A}
    \kproj(E/H(E))
    =
    \begin{cases}
      p\kval(E)
      & \text{if $\kval(E) \in \left] 0, \frac{1}{p + 1} \right]$};
        \\
        1 - \kval(E)
        & \text{if $\kval(E) \in \left] \frac{1}{p + 1}, \frac{p}{p + 1} \right[$}.
      \end{cases}
    \end{equation}
    Furthermore, $H(E)$ is also uniquely characterized by the property that for every subgroup~$C$ of~$E$ of order~$p$ that is different from~$H(E)$, we have
  \begin{equation}
    \label{e:sups canonical subgroup B}
    \kval(E/C) = p^{-1} \kval(E).
  \end{equation}
In addition, the map
  \begin{displaymath}
\begin{array}{rccl}
  \t \colon & \left\{ E \in \Sups : \kval(E) < \frac{p}{p + 1} \right\}
  & \to & \Sups
  \\
  & E & \mapsto & \t(E) \= E/ H(E)
\end{array}  
\end{displaymath}
satisfies the following properties.
\begin{enumerate}
\item[$(i)$]
  Let~$E$ be in~$\Sups$ and let~$C$ be a subgroup of~$E$ of order~$p$.
  In the case~$\kval(E) < \frac{p}{p + 1}$, assume in addition that~$C \neq H(E)$.
  Then
  \begin{displaymath}
    \kval(E/C) = p^{-1}\kproj(E)
    \text{ and } \t(E/C) = E.
  \end{displaymath}
\item[$(ii)$]
  For~$E$ in~$\Sups$ satisfying~$\frac{1}{p + 1} < \kval(E) < \frac{p}{p + 1}$, we have~$\t^2(E) = E$.
\end{enumerate}
\end{lemma}

\begin{proof}
  For~$E$ in~$\Sups$ that is not too supersingular, note that the uniqueness statements about~$H(E)$ follow from the fact that~\eqref{e:sups canonical subgroup A} and~\eqref{e:sups canonical subgroup B} imply that~$H(E)$ is the unique subgroup~$C$ of~$E$ of order~$p$ satisfying~$\kval(E/C) \neq p^{-1} \kval(E)$.

  Assume $p\geq 5$ and let~$E$ be an elliptic curve in~$\Sups$ that is not too supersingular, so that~$\kval(E) < \frac{p}{p + 1}$.
  Let~$\omega$ be a differential form in $\Omega^1_{E/\cO_p}(E)'$ and put $r_E\=\oE_{p-1}(E,\omega)\in \cO_p$.
  Since~$\Cpunalg$ and~$\C_p$ have the same valuation group we can find~$r\in \Cpunalg$ satisfying $\ord_p(r)=\ord_p(r_E)$.
  Then~$r$ lies in the ring of integers~$R_0$ of some finite extension of~$\Cpun$, and~$R_0$ is a complete discrete valuation ring of residue characteristic $p$ and generic characteristic zero. The triple $(E,\omega,r\, r_E^{-1})$ defines a $r$-situation in the sense of~\cite[Theorem~3.1]{Kat73} (see also~\cite[Section~2.2]{Kat73}) and therefore there is a canonical subgroup~$H(E)$ of~$E$ of order~$p$.
  Then~\cite[Theorem~3.10.7(2, 3)]{Kat73} implies~\eqref{e:sups canonical subgroup A} and~$(ii)$, see also the proof of~\cite[Theorem~3.3$(iii)$]{Buz03}, and~\eqref{e:sups canonical subgroup B} and~$(i)$ are given by~\cite[Theorem~3.10.7(5)]{Kat73}.
  Finally, note that for~$E$ in~$\Sups$ satisfying~$\kval(E) \ge \frac{p}{p + 1}$, the assertion~$(i)$ follows from~\cite[Theorem~3.10.7(4)]{Kat73}.
  This completes the proof of the proposition in the case~$p \ge 5$.

It remains to prove the proposition in the cases~$p = 2$ and~$p = 3$.
We only give the proof in the case~$p=2$, the case~$p = 3$ being analogous.
Let~$\oE_1$ be an algebraic modular form of weight one and level~$n_1$, with $3 \leq n_1 \leq 11$ odd, holomorphic at infinity and defined over~$\Z [1/n_1]$ whose reduction modulo~$2$ is~$A_1$, see Appendix~\ref{s:lift Hasse 2 3} for details on level structures.
Let~$E$ in~$\Supstwo$ be an elliptic curve that is not too supersingular, let~$\omega$ be a differential form in~$\Omega^1_{E/\cO_2}(E)'$ and~$\alpha_{n_1}$ a level~$n_1$ structure on~$E$ over~$\cO_2$.
  By Proposition~\ref{prop-E1-E2} and our hypothesis~$\kvaltwo(E)<\frac{2}{3}$, we have
  $$ \ord_2(\oE_1(E,\omega,\alpha_{n_1}))
  =
  \kvaltwo(E)
  <
  \frac{2}{3}. $$
  Then, \cite[Theorem~3.1]{Kat73} gives the existence of~$H(E)$ which might depend on the choice of~$\alpha_{n_1}$.
  The fact that~$H(E)$ depends only on~$E$ follows from the characterization in~\cite[Theorem~3.10.7(1)]{Kat73} of the canonical subgroup as the subgroup of order~$2$ containing the unique point corresponding to the solution with valuation $1-\kvaltwo(E)$ of the equation~$[2](X)=0$ in the formal group of~$E$ (here~$[2]$ denotes the multiplication by~$2$ map and~$X$ is a certain normalized parameter for the formal group).
Then~\eqref{e:sups canonical subgroup A}, \eqref{e:sups canonical subgroup B}, $(i)$ and $(ii)$ follow from~\cite[Theorem~3.10.7]{Kat73} as in the case~$p \ge 5$ above.
This completes the proof of the lemma.
\end{proof}

\begin{lemma}
  \label{invarv}
Let~$E$ in~$\Sups$ be such that 
\begin{displaymath}
  \kval(E)
<
\begin{cases}
  1
  & \text{if } p\geq 5;
  \\
  \frac{2p - 1}{2p}
  & \text{if $p = 2$ or~$3$}.
\end{cases}
\end{displaymath}
Then for every subgroup~$C$ of~$E$ of order not divisible by~$p$, we have $\kval(E/C)=\kval(E)$.
\end{lemma}
\begin{proof}
  For~$E_0$ in~$\Ell(\C_p)$ and~$\zeta$ in~$\Z_p$, denote by~$[\zeta]_{E_0}$ the multiplication by~$\zeta$ map in the formal group of~$E_0$.

  Put~$E' \= E/C$ and denote by~$\phi \colon E \to E'$ an isogeny with kernel~$C$.
  Let~$X$ (resp. $Y$) be a parameter of the formal group of~$E$ (resp. $E'$), such that for any $(p-1)$-th root of unity~$\zeta\in \Z_p$ we have $[\zeta]_E (X)=\zeta X$ (resp. $[\zeta]_{E'} (Y)=\zeta Y$), see~\cite[Lemma~3.6.2(2)]{Kat73}.
  Let~$\omega$ be a differential form in~$\Omega^1_{E/\cO_p}(E)'$ whose expansion in the parameter~$X$ is of the form
  $$ \omega
  =
  \left(1+\sum_{n= 1}^{\infty}a_nX^n \right)dX, $$
  where~$a_n\in \cO_p$ for all~$n \ge 1$.
  Then, by~\cite[Proposition~3.6.6]{Kat73} we have
  $$ [p]_E(X)
  =
  pX + aX^p + \sum_{m\geq 2} c_m X^{m(p-1)+1}, $$
where $c_m \in \cO_p$ for all $m\geq 2$ and~$a \in \cO_p$ satisfies
\begin{equation}
  \label{e:p-th coefficient}
  a\equiv A_{p-1}( (E,\omega)_{\cO_p/p\cO_p}) \mod p\cO_p,
\end{equation}
where $(E,\omega)_{\cO_p/p\cO_p}$ denotes the base change of~$(E,\omega)$ to $\cO_p/p\cO_p$.
Similarly,
$$[p]_{E'}(Y) = pY + a'Y^p + \sum_{m\geq 2} c_m' Y^{m(p-1)+1},$$
where $c_m' \in \cO_p$ for all $m\geq 2$ and~$a' \in \cO_p$ satisfies, for some differential form~$\omega'$ of~$\Omega^1_{E'/\cO_p}(E)'$,
\begin{equation}
  \label{e:p-th coefficient'}
  a' \equiv A_{p-1}( (E',\omega')_{\cO_p/p\cO_p}) \mod p\cO_p.
\end{equation}
Since the order of~$\Ker(\phi) = C$ is not divisible by~$p$, the isogeny~$\phi$ induces an isomorphism of formal groups of the form
$$\phi(X)=\sum_{n=1}^{\infty}t_nX^n, $$
where~$t_n\in \cO_p$ for all~$n \ge 1$.
Since $\phi(X)$ is invertible, we must have $t_1\in \cO_p^{\times}$. By the identity $[p]_{E'}\circ \phi=\phi \circ [p]_E$ we get
\begin{multline*}
  p(t_1X+t_2X^2+t_3X^3+\ldots ) + a' (t_1X+t_2X^2+t_3X^3+\ldots )^p+\ldots
  \\ =
  t_1(pX+aX^p+\ldots)+t_2(pX+aX^p+\ldots)^2+\ldots
\end{multline*}
Comparing the coefficients of~$X^p$, we get
$$pt_p + a't_1^p
=
t_1a + t_pp^p.$$
Using that $t_1\in \cO_p^{\times}$ we obtain
\begin{equation}
  \label{e:equality of p-th}
  \ord_p(a')
=
\ord_p(a't_1^{p - 1})
=
\ord_p(a + t_1^{-1} t_p(p^p-p)).
\end{equation}

In the case~$p \ge 5$, \eqref{e:p-th coefficient} implies~$\ord_p(a-\oE_{p-1}(E,\omega))\geq 1$, so by our hypothesis~$\kval(E) < 1$ we have~$\ord_p(a) = \kval(E) < 1$.
Combined with~\eqref{e:equality of p-th}, this implies~$\ord_p(a') = \ord_p(a) = \kval(E) < 1$.
Finally, by~\eqref{e:p-th coefficient'} we have~$\ord_p(a' - \oE_{p-1}(E', \omega'))\geq 1$, so~$\kval(E') = \ord_p(a') = \kval(E)$.
This proves the lemma in the case~$p \ge 5$.
For the case $p = 2$ or~$3$, \eqref{e:p-th coefficient}, \eqref{e:p-th coefficient'}, \eqref{e:equality of p-th}, our hypothesis~$\kval(E) < \frac{2p - 1}{2p}$ and Proposition~\ref{prop-E1-E2} imply in a similar way
\begin{displaymath}
  \ord_p(a) =\kval(E) < \frac{2p - 1}{2p},
  \ord_p(a') = \ord_p(a)
  \text{ and }
  \kval(E') = \ord_p(a').
\end{displaymath}
This completes the proof of the lemma.
\end{proof}

\begin{proof}[Proof of Proposition~\ref{p:katz kite}]
  By the multiplicative property of Hecke correspondences~\eqref{eq-multiplicativity-Hecke-operators} and Lemma~\ref{invarv}, it is sufficient to consider the case~$n_0 = 1$.
  Moreover, in view of~\eqref{eq-Hecke-Tpr} and the recursive definition of~$\uptau_m$ for~$m \ge 2$, it is sufficient to consider the case~$m = 1$.
  For~$E$ in~$\Sups$ satisfying~$\kproj(E) < \frac{p}{p + 1}$, this is given by~\eqref{e:sups canonical subgroup A} and~\eqref{e:sups canonical subgroup B} in Lemma~\ref{l:sups canonical subgroup}, together with the fact that~$\deg(T_p(E)) = p + 1$.
  Finally, for~$E$ in~$\Sups$ satisfying~$\kproj(E) = \frac{p}{p + 1}$ the desired statement follows from Lemma~\ref{l:sups canonical subgroup}$(i)$.
  This completes the proof of the proposition.
\end{proof}

\subsection{Proof of Theorem~\ref{t:supersingular CM points}}
\label{ss:proof of CM-sups}
The proof of Theorem~\ref{t:supersingular CM points} is below, after a couple of lemmas.

\begin{lemma}\label{lemma-CM-vp}
  Let $D<0$ be a discriminant and let~$E$ and~$E'$ be in~$\supp(\Lambda_D)$.
  Then, for every integer $m\geq 1$ there exists an isogeny $E\to E'$ of degree coprime to~$m$.
\end{lemma}
\begin{proof}
  Denote by~$d$ and~$f$ the fundamental discriminant and conductor of~$D$, respectively, and fix a field isomorphism $\C_p\simeq \C$.
  Since~$E$ and~$E'$ are \CM{} with ring of endomorphisms isomorphic to~$\cO_{d, f}$, we can find proper fractional $\cO_{d,f}$-ideals~$\mathfrak{a}$ and $\mathfrak{a}'$ in~$\Q(\sqrt{D})$ for which we have the complex uniformizations $E(\C) \simeq \C/\mathfrak{a}$ and $E'(\C) \simeq \C/\mathfrak{a}'$.
  Then there is a natural identification
  \begin{displaymath}
    \iota \colon \Hom(E,E')
    \to
    \mathfrak{a}' \mathfrak{a}^{-1}
    =
    \{ \lambda \in \C : \lambda \mathfrak{a} \subseteq \mathfrak{a}' \}.
  \end{displaymath}
  Without loss of generality, assume~$\mathfrak{a}' \subset \mathfrak{a}$, and choose $\Z$-generators $\alpha$ and~$\beta$ of the ideal~$\mathfrak{a}' \mathfrak{a}^{-1}$ of~$\cO_{d, f}$.
  Then
  \begin{displaymath}
    f(x, y)
    \=
    (\alpha x - \beta y)\overline{(\alpha x - \beta y)} / [\cO_{d, f} : \mathfrak{a}' \mathfrak{a}^{-1}]
  \end{displaymath}
  is a positive definite primitive binary quadratic form with integer coefficients and discriminant~$d$ \cite[Theorem~7.7 and Exercise~7.17]{Cox13}.
  Moreover, there are integers~$x_0$ and~$y_0$ such that~$f(x_0, y_0)$ is coprime to~$m$ \cite[Lemma~2.25]{Cox13}.
  If we denote by~$\phi_0$ the isogeny in~$\Hom(E, E')$ satisfying~$\lambda_0 \= \iota(\phi_0) = \alpha x_0 - \beta y_0$, then
  \begin{multline*}
    \deg(\phi_0)
    =
    \# \Ker(\phi_0)
    =
    [\mathfrak{a}' : \lambda_0 \mathfrak{a}]
    =
    [\mathfrak{a}' \mathfrak{a}^{-1} : \lambda_0 \cO_{d, f}]
    \\ =
    [\cO_{d, f} : \lambda_0 \cO_{d, f}] / [\cO_{d, f} : \mathfrak{a}' \mathfrak{a}^{-1}]
    =
    \lambda_0 \overline{\lambda_0} / [\cO_{d, f} : \mathfrak{a}' \mathfrak{a}^{-1}]
    =
    f(x_0, y_0).
  \end{multline*}
  This proves that~$\deg(\phi_0)$ is coprime to~$m$, and completes the proof of the lemma.
\end{proof}

When restricted to~$p \ge 3$, the following lemma is~\cite[Lemma~4.8]{ColMcM06}.

\begin{lemma}
  \label{vcm}
  Let~$D$ be a $p$-supersingular discriminant and~$m \ge 0$ the largest integer such that~$p^m$ divides the conductor of~$D$.
  Then for every~$E$ in~$\supp(\Lambda_D)$ we have
  $$\kproj(E)
  =
  \begin{cases}
    \frac{1}{2} \cdot p^{-m}
    & \text{ if }p \text{ ramifies in } \Q(\sqrt{D});
    \\
    \frac{p}{p + 1} \cdot p^{-m}
    & \text{ if $p$  is inert in } \Q(\sqrt{D}).
  \end{cases} $$
\end{lemma}

\begin{proof}
  Let~$d$ be the fundamental discriminant of~$D$ and~$f \ge 1$ the integer such that the conductor of~$D$ is equal to~$p^m f$, so~$D = d (fp^m)^2$ and~$f$ is not divisible by~$p$.
  By Lemma~\ref{invarv} and Lemma~\ref{lemma-CM-vp} with~$m = p$, we deduce that for~$E$ in~$\supp(\Lambda_D)$ the number~$\kproj(D) \= \kproj(E)$ is independent of~$E$.
  By Zhang's formula~\eqref{e:general Zhang formula} with~$\wtf = p^m$ it follows that there exists an isogeny of degree~$f$ from some elliptic curve in~$\supp(\Lambda_{dp^{2m}})$ to an elliptic curve in~$\supp(\Lambda_D)$.
  We conclude from Lemma~\ref{invarv} that $\kproj(D)=\kproj(dp^{2m})$.
  Thus, it is enough to prove the lemma in the case where~$f = 1$.

  We start with~$m = 0$ and~$m = 1$.
  By~\eqref{e:p Zhang} with~$f = 1$ and Proposition~\ref{p:katz kite} with~$m = 1$ and~$n_0 = 1$, we have
  \begin{displaymath}
    \supp(\uptau_1(\kproj(d)))
    =
    \begin{cases}
      \{ \kproj(d), \kproj(dp^2) \}
      & \text{if~$p$ ramifies in $\Q(\sqrt{d})$};
      \\
      \{ \kproj(dp^{2}) \}
      & \text{if~$p$ is inert in $\Q(\sqrt{d})$}.
    \end{cases}
  \end{displaymath}
  From the definition of~$\uptau_1$ we have that~$\frac{p}{p + 1}$ is the only value of~$x$ in~$\left] 0, \frac{p}{p + 1} \right]$ such that~$\uptau_1(x)$ is supported on a single point.
  We conclude that if~$p$ is inert in~$\Q(\sqrt{d})$, then~$\kproj(d) = \frac{p}{p + 1}$ and therefore $\kproj(dp^2)=\tfrac{1}{p+1}$.
  On the other hand, $\frac{1}{2}$ is the only value of~$x$ in~$\left] 0, \frac{p}{p + 1} \right]$ satisfying~$x \in \supp(\uptau_1(x))$.
  So, if~$p$ ramifies in~$\Q(\sqrt{d})$, then~$\kproj(d) = \frac{1}{2}$ and therefore~$\kproj(dp^{2})=\frac{1}{2}p^{-1}$.
This completes the proof of the lemma when~$m = 0$ and~$m = 1$.
Assume~$m \ge 2$ and note that by~\eqref{e:pm Zhang} with~$f = 1$ and by Proposition~\ref{p:katz kite} with~$n_0 = 1$,
\begin{displaymath}
  \{ \kproj(d p^{2m}) \}
  =
  \begin{cases}
    \supp((\uptau_m - \uptau_{m - 1})(\frac{1}{2}))
    & \text{if~$p$ ramifies in $\Q(\sqrt{d})$};
    \\
    \supp((\uptau_m - \uptau_{m - 2})(\frac{p}{p + 1}))
    & \text{if~$p$ is inert in $\Q(\sqrt{d})$}.
  \end{cases}
\end{displaymath}
From the definition of~$\uptau_m$, we see that the right-hand side contains~$\frac{1}{2} \cdot p^{-m}$ if~$p$ ramifies in~$\Q(\sqrt{d})$ and~$\frac{p}{p + 1} \cdot p^{-m}$ if~$p$ is inert in~$\Q(\sqrt{d})$.
This proves~$\kproj(d p^{2m}) = \frac{1}{2} \cdot p^{-m}$ in the former case and~$\kproj(d p^{2m}) = \frac{p}{p + 1} \cdot p^{-m}$ in the latter, and completes the proof of the lemma.
\end{proof} 

\begin{proof}[Proof of Theorem~\ref{t:supersingular CM points}]
    To prove~$(i)$, note that by Proposition~\ref{vnorm} there is~$m > 0$ so that~$\kproj(\bfB(r)) \subseteq \left] \frac{p}{p + 1} \cdot p^{-m}, \frac{p}{p + 1} \right]$.
    Then by Lemma~\ref{vcm} for every $p$-supersingular discriminant~$D < 0$ satisfying~$\ord_p(D) \ge 2m + 3$ we have~$\supp((\kproj)_*(\Lambda_D)) \cap \kproj(\bfB(r)) = \emptyset$,
        and therefore~$\deg(\Lambda_D|_{\bfB(r)}) = 0$.
    On the other hand, if~$D$ is a $p$-ordinary discriminant, then~$\supp(\Lambda_D) \subset \Ord$ is disjoint from~$\bfB(r)$, and therefore~$\deg(\Lambda_D|_{\bfB(r)}) = 0$.
    This completes the proof of~$(i)$.

  To prove~$(ii)$, note that by Proposition~\ref{vnorm} there is~$r$ in~$(0, 1)$ so that
  \begin{displaymath}
    \kproj^{-1} \left( \left[ \frac{1}{2} \cdot p^{-m}, \frac{p}{p + 1} \right] \right)
    \subseteq
    \bfB(r). 
  \end{displaymath}
  Then by Lemma~\ref{vcm} for every $p$-supersingular discriminant~$D < 0$ satisfying~$\ord_p(D) \le m$ we have~$\supp((\kproj)_*(\Lambda_D)) \subseteq \left[ \frac{1}{2} \cdot p^{-m}, \frac{p}{p + 1} \right]$ and therefore~$\supp(\Lambda_D) \subseteq \bfB(r)$.
This completes the proof of~$(ii)$ and of the theorem.
\end{proof}

\begin{proof}[Proof of Theorem~\ref{t:CM points}]
  In the case where all the discriminants in the sequence~$(D_n)_{n = 1}^{\infty}$ are $p$-ordinary (resp. $p$-supersingular), Theorem~\ref{t:CM points} is a direct consequence of Theorem~\ref{t:ordinary CM points}$(ii)$ (resp. Theorem~\ref{t:supersingular CM points}), together with~\eqref{eq-Siegel-low-general} and Lemma~\ref{l:berkovich}.
  The general case follows from these two special cases.
\end{proof}

\section{Hecke orbits}
\label{s:orbits}

The goal of this section is to prove Theorem~\ref{t:orbits} on the asymptotic distribution of Hecke orbits.
The proof is divided into three complementary cases, according to whether the starting elliptic curve class has bad, ordinary or supersingular reduction.
These are stated as Propositions~\ref{p:bad orbits}, \ref{p:ordinary orbits} and~\ref{p:supersingular orbits} in Sections~\ref{ss:bad orbits}, \ref{ss:ordinary orbits} and~\ref{ss:supersingular orbits}, respectively.
In each case we prove a stronger quantitative statement.

\subsection{Hecke orbits in the bad reduction locus}
\label{ss:bad orbits}

In this section we prove a stronger version of the part of Theorem~\ref{t:orbits} concerning the bad reduction locus, which is stated as Proposition~\ref{p:bad orbits} below.
We start by recalling some well-known results on the uniformization of $p$-adic elliptic curves with multiplicative reduction.
See~\cite{Tat95} for the case of elliptic curves over complete discrete valued field, and~\cite{Roq70} for the case of complete valued fields (see also~\cite[Chapter~V, Theorem~3.1 and Remark~3.1.2]{Sil94a}).

Let~$z$ be in $\bfD(0,1)^{*}\=\{z' \in \C_p:0 < |z'|_p < 1\}$.
We obtain, by the specialization~$q=z$ in the Tate curve, an elliptic curve~$\Tate(z)$ over~$\C_p$ whose $j$-invariant satisfies
\begin{equation}\label{jtate}
| j(\Tate(z))|_p=|z|_p^{-1}>1,
\end{equation}
see \eqref{jtateq}. This defines
a bijective map
\begin{eqnarray*}
  \bfD(0,1)^{*} &\to & \Bad
  \\
  z & \mapsto & \Tate(z).
\end{eqnarray*}
Moreover, for each $z\in \bfD(0,1)^{*}$ there exists an explicit uniformization by~$\C_p^{\times}$ of the set of $\C_p$-points of $\Tate(z)$.
This uniformization induces an isomorphism of analytic groups $\varphi_z \colon \C_p^{\times}/z^{\Z}\to \Tate(z)(\C_p)$, see~\cite[Theorem~1]{Tat95} for details.
This allows us to give, for each integer~$n \ge 1$, the following description of~$T_n(\Tate(z))$.
Note that for each positive divisor~$k$ of~$n$ and each $\ell \in \bfD(0,1)^{*}$ satisfying~$\ell^k=z^{n/k}$, the set
\begin{equation}
  \label{e:1}
  C_{n,\ell}
  \=
  \{a\in \C_p^{\times}:a^{n/k}\in \ell^{\Z}\}/z^{\Z}
\end{equation}
is a subgroup of order~$n$ of $\C_p^{\times}/z^{\Z}$.
It is the kernel of the morphism of analytic groups $\C_p^{\times}/z^{\Z}\to \C_p^{\times}/\ell^{\Z}$ induced by the map $a\mapsto a^{n/k}$.
Pre-composing this morphism with~$\varphi_z^{-1}$ and then composing with~$\varphi_{\ell}$, we obtain an isogeny~$\Tate(z)\to \Tate(\ell)$ of degree~$n$ whose kernel is $\varphi_z(C_{n,\ell})$.
Since every subgroup of order~$n$ of~$\C_p^{\times}/z^{\Z}$ is of the form~\eqref{e:1}, we deduce that
\begin{equation}\label{tateq}
T_n(\Tate(z))=\sum_{\substack{k>0,k|n\\ \ell^k=z^{n/k} }}\Tate(\ell).
\end{equation}

In the case where~$E$ is in~$\Bad$, Theorem~\ref{t:orbits} is a direct consequence of the following result together with~\eqref{eq-lowb-sigma1}, \eqref{e:2} and Lemma~\ref{l:berkovich}.

\begin{proposition}\label{p:bad orbits}
  Let~$z$ in~$\bfD(0,1)^{*}$ and $R>1$ be given.
  Then, for every~$\varepsilon>0$ there exists $C>0$ such that for every integer $n\geq 1$ we have
$$ \deg(T_n(\Tate(z))|_{\bfD^{\infty}(0,R)})
\leq 
 C n^{\frac{1}{2}}d(n).
$$
\end{proposition}

\begin{proof}
Set $C\=\sqrt{-\frac{\log(|z|_p)}{\log(R)}}$ and let~$n \ge 1$ be an integer.
By~\eqref{jtate}, for a positive divisor~$k$ of~$n$ and $\ell \in \bfD(0,1)^{*}$ with $\ell^k=z^{n/k}$, we have
$$|\Tate(\ell)|_p=|\ell|_p^{-1}=|z|_p^{-n/k^2}.$$
Noting that~$|z|_p^{-n/k^2} > R$ is equivalent to~$k < C n^{\frac{1}{2}}$, from~\eqref{tateq} we deduce
$$\deg(T_n(\Tate(z))|_{\bfD^{\infty}(0,R)})
=
\sum_{\substack{k > 0, k|n \\ 0<k< C \sqrt{n}}}k
<
C n^{\frac{1}{2}}d(n).$$
This completes the proof of the proposition.
\end{proof}

\subsection{Hecke orbits in the ordinary reduction locus}
\label{ss:ordinary orbits} 
The goal of this section is to prove the following result describing, for an elliptic curve~$E$ in~$\Ord$, the asymptotic distribution of the Hecke orbit $(T_n(E))_{n = 1}^{\infty}$.
In the case where~$E$ is in~$\Ord$, Theorem~\ref{t:orbits} with~$n = p^m n_0$ is a direct consequence of this result together with~\eqref{eq-lowb-sigma1} and Lemma~\ref{l:berkovich}.

\begin{proposition}
  \label{p:ordinary orbits}
  Let~$\bfD$ be a residue disc contained in~$\Ord$ and let~$\bfB$ be a disc of radius strictly less than~$1$ contained in~$\Ord$.
  Then for every~$\varepsilon > 0$ there is a constant~$C > 0$ such that for every~$E$ in~$\bfD$ and all integers~$m \ge 0$ and~$n_0 \ge 1$ such that~$n_0$ is not divisible by~$p$, we have
  $$ \deg\left(T_{p^m n_0}(E)|_{\bfB}\right)
  \le
  C (m + 1) n_0^{\varepsilon}.$$
\end{proposition}

To prove Proposition~\ref{p:ordinary orbits} we use the multiplicative property of the Hecke correspondences, see~\eqref{eq-multiplicativity-Hecke-operators} in Section~\ref{ss:Hecke correspondences}.
We first treat the case~$n_0=1$ (Propositions~\ref{distord}) and the case~$m=0$ (Propositions~\ref{orderrante}) separately.
The proof of Proposition~\ref{p:ordinary orbits} is given at the end of this section.

\begin{proposition}
  \label{distord}
  Let~$\zeta$ in~$\Fpalg$ be the $j$-invariant of an ordinary elliptic curve, denote by~$r$ the minimal period of~$\zeta$ under the Frobenius map $z\mapsto z^p$ and put~$\bfO \= \bigcup_{i=0}^{r-1}\bfD (\zeta^{p^i})$.
  Then for every~$E$ in~$\bfD(\zeta)$ and every integer~$m \ge 1$, we have
\begin{equation}\label{tpinresidue}
\supp(T_{p^m}(E))\subseteq \bfO.
\end{equation}
Moreover, for every disc~$\bfB$ of radius strictly less than~$1$ contained in~$\bfO$ there is a constant~$C_1 > 0$ such that for every~$E$ in~$\bfO$ and every integer~$m \ge 1$, we have
\begin{equation}
    \label{disttpm}
    \deg(T_{p^m}(E)|_{\bfB})
    \le
    C_1 m.
\end{equation}
\end{proposition}
\begin{proof}
  The inclusion~\eqref{tpinresidue} is a direct consequence of Proposition~\ref{prop-tpm} and~\eqref{eq-cong-t}.
  To prove~\eqref{disttpm}, let~$e$ be an ordinary elliptic curve with $j$-invariant~$\zeta$, for every integer~$i \ge 0$ put~$z_i \= \t^i(e^{\uparrow})$ and for every integer~$i \le - 1$ let~$i'$ be the unique integer in~$\{0, \ldots, r - 1 \}$ such that~$i - i'$ is divisible by~$r$ and put~$z_i \= z_{i'}$. Note that for all nonnegative integers~$a,b$, every integer~$i$ and every point~$z$ in~$\bfD(z_{i},1)$, the set~$\t^{-a}(\t^b(z))$ is contained in~$\bfD(z_{i+b-a},1)$.
    Let~$c$ in~$(0, 1)$ be such that~$\bfB$ is contained in~$\bfB(c) \= \bigcup_{i = 0}^{r - 1} \bfD(z_i, c)$, let~$\rho$ and~$\kappa_c$ be given by Lemma~\ref{periodicpoint} and let~$i_1 \ge 0$ be a sufficiently large integer so that~$c \kappa_c^{i_1} < \rho$.

  Fix~$E$ in~$\bigcup_{i = 0}^{r - 1} \bfD(z_i, 1)$ and let~$m \ge 1$ be a given integer.
  Without loss of generality we assume $E\in \bfD(z_0,1)$. 
  We treat the cases~$m<i_1$ and~$m\geq i_1$ separately.
  If~$m < i_1$, then we have
  \begin{displaymath}
    \deg(T_{p^m}(E)|_{\bfB(c)})
    \le
    \deg(T_{p^m}(E))
    =\frac{p^{m+1}-1}{p-1}
    \le
    p^{i_1}m.
  \end{displaymath}
  Now, assume~$m \ge i_1$.
  If for every~$i$ in~$\{0, \ldots, m \}$ the set~$\t^{-(m - i)}(\t^i(E))$ is disjoint from~$\bfD(z_{2i -m}, c)$, then 
  \begin{displaymath}
    \deg(T_{p^m}(E)|_{\bfB(c)})
    =
    \sum_{i=0}^m\deg((\t^{\ast})^{(m-i)}([\t^i(E)])|_{\bfD(z_{2i -m}, c)})
    =0.
  \end{displaymath}
So we assume this is not the case and denote by~$i_0$ the least integer~$i$ in~$\{0, \ldots, m \}$ such that~$\t^{-(m - i)}(\t^i(E))$ contains a point~$E_0$ in~$\bfD(z_{2i -m}, c)$.
Note that by Lemma~\ref{periodicpoint}$(iii)$ the point~$E_1 \= \t^{i_1}(E_0)$ satisfies
\begin{displaymath}
  |E_1 - z_{2i_0 - m + i_1}|_p \le c \kappa_c^{i_1} < \rho,
\end{displaymath}
so it is in~$\bfD( z_{2i_0 - m + i_1}, \rho)$.

If~$m \le i_0 + i_1$, then we have
\begin{multline*}
    \deg \left(T_{p^m}(E)|_{\bfB(c)}\right)
    =
    \sum_{i = i_0}^{m} \deg \left( (\t^*)^{m - i}([\t^i(E)]) \right)
    \\ \le
    \sum_{i = i_0}^{m} p^{m - i}
    =\frac{p^{m-i_0+1}-1}{p-1}
    \le
    p^{i_1}(m + 1).
  \end{multline*}
  
Suppose~$m > i_0 + i_1$, and let~$i$ be an integer satisfying $i_0\leq i \leq m - i_1$.
Noting that for every~$E'$ in~$\t^{-(m - i)}(\t^{i}(E))$ we have
\begin{displaymath}
  \deg_{\t^{m - i}}(E')
  =
  \deg_{\t^{m - i - i_1}}(\t^{i_1}(E'))\deg_{\t^{i_1}}(E'),
\end{displaymath}
we obtain
\begin{equation}
  \label{e:8}
  (\t^{*})^{m - i}([\t^i(E)])
  =
  \sum_{E'' \in \t^{-(m - i - i_1)}(\t^i(E))} \deg_{\t^{m - i - i_1}}(E'') (\t^*)^{i_1}([E'']).
\end{equation}
On the other hand, for every~$z$ in~$\t^{-(m - i)}(\t^i(E))$ contained in~$\bfD(z_{2i - m}, c)$, we have by Lemma~\ref{periodicpoint}$(iii)$ and our choice of~$i_1$,
\begin{displaymath}
  | \t^{i_1}(z) - z_{2i - m + i_1} |_p
  \le
  c \kappa_c^{i_1}
  <
  \rho,
\end{displaymath}
so~$\t^{i_1}(z) \in \bfD(z_{2i - m + i_1}, \rho)$. 
Since  for such~$z$ we have
\begin{displaymath}
  \t^{m - i - i_1}(\t^{i_1}(z))
  =
  \t^{i}(E)
  =
  \t^{m - i - i_1}(\t^{2i - 2i_0}(E_1))
\end{displaymath}
and by Lemma~\ref{periodicpoint}$(ii)$ the map~$\t^{m - i - i_1}$ is injective on~$\bfD(z_{2i - m + i_1}, \rho)$, we conclude that~$\t^{i_1}(z) = \t^{2i - 2i_0}(E_1)$.
Since we also have
\begin{displaymath}
  \deg_{\t^{m - i - i_1}}(\t^{2i - 2i_0}(E_1)) = 1
\end{displaymath}
by Lemma~\ref{periodicpoint}$(ii)$, when we restrict~\eqref{e:8} to~$\bfD(z_{2i - m}, c)$ we obtain
\begin{equation*}
  (\t^{*})^{m - i}([\t^i(E)])|_{\bfD(z_{2i - m}, c)}
  =
  (\t^*)^{i_1}([\t^{2i - 2i_0}(E_1)])|_{\bfD(z_{2i - m}, c)},
\end{equation*}
and therefore
\begin{displaymath}
  \deg \left((\t^*)^{m-i}([\t^i(E)])|_{\bfD(z_{2i-m},c)}\right)
  \le
  \deg \left( (\t^*)^{i_1} ([\t^{2i - 2i_0}(E_1)])\right)
  =
  p^{i_1}.
\end{displaymath}
Together with Proposition~\ref{prop-tpm} and our definition of~$i_0$, this implies
\begin{multline*}
    \deg \left(T_{p^m}(E)|_{\bfB(c)}\right)
    \\
    \begin{aligned}
     \leq &
    \sum_{i = i_0}^{m - i_1 - 1} \deg \left( (\t^*)^{m - i}([\t^i(E)])|_{\bfD(z_{2i - m}, c)} \right)
    + \sum_{i = m - i_1}^{m} \deg \left( (\t^*)^{m - i}([\t^i(E)]) \right)
    \\ \le &
    p^{i_1} (m - i_0 - i_1) + \sum_{i = m - i_1}^{m} p^{m - i}
    \\ \le &
    p^{i_1}(m + 1).
    \end{aligned}
\end{multline*}
This completes the proof of Proposition~\ref{distord} with~$C_1 = 2p^{i_1}$.
\end{proof}

\begin{proposition}
  \label{orderrante}
  Let~$\bfD$ and~$\bfD'$ be residue discs contained in~$\Ord$.
  Then for every $\varepsilon>0$ there is a constant~$C_1 > 0$ such that for every~$E$ in~$\bfD$ and every integer~$n \ge 1$ that is not divisible by~$p$, we have
  $$ \deg(T_n(E)|_{\bfD'})
  \le
  C_2 n^{\varepsilon}. $$
\end{proposition}

To prove this proposition we first establish an intermediate estimate.

\begin{lemma}\label{conteoo}
  Let~$\ss$ and~$\ss'$ be ordinary elliptic curves over~$\Fpalg$, and for each integer~$n \ge 1$ denote by~$\Hom_n(\ss,\ss')$ the set of isogenies from~$\ss$ to~$\ss'$ of degree~$n$.
  Then, for every~$\varepsilon > 0$ we have
  \begin{equation}
    \label{eq:3}
    \#\Hom_n(\ss,\ss') = o(n^{\varepsilon}).
  \end{equation}
\end{lemma}
\begin{proof}
    Assume there is a nonzero element~$\phi_0$ in~$\Hom(\ss',\ss)$, for otherwise there is nothing to prove.
  Then, the map $\iota \colon \Hom(\ss,\ss') \to  \End(\ss)$ given by $\iota(\phi) =\phi_0\circ \phi$ is an injection, and $\deg(\iota(\phi))=\deg(\phi_0)\deg(\phi)$.
  It is thus enough to prove~\eqref{eq:3} when $\ss'=\ss$.

  Since~$\ss$ is ordinary, the ring $\End(\ss)$ is isomorphic to an order inside a quadratic imaginary extension~$K$ of~$\Q$.
  Moreover, the isomorphism can be taken such that the degree of an isogeny is the same as the field norm of the corresponding element in~$K$, see, \emph{e.g.}, \cite[Chapter~V, Theorem~3.1]{Sil09}.
  Let~$d$ be the discriminant of~$K$.
  Then~$\cO_{d,1}$ is the ring of integers of~$K$, and hence it is enough to show
\begin{displaymath}
  \#\{ x \in \cO_{d, 1} : x \overline{x} = n\}
  =
  o(n^\varepsilon).
\end{displaymath}
Since the group of units~$\cO_{d, 1}^{\times}$ is finite, this estimate follows from~\eqref{e:2} and~\eqref{eq-Rd-char}.
\end{proof}

\begin{proof}[Proof of Proposition~\ref{orderrante}]
  Let~$\ss$ be the ordinary elliptic curve over~$\Fpalg$ so that~$\bfD' = \bfD(j(e))$. 
  In view of Lemma~\ref{conteoo}, it is sufficient to show that for every~$E$ in~$\bfD$ and every integer~$n \ge 1$ that is not divisible by~$p$ we have
  \begin{equation}
    \label{e:4}
    \deg(T_n(E)|_{\bfD'})\leq \# \Hom_n(\tE, e).
  \end{equation}
  Since the function $E \mapsto \deg(T_n(E)|_{\bfD'})$ is locally constant by Lemma~\ref{lemma-Hecke-continuity}, it is sufficient to establish this inequality in the case where~$E$ is in~$\Ordun$.
  
  To prove~\eqref{e:4}, recall that the reduction morphism $E\to \tE$ induces a bijective map $E[n]\to \tE[n]$, see for example~\cite[Chapter~VII, Proposition~3.1(b)]{Sil09}.
  In addition, note that for a subgroup~$C$ of~$E$ of order~$n$ such that~$j(E/C)$ is in~$\bfD'$, there is an isogeny~$\tE\to e$ whose kernel is equal to the reduction of~$C$.
  This defines an injective map
  \begin{displaymath}
    \{C\leq E: \#C=n,j(E/C)\in \bfD' \} \to \Hom_n(\tE,e),
  \end{displaymath}
  proving~\eqref{e:4} and completing the proof of the proposition.
\end{proof}

\begin{proof}[Proof of Proposition~\ref{p:ordinary orbits}]
  Let~$C_1$ and~$C_2$ be given by Propositions~\ref{distord} and~\ref{orderrante}, respectively.
  Let~$\zeta$ in~$\Fpalg$ be such that~$\bfB \subseteq \bfD(\zeta)$, let~$r \ge 1$ be the minimal period of~$\zeta$ under the Frobenius map and put~$\bfO \= \bigcup_{i = 0}^{r - 1} \bfD(\zeta^{p^i})$.

  Let~$E$ in~$\bfD$ be given.
  By~\eqref{tpinresidue}, for every~$E'$ in~$\supp (T_{n_0}(E))$ that is not in~$\bfO$ we have
  \begin{displaymath}
    \deg(T_{p^m}(E')|_{\bfB})
    \le
    \deg(T_{p^m}(E')|_{\bfO})
    =
    0.
  \end{displaymath}
  On the other hand, for every~$E'$ in~$\supp(T_{n_0}(E))$ that is in~$\bfO$, we have by Proposition~\ref{distord}
  \begin{displaymath}
    \deg(T_{p^m}(E')|_{\bfB})
    \le
    C_1 m + 1.
  \end{displaymath}
  Together with~\eqref{eq-multiplicativity-Hecke-operators} and Proposition~\ref{orderrante} with~$\bfD' = \bfD(\zeta), \ldots, \bfD(\zeta^{p^{r - 1}})$, this implies
  $$ \deg(T_{p^m n_0}(E)|_{\bfB})
    \le
    (C_1 m + 1) \deg(T_{n_0}(E)|_{\bfO})
  \le
  r C_2(C_1 + 1) (m + 1) n_0^{\varepsilon}. $$
  This proves the theorem with~$C = r C_2 (C_1 + 1)$.
\end{proof}

\subsection{Hecke orbits in the supersingular reduction locus}
\label{ss:supersingular orbits}
The purpose of this section is to prove the following result on Hecke orbits inside the supersingular reduction locus.
In the case where~$E$ is in~$\Sups$, Theorem~\ref{t:orbits} with~$n = p^m n_0$ is a direct consequence of this result together with~\eqref{eq-lowb-sigma1} and Lemma~\ref{l:berkovich}.

 \begin{proposition}
   \label{p:supersingular orbits}
   For every~$\ss$ in~$\tSups$ fix an arbitrary~$\gamma_{\ss}$ in~$\bfD(j(\ss))$ and for every~$r > 0$, put
   \begin{displaymath}
     \bfB(r)
     \=
     \bigcup_{\ss \in \tSups} \bfD(\gamma_{\ss}, r).
   \end{displaymath}
   Then the following properties hold.
\begin{enumerate}
\item[$(i)$]
  For every~$r$ in~$(0, 1)$ there is a constant~$C > 0$ such that for every~$E$ in~$\Sups$, every integer~$m \ge 0$ and every integer~$n_0 \ge 1$ that is not divisible by~$p$, we have
  \begin{displaymath}
    \deg(T_{p^m n_0}(E)|_{\bfB(r)})
    \le
    C \sigma_1(n_0).
  \end{displaymath}
\item[$(ii)$]
  For every~$r_0$ in~$(0, 1)$ and every integer~$m_0 \ge 0$, there is~$r$ in~$(0, 1)$ such that for every~$m$ in~$\{0, \ldots, m_0 \}$ and integer~$n_0 \ge 1$ not divisible by~$p$, we have for every~$E$ in~$\bfB(r_0)$
  \begin{displaymath}
    \supp(T_{p^m n_0}(E)) \subseteq \bfB(r).
  \end{displaymath}
\end{enumerate}
\end{proposition}

The proof of this result is based on the following lemma, giving for each integer~$m \ge 0$ a formula for the correspondence~$\uptau_m$ defined in Proposition~\ref{p:katz kite}.
To state this lemma, for each integer~$k \ge 0$ put
\begin{displaymath}
  x_k \= \frac{p}{p + 1} \cdot p^{-k}
  \text{ and }
  I_k
  \=
  [x_{k + 1}, x_k],
\end{displaymath}
and note that~$\bigcup_{k = 0}^{\infty} I_k = \left] 0, \frac{p}{p + 1} \right]$.
Moreover, for all integers~$k, k' \ge 0$ denote by
\begin{displaymath}
  A_{k, k'}^{(+1)} \colon I_k \to I_{k'}
  \text{ (resp. }
  A_{k, k'}^{(-1)} \colon I_k \to I_{k'}
  \text{)}
\end{displaymath}
the unique affine bijection preserving (resp. reversing) the orientation.
Note that for every~$k \ge 0$ we have $1 - A_{k, 0}^{(+1)} = A_{k, 0}^{(-1)}$ and that for every~$k' \ge 1$ we have
\begin{equation}
  \label{e:10}
  p A_{k, k'}^{(\pm 1)} = A_{k, k' - 1}^{(\pm 1)}.
\end{equation}
\begin{lemma}
  \label{l:katz hyper kite}
  For each integer~$m \ge 0$ denote by~$\uptau_m$ the correspondence acting on~$\left[0, \frac{p}{p + 1} \right]$ defined in Proposition~\ref{p:katz kite}.
    Then for all integers~$k, m \ge 0$, we have
    \begin{multline*}
      \uptau_m|_{I_k} =
      \\
      \begin{cases}
        \sum_{i = 0}^m p^i \left( A_{k, 2i - (m - k)}^{(+1)} \right)_{*}
        & \text{if } m \le k;
        \\
        \sum_{i = 0}^{m - k - 1} p^i \left( A_{k, i}^{((-1)^{m - k - i})} \right)_{*}
        + \sum_{i = m - k}^{m} p^i \left( A_{k, 2i - (m - k)}^{(+1)} \right)_{*}
        & \text{if } m \ge k + 1.
      \end{cases}
    \end{multline*}
  \end{lemma}
  \begin{proof}
    Fix~$k \ge 0$.
    We proceed by induction on~$m$.
    The case~$m = 0$ is trivial and the case~$m = 1$ is a direct consequence of the definition given in Proposition~\ref{p:katz kite}.
    Let~$m \ge 2$ be given and suppose that the lemma holds with~$m$ replaced by~$m - 1$ and by~$m - 2$.
    If~$m \le k$, then by~\eqref{e:10}
    \begin{align*}
        \uptau_1 (\uptau_{m - 1}|_{I_k})
      & =
      \sum_{i = 0}^{m - 1} p^i \left( A_{k, 2i - (m - k)}^{(+1)} \right)_{*}
      + \sum_{i = 0}^{m - 1} p^{i + 1} \left( A_{k, 2i - (m - k) + 2}^{(+1)} \right)_{*}
      \\ & =
      p \uptau_{m - 2}|_{I_k} + \sum_{i = 0}^{m } p^i \left( A_{k, 2i - (m - k)}^{(+1)} \right)_{*},
    \end{align*}
    which proves the induction step in the case~$m \le k$.
    In the case~$m = k + 1$, using~$1 - A_{k, 0}^{(+1)} = A_{k, 0}^{(-1)}$ we have
    \begin{align*}
        \uptau_1(\uptau_{k}|_{I_k})
      & =
      \left( A_{k, 0}^{(-1)} \right)_{*} + \sum_{i = 1}^k p^i \left( A_{k, 2i - 1}^{(+1)} \right)_{*}
      + \sum_{i = 0}^k p^{i + 1} \left( A_{k, 2i + 1}^{(+1)} \right)_{*}
      \\ & =
      p \uptau_{k - 1}|_{I_k}
      + \left( A_{k, 0}^{(-1)} \right)_{*} + \sum_{i = 1}^{k + 1} p^i \left( A_{k, 2i - 1}^{(+1)} \right)_{*}.
    \end{align*}
    This proves the induction step in the case~$m = k + 1$.
    If~$m = k + 2$, then
    \begin{align*}
        \uptau_1 (\uptau_{k + 1}|_{I_k})
      & =
      \left( A_{k, 0}^{(+1)} \right)_{*} + p \left( A_{k, 1}^{(-1)} \right)_{*}
      + \sum_{i = 1}^{k + 1} p^i \left( A_{k, 2i - 2}^{(+1)} \right)_{*}
      + \sum_{i = 1}^{k + 1} p^{i + 1} \left( A_{k, 2i}^{(+1)} \right)_{*}
      \\ & =
      \left( A_{k, 0}^{(+1)} \right)_{*}+ p \left( A_{k, 1}^{(-1)} \right)_{*}
      + p \uptau_{k}|_{I_k}
      + \sum_{i = 2}^{k + 2} p^i \left( A_{k, 2i - 2}^{(+1)} \right)_{*}.
    \end{align*}
    This proves the induction step in the case~$m = k + 2$.
    Finally, if~$m \ge k + 3$, then~$\uptau_1 (\uptau_{m - 1}|_{I_k})$ is equal to
    \begin{multline*}
        \left( A_{k, 0}^{((-1)^{m - k})} \right)_{*}
        + \sum_{j = 1}^{m - k - 2} p^j \left( A_{k, j - 1}^{((-1)^{m - k - j - 1})} \right)_{*}
        + \sum_{j = 0}^{m - k - 2} p^{j + 1} \left( A_{k, j + 1}^{((-1)^{m - k - j - 1})} \right)_{*}
        \\
        \begin{aligned}
          & \quad
        + \sum_{i = m - k - 1}^{m - 1} p^i \left( A_{k, 2i - (m - k)}^{(+1)} \right)_{*}
        + \sum_{i = m - k - 1}^{m - 1} p^{i + 1} \left( A_{k, 2i - (m - k - 2)}^{(+1)} \right)_{*}
        \\ & =
        \sum_{\ell = 0}^{m - k - 1} p^{\ell} \left( A_{k, \ell}^{((-1)^{m - k - \ell})} \right)_{*}
        + p \sum_{s = 0}^{m - k - 3} p^s \left( A_{k, s}^{((-1)^{m - k - s - 2})} \right)_{*}
        \\ & \quad
        + \sum_{i = m - k}^{m} p^i \left( A_{k, 2i - (m - k)}^{(+1)} \right)_{*}
        + p \sum_{i = m - k - 2}^{m - 2} p^i \left( A_{k, 2i - (m - k - 2)}^{(+1)} \right)_{*}
        \\ & =
        p \uptau_{m - 2}|_{I_k} + \sum_{\ell = 0}^{m - k - 1} p^{\ell} \left( A_{k, \ell}^{((-1)^{m - k - \ell})} \right)_{*}
        + \sum_{i = m - k}^{m} p^i \left( A_{k, 2i - (m - k)}^{(+1)} \right)_{*}.
        \end{aligned}
    \end{multline*}
    This completes the proof of the induction step and of the lemma.
  \end{proof}

\begin{proof}[Proof of Proposition~\ref{p:supersingular orbits}]
  Let~$\kproj$ and~$(\uptau_m)_{m = 0}^{\infty}$ be as in Proposition~\ref{p:katz kite}.
  
  To prove~$(i)$, let~$r$ in~$(0, 1)$ be given.
  By Proposition~\ref{vnorm} there is an integer~$\ell \ge 0$ such that~$\kproj(\bfB(r)) \subseteq [x_{\ell}, x_0]$.
  Then the desired assertion follows from Proposition~\ref{p:katz kite} and by the observation that by Lemma~\ref{l:katz hyper kite} for every~$x$ in~$]0, x_0]$ we have
  \begin{displaymath}
    \deg(\uptau_m(x)|_{[x_{\ell}, x_0]})
    \le
    1 + p + \cdots + p^{\ell}.
  \end{displaymath}

  To prove~$(ii)$, let~$r_0$ in~$(0, 1)$ and an integer~$m_0 \ge 0$ be given.
  By Proposition~\ref{vnorm} there is an integer~$\ell \ge 0$ such that~$\kproj(\bfB(r_0)) \subseteq [x_{\ell}, x_0]$ and~$r$ in~$(0, 1)$ such that~$\kproj^{-1}([x_{\ell + m_0}, x_0]) \subseteq \bfB(r)$.
  Then the desired inclusion follows from Proposition~\ref{p:katz kite} by noting that by Lemma~\ref{l:katz hyper kite} for every~$x$ in~$[x_{\ell}, x_0]$ and every~$m$ in~$\{0, \ldots, m_0 \}$, we have~$\supp(\uptau_m(x)) \subseteq [x_{\ell + m_0}, x_0]$.
\end{proof} 

\appendix
\section{Lifting the Hasse invariant in characteristic 2 and 3}
\label{s:lift Hasse 2 3}

When~$p$ equals~$2$ or~$3$ it is not possible to lift the Hasse invariant~$A_{p-1}$ to a modular form of level one, holomorphic at infinity, over~$\Zp$.
There are two approaches to solve this issue. On the one  hand,  there are liftings of~$A_1^4$ and~$A_2^3$ in the desired space (namely, the Eisenstein series~$\oE_4$ and~$\oE_6$). 
On the other hand, considering level structures, liftings can be constructed as algebraic modular forms over~$\Z_{(p)}$  of the expected weight but higher level. In this appendix we recall both approaches, following \cite[Section~2.1]{Kat73}, and give a quantitative comparison between them, embodied in Proposition~\ref{prop-E1-E2} below. Such comparison is needed in Section~\ref{ss:katz kite}.

We start by recalling level structures. Let~$R$ be a ring and let $n\geq 1$ be an integer which is assumed to be invertible in~$R$.
Let~$E$ be an elliptic curve over~$R$ in the sense of Section~\ref{ss:katz valuation}.
A \emph{level~$n$ structure} on~$E$ over~$R$  is an isomorphism $\alpha_n \colon E[n]\to (\Z/n\Z)^2$ of group schemes over~$R$. 

Given an integer $n\geq 1$ and an arbitrary ring~$R_0$ where~$n$ is invertible, an algebraic modular form of level $n\geq 1$ over~$R_0$  is a  family of maps
$F=(F_R)_{R\in R_0\text{-Alg}}$ such that for any $R \in R_0\text{-Alg}$, the~$R$-valued map~$F_R$ is defined on the the set of triples $(E,\omega,\alpha_n)$, where~$E$ is an elliptic curve over $R\in R_0\text{-Alg}$, together with a differential form in~$\Omega^1_{E/R}(E)'$ and a level~$n$ structure.
The element  $F_R(E,\omega,\alpha_n) \in R$  must define an assignment satisfying properties analogous to~$(i),(ii)$ and~$(iii)$ stated in Section~\ref{ss:katz valuation}.
See \cite[Section~1.2]{Kat73} for further details.

 When  $R_0$ contains~$1/n$ and a primitive $n$-th root of unity, the $q$-expansions of an algebraic modular form~$F$ of level~$n$ over~$R_0$  are defined as the elements of $\Z(\!( q)\!) \otimes_{\Z}R_0$ obtained by evaluating~$F$ at the triples $(\Tate(q^n),\omega_{\can},\alpha_n)_{R_0}$ consisting of the Tate curve~$\Tate(q^n)$ (see Section~\ref{ss:bad orbits}) with its canonical differential $\omega_{\can}$, viewed as defined over $\Z(\!( q)\!)\otimes_{\Z}R_0$, with~$\alpha_n$ varying over all level~$n$ structures of $\Tate(q^n)$ over $\Z(\!( q)\!) \otimes_{\Z}R_0$.
 If all of the $q$-expansions of~$F$ lie in $\Z\llbracket q\rrbracket \otimes_{\Z}R_0$ then~$F$ is called holomorphic at infinity. For algebraic modular forms~$F$ of level one there is only one $q$-expansion, which coincides with the previously defined~$F(q)$.

According to~\cite[p.~98]{Kat73}, for any level $3\leq n\leq 11$ odd, there exists a lifting of~$A_1$ to a modular form of level~$n$, weight one, holomorphic at infinity, over~$\Z[1/n]$. We define~$\oE_1$ as any such lifting and set~$n(\oE_1)\=n$. Similarly, when~$m\geq 4$ and~${3 \nmid m}$, there exists a lifting of~$A_2$ to a modular form of level~$m$, weight two, holomorphic at infinity, over~$\Z[1/m]$. We define~$\oE_2$ as any such lifting and set~$n(\oE_2)\=m$.

The following statement is a comparison between both approaches.

\begin{proposition}
  \label{prop-E1-E2}
Let~$E\in \Sups$ and let~$\omega$ be a differential form in~$\Omega^1_{E/\cO_p}(E)'$.
\begin{enumerate}
\item[$(i)$]
  For any level~$n(\oE_1)$ structure~$\alpha$ on~$E$ we have
$$\ord_2(\oE_4(E,\omega))<3\Leftrightarrow \ord_2(\oE^4_1(E,\omega,\alpha))<3, $$
in which case $\ord_2(\oE_4(E,\omega))=\ord_2(\oE^4_1(E,\omega,\alpha))$.
\item[$(ii)$] For any  level~$n(\oE_2)$ structure~$\alpha$ on~$E$ we have
$$\ord_3(\oE_6(E,\omega))<\frac{5}{2}\Leftrightarrow \ord_3(\oE^3_2(E,\omega,\alpha))<\frac{5}{2}, $$
in which case $\ord_3(\oE_4(E,\omega))=\ord_3(\oE^3_2(E,\omega,\alpha))$.
\end{enumerate}
\end{proposition}
\begin{proof}
In order to prove~$(i)$, we start by recalling the $q$-expansion
$$\oE_4(q)=1+240\sum_{n=1}^{\infty}\sigma_3(n)q^n,$$
obtained by setting~$k=4$ in~\eqref{eq-q-exp-Ek}.
Since~$\ord_2(240)=4$, we have~$\oE_4(q)\equiv 1$ mod~$2^4$. Now, put~$n_1\=n(\oE_1)$, let~$\zeta_{n_1}$ be a primitive $n_1$-th roof of unity and define~$R_1\=\Z[1/n_1,\zeta_{n_1}]$. By the definition of~$\oE_1$ we have 
$$\oE_1(\Tate(q^{n_1}),\omega_{\can},\alpha_{n_1}) \equiv A_1(q^{n_1})\equiv 1 \, \mod 2R_1,$$
hence 
$$\oE_1^4(\Tate(q^{n_1}),\omega_{\can},\alpha_{n_1}) \equiv 1\equiv \oE_4(q^{n_1})  \mod 2^3R_1,$$
for any level~$n_1$ structure~$\alpha_{n_1}$ on~$\Tate(q^{n_1})$. We conclude that the form~$f$ obtained by reducing modulo~$2^3\Z[1/n_1]$  the form $\oE_4-\oE^4_1$ is an algebraic modular form of weight $4$, level $n_1$ over $\Z/2^3\Z$, whose $q$-expansions over $(\Z/2^3\Z)[\zeta_{n_1}]$ vanish identically.
By~\cite[Theorem~1.6.1]{Kat73} we deduce that $f=0$. By compatibility with base change we conclude that for any $\Z[1/n_1]$-algebra $R$ and any triple $(E,\omega,\alpha_{n_1})$ over $R$ we have
$$\oE_4(E,\omega)- \oE_1^4(E,\omega,\alpha_{n_1})\equiv f((E,\omega,\alpha_{n_1})_{R/2^3R})\equiv 0 \, \mod 2^3R.$$
In particular, choosing $R=\cO_p$, we get
\begin{equation}\label{eq-cong-E4-E2}
\ord_2(\oE_4(E,\omega)- \oE_1^4(E,\omega,\alpha_{n_1}))\geq 3,
\end{equation}
for every $E\in \Sups$, every basis $\omega$ of $\Omega^1_{E/\cO_p}$ and every level~$n_1$ structure~$\alpha_{n_1}$ on~$E$.
Then, $(i)$ is a direct consequence of~\eqref{eq-cong-E4-E2} and the ultrametric inequality.

The proof of~$(ii)$ is unfortunately less straightforward.
This is because the same argument used to prove~\eqref{eq-cong-E4-E2} only yields the inequality
$$\ord_3(\oE_6(E,\omega)- \oE_2^3(E,\omega,\alpha_{n_2}))\geq 2,$$
valid for any level $n_2\=n(\oE_2)$ structure $\alpha_{n_2}$ on $E$, but such inequality does not imply the desired result. On the other hand, the above argument allows us to infer
\begin{equation}\label{eq-E4-E2}
\ord_3(\oE_4(E,\omega)- \oE_2^2(E,\omega,\alpha_{n_2}))\geq 1.
\end{equation}
In order to prove~$(ii)$ we introduce the series
\begin{equation}\label{eq-G2-q-exp}
G_2(\tau)=1+24\sum_{n=1}^{\infty}\left(\sigma_1(n)-2\sigma_1\left(\frac{n}{2}\right)\right)e^{2\pi i n\tau}, \quad \tau \in \H,
\end{equation}
where $\sigma_1\left(\tfrac{n}{2}\right)$ is defined as zero when $n$ is odd. It is known that~$G_2$ is a classical holomorphic modular form of weight two for the group $\Gamma_0(2)=\{g \in \SL_2(\Z): g \equiv \left(\begin{smallmatrix} \ast & \ast \\ 0 & \ast \end{smallmatrix}\right) \mod 2\}$.\footnote{Up to an explicit multiplicative factor, this is denoted by~$G_{2,2}$ in~\cite[Section~1.2]{DiaShu05}.}
By \cite[Corollary~1.9.1]{Kat73}, $G_2$ defines an algebraic modular over~$\Z\left[1/2\right]$ of weight two and level two.
This form satisfies the identity
\begin{equation}\label{eq-G2-E6-E4}
4\, G_2^3=\oE_6+3\, \oE_4\, G_2.
\end{equation}
Indeed, the space of modular forms over~$\C$ of weight six for~$\Gamma_0(2)$ has dimension~$2$, see the dimension formulas in~\cite[Chapter~3]{DiaShu05}.
By comparing Fourier expansions, it is easy to check that~$\oE_6$ and~$\oE_4\, G_2$ are linearly independent over~$\C$, hence they form a basis of such space. This implies that there exist~$a,b\in \C$ with $G_2^3=a \, \oE_6+b\, \oE_4\, G_2$.
Then, \eqref{eq-G2-E6-E4} follows at the level of classical modular forms by computing the values of $a$ and $b$, which can be done by comparing Fourier expansions.
Finally, the fact that~\eqref{eq-G2-E6-E4} holds as an identity between algebraic modular forms over $\Z\left[1/2\right]$ is a consequence of~\cite[Corollary~1.9.1]{Kat73}.

We also recall the identity
$$\oE_6^2-\oE_4^3=1728\, \Delta.$$
At the level of classical modular forms,  see for example~\cite[Sections~1.1 and~1.2]{DiaShu05}. Then, this identity holds at the level of algebraic modular forms by the same reasoning as before. 
Given $E\in\Sups$ and a differential form~$\omega$ in~$\Omega^1_{E/\cO_p}(E)'$, we have $\Delta(E,\omega)\in \cO_p^{\times}$ since~$E$ has good reduction. This implies
\begin{equation}\label{eq-ord-E6-E4}
\ord_3(\oE_6^2(E,\omega)-\oE_4^3(E,\omega))=3.
\end{equation}
By using~\eqref{eq-G2-E6-E4} and~\eqref{eq-ord-E6-E4}, we will now prove~$(ii)$.
Let~$\alpha$ be a level~$n_2$ structure on~$E$. First, assume that $\ord_3(\oE_2(E,\omega,\alpha))<\tfrac{5}{6}$.
From~\eqref{eq-G2-E6-E4} we see that the reduction modulo~$3$ of~$G_2$ equals~$A_2$.
Since the same holds for~$\oE_2$, we conclude that
\begin{equation}\label{eq-E2-G2}
\ord_3(\oE_2(E,\omega,\alpha)-G_2(E,\omega,\beta))\geq 1, 
\end{equation}
for any level two structure~$\beta$. In particular
$$\ord_3(G_2(E,\omega,\beta))=\ord_3(\oE_2(E,\omega,\alpha))<\frac{5}{6}.$$
By~\eqref{eq-G2-E6-E4} we have
$$\oE_6(E,\omega)=G_2(E,\omega,\beta)\, (4G_2^2(E,\omega,\beta)-3\oE_4(E,\omega)).$$
But by~\eqref{eq-E4-E2} and~\eqref{eq-E2-G2} we also have
\begin{align*}
    \ord_3(3\, \oE_4(E,\omega))
    & =
    1+\ord_3(\oE_4(E,\omega))
    \\ & \geq
    1+\min\{1,\ord_3(G_2^2(E,\omega,\beta))\}
    \\ & >
    \ord_3(G_2^2(E,\omega,\beta)),
\end{align*}
hence
$$\ord_3(\oE_6(E,\omega))=\ord_3(G_2^3(E,\omega,\beta))=\ord_3(\oE_2^3(E,\omega,\alpha)).$$
This proves one implication. Let us now prove the reciprocal. We start by assuming that $\ord_3(\oE_6(E,\omega))<\tfrac{5}{2}$. If $\ord_3(\oE_4(E,\omega))<1$, then we can use~\eqref{eq-E4-E2}, \eqref{eq-ord-E6-E4} and~\eqref{eq-E2-G2} to deduce that $\ord_3(\oE^3_4(E,\omega))=\ord_3(\oE^2_6(E,\omega))$ and $\ord_3(G_2^2(E,\omega,\beta))=\ord_3(\oE_4(E,\omega))$. This implies
$$\ord_3(3\, G_2(E,\omega,\beta)\, \oE_4(E,\omega))=1+\ord_3(\oE_6(E,\omega))> \ord_3(\oE_6(E,\omega)).$$
By~\eqref{eq-G2-E6-E4} and~\eqref{eq-E2-G2} we conclude
$$\ord_3(\oE_2^3(E,\omega,\alpha))=\ord_3(G_2^3(E,\omega,\beta))=\ord_3(\oE_6(E,\omega)).$$
Now, if $\ord_3(\oE_4(E,\omega))\geq 1$ then~\eqref{eq-E4-E2} and~\eqref{eq-E2-G2} imply $\ord_3(G_2^2(E,\omega,\beta))\geq 1$, giving
$$\ord_3(3\, G_2(E,\omega,\beta)\, \oE_4(E,\omega))\geq \frac{5}{2}>\ord_3(\oE_6(E,\omega)).$$
As before, we conclude $\ord_3(\oE_2^3(E,\omega,\alpha))=\ord_3(\oE_6(E,\omega))$. This proves the reciprocal implication and completes the proof of the proposition.
\end{proof}

\section{Eichler--Shimura analytic relation}
\label{s:canonical analyticity}

In this appendix we further study the canonical branch~$\t$ of~$T_p$ that is defined on~$\Ord$ in Section~\ref{ss:canonical branch}.
We start extending~$\t$, as follows.
Recall that~$\kval$ denotes Katz' valuation, defined in Section~\ref{ss:katz valuation}.
Extend~$\kval$ to~$\Ell(\C_p)$ as~$\kval \equiv 0$ outside~$\Sups$, and put
\begin{equation}
  \label{eq:4}
  N_p
  \=
  \left\{ E \in \Ell(\C_p) : \kval(E) < \frac{p}{p + 1} \right\}.
\end{equation}
On $N_p \cap \Sups$, we use the definition of~$\t$ in Lemma~\ref{l:sups canonical subgroup}.
To define~$\t$ at a point~$E$ in~$\Bad$, let~$z$ in~$\bfD(0, 1)^*$ and let~$\varphi_z \colon \C_p^{\times} / z^{\Z} \to \Tate(z)(\C_p)$ be the isomorphism of analytic groups as in Section~\ref{ss:bad orbits}.
Then we define
\begin{displaymath}
  H(E)
  \=
  \varphi_z( \{ \zeta z^n \in \C_p^{\times} : \zeta^p = 1, n \in \Z \} / z^{\Z} ),
  \text{ and }
  \t(E) \= E / H(E).
\end{displaymath}
Note that in the notation~\eqref{e:1} of Section~\ref{ss:bad orbits}, we have $H(E) = C_{p, z^p}$.
The map~$\t \colon N_p \to \Ell(\C_p)$ so defined is the \emph{canonical branch of~$T_p$}.

The goal of this appendix is to prove the following result.

\begin{theorem}[Eichler--Shimura analytic relation]
  \label{t:canonical analyticity}
  The canonical branch~$\t$ of~$T_p$ is given by a finite sum of Laurent series, each of which converges on all of~$N_p$.
  Furthermore, for every~$E$ in~$N_p \setminus \Bad$ we have 
\begin{equation}
  \label{eq:5}
  \ord_p(\t(j(E)) - j(E)^p)
  \ge
  1 - \kval(E),
\end{equation}
and for every~$E$ in~$\Ell(\C_p)$ we have
  \begin{equation}
    \label{tp1 bis}
    T_p(E)
    =
  \begin{cases}
    \t^*(E) + [\t(E)]
    & \text{if } \kval(E) \le \frac{1}{p + 1};
    \\
    \t^*(E)
    & \text{if } \kval(E) > \frac{1}{p + 1}.
  \end{cases}
\end{equation}
\end{theorem}

In view of~\eqref{eq:5}, the relation~\eqref{tp1 bis} can be seen as refinement and a lift to~$N_p$ of the classical Eichler--Shimura congruence relation, see for example~\cite[Section~7.4]{Shi71} or~\cite[Section~8.7]{DiaShu05}.

The proof of Theorem~\ref{t:canonical analyticity} is at the end of this appendix.
When restricted to~$\Ord$, it is a direct consequence of Theorem~\ref{teo-Deligne} and Proposition~\ref{prop-tpm} with~$m = 1$.
To prove~\eqref{tp1 bis} for~$E$ in~$\Sups$, we use Lemma~\ref{l:sups canonical subgroup}.
To prove this relation on~$\Bad$, we use the results on the uniformization of $p$-adic elliptic curves with multiplicative reduction, recalled in Section~\ref{ss:bad orbits}.
To prove~\eqref{eq:5} and that~$\t$ is a finite sum of Laurent series for~$p \ge 5$, we use Theorem~\ref{teo-Deligne} in Section~\ref{ss:canonical branch}.
For~$p = 2$ and~$3$, we use Proposition~\ref{p:low canonical analyticity} below, whose proof is based on the explicit formulae in~\cite[\emph{Appendice}]{Mes86}.
This result also provides a proof of Theorem~\ref{teo-Deligne} when~$p = 2$ and~$3$

Note that for~$p = 2$ and~$3$, the set~$\tSups$ consists of a single point whose $j$-invariant is equal to~$0$ and to~$1728$, see for example~\cite[Chapter~V, Section~4]{Sil09}.
\begin{proposition}
  \label{p:low canonical analyticity}
  Put~$\fj_2 \= 0$ and~$\fj_3 \= 1728$, and consider the polynomials
  \begin{displaymath}
    \chk_2(z)
    \=
    - 93 \cdot 2^4 z + 627 \cdot 2^8
    \text{ and }
    \chk_3(z)
    \=
    328 \cdot 3^2 z^2 + 85708 \cdot 3^3 z + 1263704 \cdot 3^5.
  \end{displaymath}
  Then for~$p = 2$ and~$3$, the canonical branch~$\t$ of~$T_p$ admits a Laurent series expansion of the form
  \begin{displaymath}
    \t(z)
    =
    (z - \fj_p)^p + \fj_p + \chk_p(z - \fj_p) + \sum_{n = 1}^{\infty} \frac{A_n^{(p)}}{(z - \fj_p)^n},
  \end{displaymath}
  where for every~$n \ge 1$ the coefficient~$A_n^{(p)}$ is in~$\Z$ and satisfies
  \begin{displaymath}
    \ord_p(A_n^{(p)})
    \ge
    \begin{cases}
      4 + 8n
      & \text{if $p = 2$};
      \\
      \frac{3}{2} + \frac{9}{2} n
      & \text{if $p = 3$},
    \end{cases}
  \end{displaymath}
  with equality if~$n = 1$.
\end{proposition}

To prove this proposition, we introduce some notation and recall the explicit formulae in~\cite[\emph{Appendice}]{Mes86}.
For~$\K=\C$ or~$\C_p$, we use~$j$ to identify~$\Ell(\K)$ with~$\K$ and consider~$T_p$ as a correspondence acting on~$\Div(\K)$.
Let~$\Ell_0(p)$, $\alpha_p$ and~$\beta_p$ be as in Section~\ref{ss:Hecke correspondences}, so that~$T_p = (j \circ \alpha_p)_* \circ (j \circ \beta_p)^*$.
Denote by
\begin{displaymath}
  w_p \colon \Ell_0(p)(\K) \to \Ell_0(p)(\K)
\end{displaymath}
the \emph{Atkin--Lehner} or \emph{Fricke involution}, defined by~$w_p(E, C) \= (E/C, E[p]/C)$ and note that~$\beta_p = \alpha_p \circ w_p$.
Identify~$\Ell_0(p)(\C)$ with the quotient~$\Gamma_0(p) \setminus \H$ and denote by~$\eta \colon \H \to \C$ \emph{Dedekind's eta function}, defined by
\begin{displaymath}
  \eta(\tau)
  \=
  \exp \left( \frac{\pi i \tau}{12} \right) \prod_{n = 1}^{\infty} (1 - \exp(2\pi i n \tau)).
\end{displaymath}
Then for~$p = 2$ or~$3$, the function~$\whx_p \colon \H \to \C$ defined by
\begin{displaymath}
  \whx_p(\tau)
  \=
  \left( \frac{\eta(\tau)}{\eta(p \tau)} \right)^{\frac{24}{p - 1}}
\end{displaymath}
descends to a complex analytic isomorphism~$x_p \colon \Ell_0(p)(\C) \to \C$.
Moreover, defining
\begin{displaymath}
  \halpha_p(z)
  \=
  \begin{cases}
    \frac{(z + 2^4)^3}{z}
    & \text{if $p = 2$};
    \\
    \frac{(z + 3^3)(z + 3)^3}{z}
    & \text{if $p = 3$},
  \end{cases}
  \text{ and }
  \whw_p(z)
  \=
  \begin{cases}
    \frac{2^{12}}{z}
    & \text{if $p = 2$};
    \\
    \frac{3^6}{z}
    & \text{if $p = 3$},
  \end{cases}
\end{displaymath}
we have $j \circ \alpha_p = \halpha_p \circ x_p$ and $x_p \circ w_p = \whw_p \circ x_p$, see~\cite[pp. 238, 239]{Mes86}.
It follows that, if we put
\begin{displaymath}
  \hbeta_p(z)
  \=
  \halpha_p \circ \whw_p(z)
  =
  \begin{cases}
    \frac{(z + 2^8)^3}{z^2}
    & \text{if $p = 2$};
    \\
    \frac{(z + 3^3)(z + 3^5)^3}{z^3}
    & \text{if $p = 3$},
  \end{cases}
\end{displaymath}
then~$j \circ \beta_p = \hbeta_p \circ x_p$ and therefore~$T_p = (\halpha_p)_* \circ \hbeta_p^*$ as algebraic correspondences over~$\C$.
Since~$T_p$, $\halpha_p$ and~$\hbeta_p$ are all defined over~$\Q$, we have that the equality~$T_p = (\halpha_p)_* \circ \hbeta_p^*$ also holds as algebraic correspondences over~$\Div(\Ell(\C_p))$.

The following elementary lemma is used the proof of Proposition~\ref{p:low canonical analyticity}.
Given~$r$ in~$(0, 1)$, and a Laurent series~$\sum_{n = 0}^{\infty} \frac{A_n}{z^n}$ in~$\Z \left\llbracket \frac{1}{z} \right\rrbracket$, put
\begin{displaymath}
  \left\| \sum_{n = 0}^{\infty} \frac{A_n}{z^n} \right\|_r
  \=
  \sup \{ |A_n|_p r^{- n} : n \ge 0 \}.
\end{displaymath}

\begin{lemma}
  \label{l:invertibility}
  Let~$\delta(z)$ in~$\frac{1}{z} \Z \left\llbracket \frac{1}{z} \right\rrbracket$ be given and put $f(z) \= z(1 + \delta(z))$.
  Then there is~$\Delta(z)$ in~$\frac{1}{z} \Z \left\llbracket \frac{1}{z} \right\rrbracket$ such that~$F(z) \= z(1 + \Delta(z))$ satisfies~$F(f(z)) = z$.
  If in addition for some~$r$ in~$(0, 1)$ we have~$\| \delta \|_r \le 1$, then~$\| \Delta \|_r \le 1$.
\end{lemma}
\begin{proof}
  We start defining recursively a sequence~$(\Delta_n)_{n = 0}^{\infty}$ in~$\frac{1}{z}  \Z \left[ \frac{1}{z} \right]$ such that for every integer~$n \ge 0$,
  \begin{displaymath}
    z^n \Delta_n(z) \in \Z [z],
    \Delta_{n + 1}(z) \equiv \Delta_n(z) \mod \frac{1}{z^{n + 1}} \Z \left[ \frac{1}{z} \right],
  \end{displaymath}
  and the Laurent polynomial~$F_n(z) \= z(1 + \Delta_n(z))$ satisfies
  \begin{displaymath}
    F_n(f(z)) \equiv z  \mod  \frac{1}{z^n} \Z \left[ \frac{1}{z} \right].
  \end{displaymath}
  For~$n = 0$ put~$\Delta_0(z) = 0$, so~$F_0(f(z)) = f(z) \equiv z \mod  \Z \left[ \frac{1}{z} \right]$.
  Let~$n \ge 0$ be an integer so that~$\Delta_n$ is already defined and let~$A$ in~$\Z$ be the coefficient of~$\frac{1}{z^n}$ in~$F_n(f(z))$.
  Then for~$\Delta_{n + 1}(z) \= \Delta_n(z) - \frac{A}{z^{n + 1}}$, we have
  \begin{multline*}
    (F_{n + 1} - F_n)(f(z))
    =
    - \frac{A}{z^n(1 + \delta(z))^n}
    =
    - \frac{A}{z^n} \left(1 + \sum_{k = 1}^{\infty} (- \delta(z))^k \right)^n
    \\ \equiv
    - \frac{A}{z^n} \mod \frac{1}{z^{n + 1}} \Z \left[ \frac{1}{z} \right],
  \end{multline*}
 and therefore
  \begin{displaymath}
    F_{n + 1}(f(z)) - z
    =
    F_n(f(z)) - z + (F_{n + 1} - F_n)(f(z))
    \equiv
    0 \mod \frac{1}{z^{n + 1}}  \Z \left[ \frac{1}{z} \right].
  \end{displaymath}
  This completes the definition of the sequence~$(\Delta_n)_{n = 0}^{\infty}$.
  It follows that the unique series~$\Delta$ in~$\frac{1}{z} \Z \left\llbracket \frac{1}{z} \right\rrbracket$ satisfying for every~$n \ge 0$ the congruence
  \begin{displaymath}
    \Delta(z)
    \equiv
    \Delta_n(z) \mod \frac{1}{z^{n + 1}}  \Z \left\llbracket \frac{1}{z} \right\rrbracket,
  \end{displaymath}
  satisfies~$F(f(z)) = z$.

  To prove the last assertion, note that for every~$r$ in~$(0,1)$,
  \begin{displaymath}
    I_r
    \=
    \left\{ z(1 + g(z)) : g(z) \in \tfrac{1}{z}\Z \left\llbracket \tfrac{1}{z} \right\rrbracket, \| g \|_r\leq 1 \right\}
  \end{displaymath}
  is a collection of series in~$\Z \left\llbracket \frac{1}{z} \right\rrbracket$ that is closed under composition.
  It follows from the above construction that, if for some~$r$ in~$(0, 1)$ we have~$\| \delta \|_r \le 1$, then for every integer~$n \ge 0$ the series~$F_n$ and~$F_n \circ f$ are both in~$I_r$.
  This implies that~$F$ is in~$I_r$, as wanted. 
\end{proof}

The proof of Proposition~\ref{p:low canonical analyticity} is given after the following lemma, which is also used in the proof of Theorem~\ref{t:canonical analyticity}.
\begin{lemma}
  \label{l:bad analyticity}
  For an arbitrary prime number~$p$, the right-hand side of~\eqref{eq:6} converges to~$\t$ on~$\Ord \cup \Bad$.
\end{lemma}
\begin{proof}
  Let~$\Phi_p(X, Y)$ be the modular polynomial of level~$p$, as defined in Section~\ref{ss:Hecke correspondences}, so that for every~$z$ in~$\Ord$ we have~$\Phi_p(z, \t(z)) = 0$.
  By Theorem~\ref{teo-Deligne}, the finite sum of Laurent series on the right-hand side of~\eqref{eq:6} converges on~$\Ord \cup \Bad$ to a function~$\widehat{\t}$ extending~$\t$, and for~$z$ in~$\Bad$ we have~$|\widehat{\t}(z)|_p = |z|_p^p$.
  It follows that for every~$z$ in~$\Bad$ we have~$\Phi_p(z, \widehat{\t}(z)) = 0$, so~$\widehat{\t}(z)$ is in the support of~$T_p(z)$.
  Combining~\eqref{jtate} and~\eqref{tateq}, we conclude that~$\widehat{\t}(z) = \t(z)$.
\end{proof}

\begin{proof}[Proof of Proposition~\ref{p:low canonical analyticity}]
Note that if we put~$r_2 \= 2^{-8}$ and~$r_3 \= 3^{- \frac{9}{2}}$, then for~$p = 2$ and~$3$ we have by Proposition~\ref{vnorm},
\begin{displaymath}
  N_p
  \=
  \{ z \in \C_p : |z - \fj_p|_p > r_p \}.
\end{displaymath}
  For~$p = 2$ and~$3$, put
  \begin{displaymath}
    \chalpha_p \= \halpha_p - \fj_p
    \text{ and }
    \chbeta_p \= \hbeta_p - \fj_p.
  \end{displaymath}
  Note that for~$p = 3$, we have
  \begin{displaymath}
    \chalpha_3(z) = \frac{(z^2 + 2 \cdot 3^2 z - 3^3)^2}{z}
    \text{ and }
    \chbeta_3(z) = \frac{(z^2 - 2 \cdot 3^5 z - 3^9)^2}{z^3}.
  \end{displaymath}
  So, for~$p = 2$ and~$3$ the rational map~$\delta_p(z) \= z^{-1}\chbeta_p(z) - 1$ is a Laurent polynomial in~$\frac{1}{z}\Z \left[ \frac{1}{z} \right]$ satisfying~$\| \delta_p \|_{r_p} \le 1$.
  In particular, for every~$z$ in the set
  \begin{displaymath}
    \chN_p
    \=
    \{ z' \in \C_p : |z'|_p > r_p \},
  \end{displaymath}
  we have~$|\chbeta_p(z)|_p = |z|_p$, so~$\chbeta_p$ maps~$\chN_p$ into itself.
  By Lemma~\ref{l:invertibility} there is~$\Delta_p(w)$ in~$\frac{1}{w} \Z \left\llbracket \frac{1}{w} \right\rrbracket$ such that~$\| \Delta_p \|_{r_p} \le 1$ and such that the map
  \begin{displaymath}
    \begin{array}{rccl}
      F_p \colon & \chN_p & \to & \chN_p
                               \\ & w & \mapsto & F_p(w) \= w(1 + \Delta_p(w))
    \end{array}
  \end{displaymath}
  is an inverse of~$\chbeta_p|_{\chN_p}$.

  We show below that~$\t$ coincides with the map
  \begin{displaymath}
    \begin{array}{rccl}
      \cht \colon & N_p & \to & \C_p
                                \\ & z & \mapsto & \cht(z) \= (\chalpha_p \circ F_p)(z - \fj_p) + \fj_p.
    \end{array}
  \end{displaymath}
  Once this is established, the proposition follows from explicit computations using the estimates,
  \begin{displaymath}
    \| \Delta_p \|_{r_p} \le 1,
    \left\| \frac{\chalpha_2(w)}{w^2} \right\|_{2^{-4}} \le 1 \text{ for~$p = 2$, and }
    \left\| \frac{\chalpha_3(w)}{w^3} \right\|_{3^{-\frac{3}{2}}} \le 1 \text{ for~$p = 3$}.
  \end{displaymath}
  By definition, for each~$z$ in~$\hN_p$ the point~$\cht(z)$ is in the support of~$T_p(z) = {(\alpha_p)_* \circ \beta_p^*(z)}$.
  Moreover, for every~$z$ in~$\Bad$ we have~$|\cht(z)|_p = |z|_p^p$, so by~\eqref{jtate} and~\eqref{tateq} we have~$\cht(z) = \t(z)$.
  Combined with Lemma~\ref{l:bad analyticity}, this implies that~$\cht$ and~$\t$ agree on~$\Ord \cup \Bad$.
  In view of Proposition~\ref{vnorm} and Lemma~\ref{l:sups canonical subgroup}, to prove that~$\cht$ and~$\t$ agree on~$N_p \cap \Sups$ it is sufficient to show that for every~$w$ in~$\chN_p \cap \cM_p$ we have~$|(\chalpha_p \circ F_p)(w)|_p \neq |w|_p^{\frac{1}{p}}$.
  Note that for every~$w$ in~$\chN_p$ we have~$|F_p(w)|_p = |w|_p$.
  A direct computation shows that for~$p = 2$ we have
  \begin{displaymath}
    |(\chalpha_2 \circ F_2)(w)|_2
    \begin{cases}
      = |w|_2^2
      & \text{if $2^{-4} < |w|_2 < 1$};
      \\
      \le 2^{-8}
      & \text{if $|w|_2 = 2^{-4}$};
      \\
      = \frac{2^{-12}}{|w|_2}
      & \text{if $r_2 < |w|_2 < 2^{-4}$},
    \end{cases}
  \end{displaymath}
  and that for~$p = 3$ we have
  \begin{displaymath}
    |(\chalpha_3 \circ F_3)(w)|_3
    \begin{cases}
      = |w|_3^3
      & \text{if $3^{-\frac{3}{2}} < |w|_3 < 1$};
      \\
      \le 3^{-\frac{9}{2}}
      & \text{if $|w|_3 = 3^{-\frac{3}{2}}$};
      \\
      = \frac{3^{-6}}{|w|_3}
      & \text{if $r_3 < |w|_3 < 3^{-\frac{3}{2}}$}.
    \end{cases}
  \end{displaymath}
  In all the cases we have~$|(\chalpha_p \circ F_p)(w)|_p \neq |w|_p^{\frac{1}{p}}$.
  This completes the proof of~$\t = \cht$, and of the proposition.
\end{proof}

\begin{proof}[Proof of Theorem~\ref{t:canonical analyticity}]
  We first prove~\eqref{eq:5}, and the assertions about the Laurent series expansion.
  For~$p = 2$ and~$3$, these are given by Proposition~\ref{p:low canonical analyticity}.
  Assume~$p \ge 5$.
  For each~$\ss$ in~$\tSups$, let~$\fj_{\ss}$ be given by Proposition~\ref{vnorm}, and define~$P_{\sups}(X) = \prod_{\ss \in \tSups} (X - \fj_{\ss})$ as in the proof of this proposition.
  Since the reduction modulo~$p$ of the polynomial~$P_{\sups}$ is separable, for every~$\ss$ in~$\tSups$ we have that~$\fj_{\ss}$ is in~$\Q_p^{\unr}$.
  Put~$\beta_{\ss} \= \fj_{\ss}$.
  Denote by~$\widehat{\t}$ the finite sum of Laurent series in the right-hand side of~\eqref{eq:6} for these choices of~$(\beta_{\ss})_{\ss \in \tSups}$.
  It follows from Theorem~\ref{teo-Deligne} and Proposition~\ref{vnorm} that~$\widehat{\t}$ converges on~$N_p$, and by Lemma~\ref{l:bad analyticity} that for every~$z$ in~$\Bad \cup \Ord$ we have~$\widehat{\t}(z) = \t(z)$.
  We proceed to prove that for every~$z$ in~$\hN_p \= N_p \cap \Sups$ we also have~$\widehat{\t}(z) = \t(z)$.
  
  Denote by~$\Phi_p(X, Y)$ the modular polynomial of level~$p$ defined in Section~\ref{ss:Hecke correspondences}.
  Note that for every~$z$ in~$\Bad \cup \Ord$ we have
  \begin{equation}
    \label{eq:11}
    \Phi_p(\widehat{\t}(z), z)
    =
    \Phi_p(z, \widehat{\t}(z))
    =
    0.
  \end{equation}
  Since~$\widehat{\t}$ is analytic, \eqref{eq:11} holds for every~$z$ in~$N_p$.
  In view of Lemma~\ref{l:sups canonical subgroup}, this implies that for every~$E$ in~$\hN_p$ we have either~$\kval(\widehat{\t}(E)) = \frac{1}{p} \kval(E)$, or
  \begin{equation}
    \label{eq:9}
    \kval(\widehat{\t}(E))
    \begin{cases}
      = p \kval(E)
      & \text{if } \kval(E) \in \left] 0, \frac{1}{p + 1} \right];
      \\
      \ge p \kval(E)
      & \text{if } \kval(E) = \frac{1}{p + 1};
      \\
      = 1 - \kval(E)
      & \text{if } \kval(E) \in \left] \frac{1}{p + 1}, \frac{p}{p + 1} \right[.
    \end{cases}
  \end{equation}
  We now prove that~\eqref{eq:9} holds for every~$E$ in~$\hN_p$.
  Fix~$\ss$ in~$\tSups$, and note that the function
  \begin{displaymath}
    \begin{array}{rccl}
    \nu \colon & \left] 0, \frac{p}{p + 1} \right[ \cap \Q & \to & \Q
                                                                   \\ & r & \mapsto & \nu(r) \= \inf \{ \kval(\widehat{\t}(E)) : E \in \bfD(j(\ss)), \kval(E) = r \},
  \end{array}
\end{displaymath}
extends continuously to~$\left] 0, \frac{p}{p + 1} \right[$.
Thus, either~\eqref{eq:9} holds for every~$E$ in~$N_p \cap \bfD(j(\ss))$, or for every~$E$ in this set we have~$\kval(\widehat{\t}(E)) = \frac{1}{p} \kval(E)$.
So, to prove that~\eqref{eq:9} holds for every~$E$ in~$N_p \cap \bfD(j(\ss))$ it is sufficient to prove that it holds for some~$E_0$ in~$N_p \cap \bfD(j(\ss))$.
  Choose~$E_0$ in~$N_p \cap \bfD(j(\ss))$ such that~$z_0 \= j(E_0)$ satisfies
  \begin{displaymath}
    0 < \ord_p(z_0 - \fj_{\ss}) < \frac{1}{p + 1}.
  \end{displaymath}
  By Theorem~\ref{teo-Deligne} we have
  \begin{displaymath}
    \ord_p \left(\widehat{\t}(z_0) - z_0^p - p k(z_0) \right)
    \ge
    1 - \ord_p(z_0 - \fj_{\ss})
    >
    \frac{p}{p + 1}.
  \end{displaymath}
  Since~$\ord_p(z_0 - \fj_{\ss}) < \frac{1}{p}$, we also have
  \begin{displaymath}
    \ord_p(\widehat{\t}(z_0) - \fj_{\ss}^p)
    =
    p \ord_p(z_0 - \fj_{\ss})
    <
    \frac{p}{p + 1}.
  \end{displaymath}
  Combined with~$\ord_p(\fj_{\ss}^p - \fj_{\ss^{(p)}}) \ge 1$ and $\ord_p(p k(z_0)) \ge 1$, this implies
  \begin{equation}
    \label{eq:10}
    \ord_p \left( \widehat{\t}(z_0) - \fj_{\ss^{(p)}} \right)
    =
    p \ord_p(z_0 - \fj_{\ss}),
  \end{equation}
  and therefore~\eqref{eq:9} with~$E = E_0$.
  This completes the proof that~\eqref{eq:9} holds for every~$E$ in~$\hN_p$.
  In view of \eqref{eq:11}, Proposition~\ref{vnorm}, and Lemma~\ref{l:sups canonical subgroup}, it follows that for every~$z$ in~$\hN_p$ we have~$\widehat{\t}(z) = \t(z)$.
  By Theorem~\ref{teo-Deligne} we also obtain~\eqref{eq:5}.
  
  It remains to prove~\eqref{tp1 bis} for an arbitrary prime number~$p$.
  Note that for~$E$ in~$\Ord$ this is given by Proposition~\ref{prop-tpm} with~$m = 1$, and that for~$E$ in~$\Bad$ this follows from the combination of~\eqref{jtate}, and of~\eqref{tateq} with~$n = p$.
  It remains to prove~\eqref{tp1 bis} for~$E$ in~$\hN_p$.
  By the considerations above, and the proof of Proposition~\ref{p:low canonical analyticity}, we have that~\eqref{eq:9} holds for every prime number~$p$ and for every~$E$ in~$\hN_p$.
By Lemma~\ref{l:sups canonical subgroup} we deduce that:
\begin{enumerate}
\item[1.]
  $\t$ maps
\begin{displaymath}
  N_p'
  \=
  \left\{ E \in \Ell(\C_p) : 0 < \kval(E) < \frac{1}{p + 1} \right\}
\end{displaymath}
onto~$\hN_p$, and for every~$E$ in~$\hN_p$ the divisor~$(\t|_{N_p'})^*(E)$ has degree~$p$;
\item[2.]
  $\t$ maps
\begin{displaymath}
  S_p
\=
\left\{ E \in \Ell(\C_p) : \kval(E) = \frac{1}{p + 1} \right\}
\end{displaymath}
onto~$B_p \= \Sups \setminus \hN_p$, and for every~$E$ in~$B_p$ the divisor~$(\t|_{S_p})^*(E)$ has degree~$p + 1$;
\item[3.]
  $\t$ maps~$A_p \= \hN_p \setminus (N_p' \cup S_p)$ onto itself, and for every~$E$ in~$A_p$ we have~$(\t|_{A_p})^*(E) = [\t(E)]$.
\end{enumerate}
The proof of~\eqref{tp1 bis} is divided in the following cases:
\begin{enumerate}
\item[1.]
  For~$E$ in~$B_p$, we have~$\t^*(E) = (\t|_{S_p})^*(E)$ and this divisor has degree~$p + 1$.
  Together with~\eqref{eq:11} this implies~$T_p(E) = \t^*(E)$;
\item[2.]
  For~$E$ in~$A_p$, we have~$\t^*(E) = (\t|_{N_p'})^*(E) + (\t|_{A_p})^*(E)$ and this divisor has degree~$p + 1$.
  As in the previous case we conclude that~$T_p(E) = \t^*(E)$;
\item[3.]
  For~$E$ in~$N_p' \cup S_p$, we have~$\t^*(E) = (\t|_{N_p'})^*(E)$ and this divisor is of degree~$p$.
  Combined with~\eqref{eq:11} this implies that the divisor~$T_p(E) - \t^*(E)$ has degree~$1$.
  On the other hand, by~\eqref{eq:9} the point~$\t(E)$ is not in the support of~$\t^*(E)$, so by~\eqref{eq:11} we have~$T(E) - \t^*(E) = [\t(E)]$.
\end{enumerate}
This completes the proof of~\eqref{tp1 bis}, and of the theorem.
\end{proof}

\bibliographystyle{alpha}

\end{document}